\documentclass[12pt]{article}

\usepackage[T1]{fontenc}
\usepackage[utf8]{inputenc}
\usepackage{amsmath,amsfonts,amsthm,amssymb,amsxtra,setspace,xspace,graphicx,lmodern,mathrsfs}
\usepackage[colorlinks=true]{hyperref}
\hypersetup{urlcolor=blue, citecolor=red, linkcolor=blue}
\usepackage[usenames,dvipsnames]{color}
\usepackage{tcolorbox}
\usepackage{bbm}
\usepackage{tabularx}

\makeatletter\newcommand{\Capitalize}[1]{\@ifpackagewith{inputenc}{utf8x}{\PrerenderUnicode{#1}}{}\edef\@tempa{\expandafter\@gobble\string#1}\edef\@tempb{\expandafter\@car\@tempa\@nil}\edef\@tempa{\expandafter\@cdr\@tempa\@nil}\uppercase\expandafter{\expandafter\def\expandafter\@tempb\expandafter{\@tempb}}\@namedef{\@tempb\@tempa}{\expandafter\MakeUppercase\expandafter{#1}}}\makeatother

\makeatother

\renewcommand{\thesection}{\arabic{section}}
\renewcommand{\thesubsection}{\arabic{section}.\arabic{subsection}}

\usepackage{titlesec}
\titleformat*{\section}{\large\bfseries}
\titleformat*{\subsection}{\bfseries}
\titleformat*{\subsubsection}{\itshape}

\newtheorem{theorem}{Theorem}
\newtheorem{lemma}[theorem]{Lemma}

\newtheorem{proposition}[theorem]{Proposition}

\newcommand{\Email}[1]{\newline\emph{E-mail:} \href{mailto:#1}{\textsf{#1}}}
\newcommand{\be}[1]{\begin{equation}\label{#1}}
\newcommand{\ee}{\end{equation}}
\renewcommand{\(}{\left(}
\renewcommand{\)}{\right)}
\newcounter{step}
\newenvironment{steps}{\setcounter{step}{0}}{}
\newcommand{\stepitem}{\addtocounter{step}1\par\smallskip\noindent\emph{Step \thestep.} }

\newcommand{\R}{{\mathbb R}}
\newcommand{\N}{{\mathbb N}}

\newcommand{\ird}[1]{\int_{\R^d}{#1}\,dx}

\newcommand{\nrm}[2]{\left\|{#1}\right\|_{#2}}
\newcommand{\dx}{\,{\rm d}x}

\newcommand{\RR}{\mathbb R}
\newcommand{\mB}{\mathcal B}
\newcommand{\Mstar}{\mathcal M}
\newcommand{\pc}{p_m}

\newcommand{\lambdaBarenblatt}{\mathtt{b}}

\newcommand{\dt}{\,{\rm d}t}
\newcommand{\rd}{{\rm d}}

\newcommand{\ka}{\overline{\kappa}}

\newcommand{\kb}{\underline{\kappa}}

\newcommand{\NN}{\mathbb{N}}

\newcommand{\cc}{\mathsf{c}_3}
\newcommand{\taustar}{{\mathsf c_\star}}
\newcommand{\pcc}{{\mathsf p}}

\def\supp{\mathrm{supp}}

\newcommand{\boxedresult}[1]{\noindent\fbox{\parbox{14.47cm}{#1}}\\[4pt]}
\usepackage{mathtools}
\newcommand{\eqrefstd}[1]{\renewtagform{starred}{(}{)}\usetagform{starred}\eqref{#1}\renewtagform{starred}{\cite[Eq.~(}{)]{BDNS2020}}\usetagform{starred}}

\begin{document}
%%%%%%%%%%%%%%%%%%%%%%%%%%%%%%%%%%%%%%%%%%%%%%%%%%%%%%%%%%%%%%%%%%%%%%
%%%%%%%%%%%%%%%%%%%%%%%%%%%%%%%%%%%%%%%%%%%%%%%%%%%%%%%%%%%%%%%%%%%%%%
\parindent=0pt\parskip=6pt\begin{spacing}{1.5}
{\Large\MakeUppercase{Stability in Gagliardo-Nirenberg\\ inequalities \\ Supplementary material}}
\vspace*{5pt}\hrule\end{spacing}\begin{spacing}{1.2}
{\large\MakeUppercase{M.~Bonforte, J.~Dolbeault, B.~Nazaret, N.~Simonov}}
\vspace*{12pt}
\end{spacing}

\emph{M.~Bonforte:} Departamento de Matem\'{a}ticas, Universidad Aut\'{o}noma de Madrid, and ICMAT, Campus de Cantoblanco, 28049 Madrid, Spain. \Email{matteo.bonforte@uam.es}
\\[6pt]
\emph{J.~Dolbeault:} Ceremade, UMR CNRS n$^\circ$~7534, Universit\'e Paris-Dauphine, PSL Research University, Place de Lattre de Tassigny, 75775 Paris Cedex~16, France. \Email{dolbeaul@ceremade.dauphine.fr}
\\[6pt]
\emph{B.~Nazaret:} SAMM (EA 4543), FP2M (FR CNRS 2036), Universit\'e Paris 1, 90, rue de Tolbiac, 75634 Paris Cedex~13, and Mokaplan, Inria Paris, France. \Email{bruno.nazaret@univ-paris1.fr}
\\[6pt]
\emph{N.~Simonov:} Ceremade, UMR CNRS n$^\circ$~7534, Universit\'e Paris-Dauphine, PSL Research University, Place de Lattre de Tassigny, 75775 Paris Cedex~16, France.
\Email{simonov@ceremade.dauphine.fr}

\medskip{\large\bf Abstract.} \emph{This document comes as supplementary material of the paper \emph{Stability in Gagliardo-Nirenberg inequalities~\cite{BDNS2020}} by the same authors. It is intended to state a number of classical or elementary statements concerning constants and inequalities for which we are not aware of existing published material or expressions detailed enough for our purpose. We claim no originality on the theoretical results and rely on standard methods in most cases, except that we keep track of the constants and provide constructive estimates.}

\bigskip\noindent This document is divided into two Parts. Part~\ref{PartI} is devoted to the computation of the constant in Moser's Harnack inequality based on~\cite{Moser1964,Moser1971} and has its own interest. Part~\ref{PartII} is thought as a series of fully explicit and constructive estimates for the reader of~\cite{BDNS2020} interested in the details of the computations. In order to make formulas tractable, whenever possible, simplicity has been privileged over sharpness and the computations leave plenty of space for improvements. For a comprehensive introduction to the topic, some motivations and a review of the literature, the reader is invited to refer to~\cite[Section~1]{BDNS2020}. Boxed inequalities are used to recall the results of~\cite{BDNS2020} which involve the constants needed for our computations.\vspace*{-0.5cm}

%%%%%%%%%%%%%%%%%%%%%%%%%%%%%%%%%%%%%%%%%%%%%%%%%%%%%%%%%%%%%%%%%%%%%%
\thispagestyle{empty}
\footnote{\hspace*{-13pt}\begin{minipage}{14.5cm}


\end{minipage}}
\thispagestyle{empty}
\parindent=12pt\parskip=6pt

%%%%%%%%%%%%%%%%%%%%%%%%%%%%%%%%%%%%%%%%%%%%%%%%%%%%%%%%%%%%%%%%%%%%%%
%%%%%%%%%%%%%%%%%%%%%%%%%%%%%%%%%%%%%%%%%%%%%%%%%%%%%%%%%%%%%%%%%%%%%%
%%%%%%%%%%%%%%%%%%%%%%%%%%%%%%%%%%%%%%%%%%%%%%%%%%%%%%%%%%%%%%%%%%%%%%
%%%%%%%%%%%%%%%%%%%%%%%%%%%%%%%%%%%%%%%%%%%%%%%%%%%%%%%%%%%%%%%%%%%%%%
\newpage\part{\Large The constant in Moser's Harnack inequality}\label{PartI}

Let $\Omega$ be an open domain and let us consider a positive \emph{weak solution} to
\be{HE.coeff}
\frac{\partial v}{\partial t}=\nabla\cdot\big(A(t,x)\,\nabla v\big)
\ee
on $\Omega_T:=\left(0, T\right)\times \Omega$, where $A(t,x)$ is a real symmetric matrix with bounded measurable coefficients satisfying the uniform ellipticity condition
\be{HE.coeff.lambdas}
0<\lambda_0\,|\xi|^2\le\sum_{i,j=1}^dA_{i,j}(t,x)\,\xi_i\,\xi_j\le\lambda_1\,|\xi|^2\quad\forall\,(t,x,\xi)\in\R^+\times\Omega_T\times\R^d\,,
\ee
for some positive constants $\lambda_0$ and $\lambda_1$. Let us consider the neighborhoods
\be{cylinder.harnack}
\begin{split}
& D_R^+(t_0,x_0):=(t_0+\tfrac34\,R^2,t_0+R^2)\times B_{R/2}(x_0)\,,\\
& D_R^-(t_0,x_0):=\left(t_0-\tfrac34\,R^2,t_0-\tfrac14\,R^2\right)\times B_{R/2}(x_0)\,,
\end{split}\ee
and the constant
\be{h}
\mathsf h:=\exp\left[2^{d+4}\,3^d\,d+c_0^3\,2^{2\,(d+2)+3}\left(1+\frac{2^{d+2}}{(\sqrt2-1)^{2\,(d+2)}}\right)\sigma\right]
\ee
where
\be{c_0}
c_0=3^\frac2{d}\,2^\frac{(d+2)\,(3\,d^2+18\,d+24)+13}{2\,d}\(\tfrac{(2+d)^{1+\frac4{d^2}}}{d^{1+\frac2{d^2}}}\)^{(d+1)(d+2)}\,\mathcal K^\frac{2\,d+4}{d}\,,
\ee
\be{sigma}
\sigma=\sum_{j=0}^{\infty}\left(\tfrac34\right)^j\,\big((2+j)\,(1+j)\big)^{2\,d+4}\,.
\ee
Let $\pcc:=2\,d/(d-2)=2^\star$ if $d\ge3$, $\pcc:=4$ if $d=2$ and $\pcc\in(4,+\infty)$ if $d=1$. The constant $\mathcal K$ in~\eqref{c_0} is the constant in the inequality
\be{sob.step2}
\|f\|_{\mathrm L^{\pcc}(B_R)}^2\le\mathcal K\(\|\nabla f\|_{\mathrm L^2(B_R)}^2+\tfrac1{R^2}\,\|f\|_{\mathrm L^2(B_R)}^2\)\quad\forall\,f\in\mathrm H^1(B_R)\,.
\ee
If $d\ge 3$, then $\mathcal K$ is independent of $R$. For $d=1$, $2$, we further assume that $R\le 1$. We learn from~\cite[Appendices B and C]{BDNS2020} that
\be{estim.S-p}
\mathcal K \le
\begin{cases}\begin{array}{ll}
2\,\mathcal S_1^2=\tfrac2{\pi}\,\Gamma\left(\frac d2+1\right)^\frac2d&\quad\mbox{if}\quad d\ge 3\,,\\
\frac2{\sqrt{\pi}}&\quad\mbox{if}\quad d=2\,,\\
2^{1+\frac2{\pcc}}\max\left\{\tfrac{\pcc-2}{\pi^2},\tfrac14\right\}&\quad\mbox{if}\quad d=1\,.
\end{array}
\end{cases}
\ee
Also see~Table~\ref{table.k.bar} below. We shall also need some numerical constants associated with balls and spheres. The volume of the unit sphere $\mathbb S^{d-1}\subset\R^d$ is
\be{omega-d}
\omega_d=|\mathbb S^{d-1}|=\frac{2\,\pi^{d/2}}{\Gamma(d/2)}\,\le\,\frac{16}{15}\,\pi^3\,.
\ee
As a consequence, the volume of a $d$-dimensional unit ball is $\omega_d/d$ and
\be{omega-over-d}
\frac{\omega_d}{d}\le\pi^2
\ee
whenever sharpness is not needed. Let
\be{h-bar}
\overline{\mathsf h}:=\mathsf h^{\lambda_1+1/\lambda_0}\,.
\ee
%---------------------------------------------------------------------
\begin{theorem}\label{Claim:3} Let $T>0$, $R\in(0,\sqrt T)$, and take $(t_0,x_0)\in(0,T)\times\Omega$ such that $\left(t_0-R^2, t_0+R^2\right)\times B_{2\,R}(x_0)\subset\Omega_T$. Under Assumption~\eqref{HE.coeff.lambdas}, if $u$ is a weak solution of~\eqref{HE.coeff}, then
\be{harnack}
\sup_{D^{-}_R(t_0,x_0)} v\le\overline{\mathsf h}\,\inf_{D^{+}_R(t_0,x_0)} v\,.
\ee
\end{theorem}
%---------------------------------------------------------------------
Here a \emph{weak solution} is defined as in~\cite[p.~728]{Moser1971},~\cite[Chapter~3]{Ladyzenskaja1995} or~\cite{Aronson1967b}. The \emph{Harnack inequality} of Theorem~\ref{Claim:3} goes back to J.~Moser~\cite{Moser1964,Moser1971}. The dependence of the constant on the ellipticity constants $\lambda_0$ and $\lambda_1$ was not clear before the paper of J.~Moser~\cite{Moser1971}, where he shows that such a dependence is optimal by providing an explicit example,~\cite[p.~729]{Moser1971}. The fact that $\mathsf h$ only depends on the dimension $d$ is also pointed out by C.E.~Gutierrez and R.L.~Wheeden in~\cite{Gutierrez1990} after the statement of their Harnack inequalities,~\cite[Theorem~A]{Gutierrez1990}. However, to our knowledge, a complete constructive proof was still missing. We do not claim any originality concerning the strategy but provide for the first time an explicit expression for the constant~$\overline{\mathsf h}$.

The proof of the above theorem is quite long and technical, and relies on three main ingredients, contained in the first three sections:
\begin{itemize}
\item \textit{Moser iteration procedure.} In Section~\ref{Sec:MoserIteration}, the main local upper and lower smoothing effects are obtained, through the celebrated Moser iteration, in the form of precise $\mathrm L^p-\mathrm L^\infty$ and $\mathrm L^{\kern-2pt-p}-\mathrm L^{\kern-2pt-\infty}$ bounds for arbitrarily small $p>0$. The next task would be to relate such upper and lower bounds, to produce the desired Harnack inequalities. This can be done by means of parabolic BMO estimates, but in this case one may lose control of the estimates, since not all the proofs of such BMO bounds are constructive. We hence follow the ideas of J.~Moser in~\cite{Moser1971}, which avoids the use of BMO spaces, as follows.
\item \textit{Logarithmic Estimates.} The idea is to obtain detailed informations on the level sets of solutions. This can be done by estimating the logarithm of the solution to~\eqref{HE.coeff}. This is a fundamental estimate needed both in the approach with BMO spaces (it indeed implies that $u$ has bounded mean oscillation) and in the alternative approach used here. 
\item \textit{A lemma by E.~Bombieri and E.~Giusti.} In Section~\ref{Sec:Bombieri-Giusti}, we prove a parabolic version of the Bombieri-Giusti Lemma, following again Moser's proof in~\cite{Moser1971} (also see~\cite{Bombieri_1972}). This refinement of the upper bounds may seem trivial at first sight, but it is not and turns out to be crucial for our constructive method. 
\item \textit{Proof of Moser's Harnack inequality.} We finally prove Theorem~\ref{Claim:3} using a suitably rescaled solution.
\item \textit{Harnack inequality implies H\"older Continuity.} As an important consequence of Theorem~\ref{Claim:3}, we obtain explsicit and quantitative H\"older continuity estimates in Section~\ref{sec:holder}, by following Moser's approach in~\cite{Moser1964}. We find an explicit expression of the H\"older exponent, which only depends on the dimension and on the ellipticity constants.
\end{itemize}
 
%%%%%%%%%%%%%%%%%%%%%%%%%%%%%%%%%%%%%%%%%%%%%%%%%%%%%%%%%%%%%%%%%%%%%%%%%%
%%%%%%%%%%%%%%%%%%%%%%%%%%%%%%%%%%%%%%%%%%%%%%%%%%%%%%%%%%%%%%%%%%%%%%%%%%
\section{Upper and lower Moser iteration}\label{Sec:MoserIteration}

Let us start by recalling the definition of the parabolic cylinders
\[\label{Parab.Cylinders}\begin{split}
&Q_\varrho=Q_\varrho(0,0)=\left\{ |t|<\varrho^2\,,\;|x|<\varrho\right\}=(-\varrho^2,\varrho^2)\times B_\varrho(0)\,,\\
&Q^+_\varrho=Q_\varrho(0,0)=\left\{ 0<t<\varrho^2\,,\;|x|<\varrho\right\}=(0,\varrho^2)\times B_\varrho(0)\,,\\
&Q^-_\varrho=Q_\varrho(0,0)=\left\{ 0<-t<\varrho^2\,,\;|x|<\varrho\right\}=(-\varrho^2,0)\times B_\varrho(0)\,.
\end{split}
\]
In order to perform the celebrated Moser iteration, we establish an important lemma, which relies on~\eqref{sob.step2}. We follow the method of~\cite{Moser1964,Moser1971} and provide a quantitative and constructive proof, with explicit constants. From here on, we assume that $u$ is a positive solution, as was done by J.~Moser in~\cite[p.~729, l.~8-9]{Moser1971}.
%---------------------------------------------------------------------
\begin{lemma}[Moser iteration,~\cite{Moser1964,Moser1971}]\label{Lem.Moser}
Assume that $r$ and $\rho$ are such that $1/2\le \varrho\le r\le1$ and $\mu=\lambda_1+1/\lambda_0$ and let $v$ be a nonnegative solution to~\eqref{HE.coeff}. Then there exists a positive constant $c_1=c_1(d)$ such that
\be{Lem.Moser.Upper}
\sup_{Q_\varrho} v^p \le\frac{c_1}{(r-\varrho)^{d+2}}\iint_{Q_r}v^p \dx\dt\quad\forall\,p\in\(0,\tfrac1\mu\)
\ee
and
\be{Lem.Moser.Lower}
\sup_{Q^-_\varrho} v^p \le\frac{c_1}{(r-\varrho)^{d+2}}\iint_{Q_r^-}v^p \dx\dt\quad\forall\,p\in\(-\tfrac1\mu,0\).
\ee
\end{lemma}
%---------------------------------------------------------------------
Let us observe that the second estimate is a lower bound on $v$ because $p$ is negative. Our contribution is to establish that the constant $c_1=c_1(d)$ is given by  
\be{Lem.Moser.constant}
c_1=3^{\gamma-1}\(2^{2\gamma^2+7(\gamma-1)}\,\gamma^{(\gamma+1)(2\gamma-1)}\,d^{(\gamma+1)(\gamma-1)}\,\mathcal K^{\gamma-1}\)^\frac\gamma{(\gamma-1)^2}\,,
\ee
where $\gamma=(d+2)/d$ if $d\ge3$, $\gamma=5/3$ if $d=1$ or $2$, and $\mathcal K$ is the constant of~\eqref{estim.S-p}.

\begin{proof}[Proof of Lemma~\ref{Lem.Moser}] We first notice that it is sufficient to prove the lemma for $\varrho=1/2$ and $r=1$. By \emph{admissible transformations}, as they are called in Moser's papers~\cite{Moser1964,Moser1971}, we can change variables according to 
\be{admissible.transformations}
t\rightsquigarrow \alpha^2\,t+t_0\quad\mbox{and}\quad x\rightsquigarrow \alpha\,x+x_0
\ee
without changing the class of equations: $\lambda_0$ and $\lambda_1$ are invariant under~\eqref{admissible.transformations}. Therefore it is sufficient to prove
\[
\sup_{Q_{\theta/2}}v^p\le\frac{c_1}{\theta^{d+2}}\iint_{Q_{\theta}}v^p\dx\dt\quad\forall\,\theta>0\,.
\]
We recover~\eqref{Lem.Moser.Upper} by setting $\theta=r-\varrho$ and applying the above inequality to all cylinders in $Q_r$ obtained by translation from $Q_\theta$ with admissible transformations. The centers of the corresponding cylinders certainly cover $Q_\varrho$ and~\eqref{Lem.Moser.Upper} follows. Analogously, one reduces~\eqref{Lem.Moser.Lower} to the case $\varrho=1/2$ and $r=1$.
\begin{steps}
\stepitem\textit{Energy estimates.} By definition of weak solutions, we have
\be{Lem.Moser.Proof.1}
\iint_{Q_1}(-\varphi_t\,v+(\nabla\varphi)^T\,A\,\nabla v)\dx\dt=0
\ee
for any test function $\varphi$ which is compactly supported in $B_1=\{x\in\R^d\,:\,|x|<1\}$, for any fixed $t$. For any $p\in\R\setminus\{0,1\}$, we define
\[
w=v^{p/2}\quad\mbox{and}\quad \varphi=v^{p-1}\,\psi^2\,,
\]
where $\psi$ is a $C^\infty$ function which, like $\varphi$, has compact support in $B_1$ for fixed $t$. We rewrite~\eqref{Lem.Moser.Proof.1} in terms of $w$ and $\psi$ as
\be{Lem.Moser.Proof.2}
\tfrac14\iint \psi^2\,\partial_t w^2\dx\dt+\tfrac{p-1}{p}\iint\psi^2\,(\nabla w)^T\,A\,\nabla w\dx\dt
=-\iint\psi\,w\,(\nabla \psi)^T\,A\,\nabla w\dx\dt
\ee
where we may integrate over a slice $t_1<t<t_2$ of $Q_1$. From here on we adopt the convention that the integration domain is not specified whenever we integrate compactly supported functions on $\R^d$ or on $\R\times\R^d$. Setting $p\ne 1$,
\[
\varepsilon=\tfrac12\left|1-\tfrac1{p}\right|
\]
and recalling that
\[\label{Lem.Moser.Proof.3}
\psi\,w\,(\nabla \psi)^T\,A\,\nabla w\le\frac1{4\,\varepsilon}\,w^2 (\nabla \psi)^T\,A\,\nabla \psi+\varepsilon\,\psi^2\,(\nabla w)^T\,A\,\nabla w\,,
\]
we deduce from~\eqref{Lem.Moser.Proof.2} that
\begin{multline}\label{Lem.Moser.Proof.4}
\pm\,\frac14\iint \partial_t \left(\psi^2\,w^2\right)\dx\dt+\varepsilon\iint\psi^2\,(\nabla w)^T\,A\,\nabla w\dx\dt\\
\le\frac14\iint\(\frac1\varepsilon\,(\nabla \psi)^T\,A\,\nabla \psi+2\,|\psi\,\psi_t|\)w^2\dx\dt\,,
\end{multline}
where the plus sign in front of the first integral corresponds to the case $1/p<1$, while the minus sign corresponds to $1/p>1$. Recall that $p$ can take negative values. Using the ellipticity condition~\eqref{HE.coeff.lambdas} and~\eqref{Lem.Moser.Proof.2}, we deduce
\begin{multline}\label{Lem.Moser.Proof.5}
\pm\,\frac14\iint \partial_t \left(\psi^2\,w^2\right)\dx\dt+\lambda_0\,\varepsilon\iint\psi^2 \left|\nabla w\right|^2\dx\dt\\
\le\frac14\iint\(\frac{\lambda_1}\varepsilon\,|\nabla \psi|^2+2\,|\psi\,\psi_t|\)w^2\dx\dt\,.
\end{multline}
By choosing a suitable test function $\psi$, compactly supported in $Q_1$, and such that
\[\label{Lem.Moser.Proof.6b}
\| \nabla \psi \|_{\mathrm L^\infty(Q_1)}\le\frac2{r-\varrho}\quad\mbox{and}\quad
\|\psi_t \|_{\mathrm L^\infty(Q_1)}\le\frac{4}{r-\varrho}
\]
(see Lemma~\ref{lem.test.funct} in Appendix~\ref{Appendix:Truncation}), we have
\begin{multline}\label{Lem.Moser.Proof.7}
\frac14\iint\(\frac{\lambda_1}\varepsilon\,|\nabla \psi|^2+2\,|\psi\,\psi_t|\)w^2\dx\dt\le\(\frac{\lambda_1}\varepsilon\,\frac1{(r-\varrho)^2}+\frac1{r-\varrho}\)\iint_{\supp(\psi)}\!\!\!\!w^2\dx\dt\\
\le\frac1{(r-\varrho)^2}\(\frac{\lambda_1}\varepsilon+1\)\iint_{\supp(\psi)}\!\!\!\!w^2\dx\dt\,.
\end{multline}
for any $r$ and $\varrho$ such that $0<\varrho<r\le 1$. If $1/p>1$, let us take $\tilde{t}\in (-\varrho^2, \varrho^2)$ to be such that
\[
\int_{B_\varrho}w^2(\tilde{t},x)\dx \ge\frac14\sup_{0<|t|<\varrho^2}\int_{B_\varrho}w^2(t,x)\dx
\]
and choose $\psi$ such that $\psi(0,x)=1$ on $Q_\varrho$ and $\psi(0,x)=0$ outside $Q_r$, so that
\be{Lem.Moser.Proof.8}\begin{split}
\sup_{0<|t|<\varrho^2}\int_{B_\varrho}w^2(t,x)\dx &\le 4 \int_{B_\varrho}w^2(\tilde{t},x)\dx \\ &\le 4 \int_{B_r}w^2(\tilde{t},x)\,\psi^2(\tilde{t},x)\dx
\le 4 \iint_{Q_r} \partial_t \left(\psi^2\,w^2\right)\dx\dt\,.
\end{split}\ee
The same holds true if we replace $Q_r$ by $Q^+_r$ and $0<|t|<\varrho^2$ by $0<t<\varrho^2$.

If $1/p<1$ (which includes the case $p<0$), similar arguments yield
\be{Lem.Moser.Proof.9}
\sup_{-\varrho^2<t<0}\int_{B_\varrho}w^2(t,x)\dx
\le 4 \iint_{Q^-_r} \partial_t \left(\psi^2\,w^2\right)\dx\dt\,.
\ee

\stepitem\textit{Sobolev's inequality.} For any $f\in\mathrm H^1(Q_R)$, we have
\begin{multline}\label{Lem.Moser.Proof.10}
\iint_{Q_R}f^{2\,\gamma}\dx \dt\le\,2\,\pi^2\,\mathcal K
\left[\frac1{R^2}\iint_{Q_R} f^2\dx\dt+\iint_{Q_R}\big|\nabla f\big|^2\dx\dt\right]\\\times\sup_{|s|\in (0,\varrho^2)}\left[\int_{B_R}f^2(s,x)\dx\right]^\frac2d
\end{multline}
with $\gamma=1+2/d$ if $d\ge3$. If $d=1$ or $2$, we rely on~\eqref{sob.step2}, take $\gamma=5/3$, use H\"older's inequality with $2\,\gamma=10/3<4$ and $\pcc\ge4$ if $d=2$, $\pcc>4$ if $d=1$. In order to fix ideas, we take $\pcc=4$ if $d=2$ and $\pcc=8$ if $d=1$. Hence
\[
\iint_{Q_R}f^{2\,\gamma}\dx \dt \le |Q_1|^{1-\frac{2\,\gamma}{\pcc}}
\(\iint_{Q_R}f^{\pcc}\dx \dt\)^{1-\frac{2\,\gamma}{\pcc}}\,.
\]
According to~\eqref{omega-over-d}, we know that $|Q_1|=|(-1,1)|\,|B_1|\le 2\,\pi^2$ in any dimension.
\stepitem\textit{The case $p>0$ and $p\ne 1$.} Assume that $1/2\le\varrho<r\le1$. We work in the cylinder $Q_r=\supp(\psi)$. Here, we choose $\psi(t,x)=\varphi_{\rho, r}(|x|)\,\varphi_{\rho^2, r^2}(|t|)$ where $\varphi_{\rho, r}$ and $\varphi_{\rho^2, r^2}$ are defined in Appendix~\ref{Appendix:Truncation}, so that $\psi=1$ on $Q_\varrho$ and $\psi=0$ outside~$Q_r$.

Collecting inequalities~\eqref{Lem.Moser.Proof.5},~\eqref{Lem.Moser.Proof.7} and~\eqref{Lem.Moser.Proof.8}, we obtain
\[\label{Lem.Moser.Proof.11}
 \sup_{0<|t|<\varrho^2}\int_{B_\varrho}w^2(t,x)\dx+\lambda_0\,\varepsilon\iint_{Q_\varrho} \left|\nabla w\right|^2\dx\dt
\le\frac1{(r-\varrho)^2}\(\frac{\lambda_1}\varepsilon+1\)\iint_{Q_r}\,w^2\dx\dt\,.
\]
Now apply~\eqref{Lem.Moser.Proof.10} to $f=w$ and use the above estimates to get
\begin{align*}\label{Lem.Moser.Proof.12}
&\iint_{Q_\varrho}w^{2\,\gamma}\dx \dt\\
& \le\,2\,\pi^2\,\mathcal K
 \left[\frac1{\varrho^2}\iint_{Q_\varrho}\!\!\!\! w^2\dx\dt+\iint_{Q_\varrho}\big|\nabla w\big|^2\dx\dt\right]
\,\sup_{|s|\in (0,\varrho^2)}\(\int_{B_\varrho}\!\!\!\!w^2(s,x)\dx\)^\frac2d\\
&\le\,2\,\pi^2\,\mathcal K
 \left[\frac1{\varrho^2}\iint_{Q_\varrho} w^2\dx\dt+\frac1{(r-\varrho)^2\,\lambda_0\,\varepsilon}\(\tfrac{\lambda_1}\varepsilon+1\)\iint_{Q_r}\,w^2\dx\dt\right]\\
 &\hspace*{6cm}\times\(\frac1{(r-\varrho)^2}\(\tfrac{\lambda_1}\varepsilon+1\)\iint_{Q_r}\,w^2\dx\dt\)^\frac2d\\
&\le\,2\,\pi^2\,\mathcal K
 \left[\frac1{\varrho^2}+\frac1{(r-\varrho)^2\,\lambda_0\,\varepsilon}\(\tfrac{\lambda_1}\varepsilon+1\)\right]
\,\left[\frac1{(r-\varrho)^2}\(\tfrac{\lambda_1}\varepsilon+1\)\right]^\frac2d\(\iint_{Q_r}\,w^2\dx\dt\)^{\frac2{d}+1}\\
&\hspace*{6cm}:=A(d,\varrho,r,\lambda_0,\lambda_1,\varepsilon, 2\,\pi^2\,\mathcal K) \left( \iint_{Q_r}\,w^2\dx\dt\right)^{\gamma}\,.
\end{align*}
Using the fact that $\mu=\lambda_1+1/\lambda_0>1$ and $1/2\le\varrho<r\le1$, we can estimate the constant $A$ as follows:
\[\begin{split}
A &\le 2\,\pi^2\,\mathcal K\left[\frac1{\varrho^2}+\frac1{(r-\varrho)^2\,\lambda_0\,\varepsilon}\(\tfrac{\lambda_1}\varepsilon+1\)\right]
\(\frac1{(r-\varrho)^2}\(\tfrac{\lambda_1}\varepsilon+1\)\)^\frac2d\\
 &\le\frac{ 2\,\pi^2\,\mathcal K}{(r-\varrho)^{2\,\gamma}}\(\tfrac12+\tfrac{\lambda_1}{\varepsilon^2\,\lambda_0}\)
\(\tfrac{\lambda_1}\varepsilon\)^\frac2d\\
& \le\frac{2\,\pi^2\,\mathcal K}{(r-\varrho)^{2\,\gamma}}\(1+\tfrac{\mu^2}{\varepsilon^2}\)
\(\tfrac\mu\varepsilon\)^\frac2d
 \le\frac{2^5\,\mathcal K}{(r-\varrho)^{2\,\gamma}}\(1+\tfrac\mu\varepsilon\)^{\gamma+1}
\end{split}\]
where we have used that $\lambda_1/\lambda_0\le\frac12(\lambda_1^2+1/\lambda_0^2)\le\frac12(\lambda_1+1/\lambda_0)^2=\mu^2$ and $\pi\le4$.

\noindent\textit{First iteration step.} Recall that $w=v^{p/2}$, $\varepsilon=\frac12\left|1-\frac1p\right|$, and $\gamma=1+\frac2{d}$ if $d\ge3$, $\gamma=5/3$ if $d=1$ or $2$,  $\mu=\lambda_1+1/\lambda_0>1$ and $1/2\le\varrho<r\le1$. We can summarize these results by
\[\label{Lem.Moser.Proof.13}
\left(\iint_{Q_\varrho}v^{\gamma\,p}\dx \dt\right)^{\frac1{\gamma\,p}}
\le\(\frac{(2^5\,\mathcal K)^{\frac1{\gamma}}}{(r-\varrho)^2}\)^\frac1p\left(1+\tfrac\mu\varepsilon\right)^{\frac{\gamma+1}{\gamma\,p}}\left(\iint_{Q_r}\,v^p\dx\dt\right)^\frac1p
\]
for any $p>0$ such that $p\ne 1$. For any $n\in\NN$, let
\[\label{Lem.Moser.Proof.14}
\varrho_n=\frac12\left(1-2^{-n}\right)\,,\quad p_n=\frac{\gamma+1}2\,\gamma^{n-n_0}=p_0\,\gamma^n\,,\quad \varepsilon_n=\frac12\,\left|1-\frac1{p_n}\right|
\]
for some fixed $n_0\in\NN$. Note that $\varrho_0=1$, $p_0=\frac{1+\gamma}{2\,\gamma^{n_0}}$, $\varrho_n$ monotonically decrease to $1/2$, and $p_n$ monotonically increase to $\infty$. We observe that for all $n$, $n_0\in\NN$, we have $p_n\ne 1$ and, as a consequence, $\varepsilon_n>0$. Indeed, if $d\ge3$, $p_n=1$ would mean that
\[
n_0-n=\frac{\log\left(\frac{1+\gamma}2\right)}{\log\gamma}=\frac{\log\(1+\frac1d\)}{\log\(1+\frac2d\)}
\]
and, as a consequence, $0<n_0-n\le\log(4/3)/\log(5/3)<1$, a contradiction with the fact that $n$ and $n_0$ are integers. The same argument holds if $d=1$ or $d=2$ with $n_0-n=\log(4/3)/\log(5/3)$, as $\gamma=5/3$ corresponds to the value of $\gamma$ for $d=1$, $2$ or $3$. It is easy to check that for any $n\ge 0$,
\[\label{Lem.Moser.Proof.15}
|p_n-1|\ge\min\{p_{n_0}- 1,1-p_{n_0-1}\}=\min\left\{\tfrac1d,\tfrac1{d+2}\right\}=\tfrac1{d+2}\,.
\]
For an arbitrary $p\in(0,1/\mu)$, we choose
\[
n_0={\rm i.p.}\left(\frac{\log\left(\frac{1+\gamma}{2\,p}\right)}{\log\gamma}\right)+1
\]
where ${\rm i.p.}$ denotes the integer part, so that $0<p_0\le p<\gamma\,p_0$. By monotonicity of the $\mathrm L^q$ norms, that is,
\[
\left(\iint_{Q_r} v^{p_0}\,\frac{\dx\dt}{|Q_r|}\right)^{\frac1{p_0}}\le \left(\iint_{Q_r} v^p\, \frac{\dx\dt}{|Q_r|}\right)^\frac1p
\le \left(\iint_{Q_r} v^{\gamma\,p_0} \frac{\dx\dt}{|Q_r|}\right)^{\frac1{\gamma\,p_0}}\,,
\]
it is sufficient to prove inequality~\eqref{Lem.Moser.Upper} for $p=p_0$.

Let us define $p_\mu\in(p_0\,\mu,1]$ such that
\be{Lem.Moser.Proof.16}
1+\frac\mu{\varepsilon_n}=1+\frac{2\,\mu\,p_n}{|p_n-1|}=1+\frac{2\,\mu\,p_0\,\gamma^n}{|p_n-1|}\le 1+2\,(d+2)\,\gamma^n\le4\,(d+2)\,\gamma^n=4\,d\,\gamma^{n+1}
\ee
because $d+2=d\,\gamma$ if $d\ge3$ and $\gamma=5/3$ if $d\le3$. Finally, let us define
\[
Y_n:=\left(\iint_{Q_{\varrho_n}}v^{p_n}\dx \dt\right)^\frac1{p_n}\,,\;I_0=(2^5\,\mathcal K)^{\frac1{\gamma}} (4\,d\,\gamma^2)^\frac{\gamma+1}\gamma
\]
and $C=4\,\gamma^\frac{\gamma+1}\gamma$, $\theta=\frac1\gamma\in (0,1)$, and $\xi=\frac1{p_0}$.

\noindent\textit{Iteration.} Summing up, we have the following iterative inequality
\[
Y_n\le\(\frac{(2^5\,\mathcal K)^{\frac1\gamma}}{(\varrho_{n-1}-\varrho_n)^2}
\left(1+\tfrac\mu{\varepsilon_n}\right)^\frac{\gamma+1}\gamma\)^\frac1{p_{n-1}}\,Y_{n-1}\,.
\]
Using $\varrho_{n-1}-\varrho_n=2^{-n}$ and inequality~\eqref{Lem.Moser.Proof.16}, we obtain
\be{hyp.num}
Y_n\le I_{n-1}^{\,\xi\,\theta^{n-1}}\,Y_{n-1}\quad\mbox{with}\quad I_{n-1}\le I_0\,C^{\,n-1}\,.
\ee
%---------------------------------------------------------------------
\begin{lemma}[See~\cite{Bonforte2012a}]\label{lem.num.iter} The sequence $(Y_n)_{n\in\N}$ is a bounded sequence such that
\[\label{iteration.num}
Y_\infty:=\limsup_{n\to+\infty}Y_n\le I_0^{\frac{\xi}{1-\theta}}\,C^{\frac{\xi\,\theta}{(1-\theta)^2}}\,Y_0\,.
\]
\end{lemma}
%---------------------------------------------------------------------
The proof follows from the observation that
\begin{multline*}
Y_n  \le I_{n-1}^{\,\xi\,\theta^{n-1}}Y_{n-1}\le \left(I_0\,C^{\,n-1}\right)^{\,\xi\,\theta^{n-1}}\,Y_{n-1}
=I_0^{\,\xi\,\theta^{n-1}}\,C^{\,\xi\,(n-1)\,\theta^{n-1}}\,Y_{n-1}\\
\le \prod_{j=0}^{n-1}\,I_0^{\,\xi\,\theta^j} C^{\,\xi\,j\,\theta^j}\,Y_0
=I_0^{\,\xi\sum_{j=0}^{n-1}\theta^j} C^{\,\xi\sum_{j=0}^{n-1}j\,\theta^j}\,Y_0\,.
\end{multline*}

With the estimates
\[
\left(\iint_{Q_1}\,v^{p_0}\dx\dt\right)^{\frac1{p_0}}\le|Q_1|^{\frac1{p_0}-\frac1p}\,\left(\iint_{Q_1}\,v^{p}\dx\dt\right)^\frac1p\,,
\]
$\frac1{p_0}-\frac1p\le\frac{\gamma-1}p$ and $|Q_1|=2\,|B_1|\le2\,\pi^2$, we obtain 
\[
\sup_{Q_{1/2}}v\le\(2^5\,\mathcal K\,(4\,d\,\gamma^2)^{\gamma+1 }\)^{\frac1p\,\frac\gamma{\gamma-1}}
\(4^\gamma\,\gamma^{\gamma+1}\)^{\frac1p\,\frac\gamma{(\gamma-1)^2}}\,(2\,\pi^2)^\frac{\gamma-1}p
\left(\iint_{Q_1}\,v^{p}\dx\dt\right)^\frac1p
\]
which, using $2\,\pi^2\le24$  and after raising to the power $p$, is~\eqref{Lem.Moser.Upper} with $c_1$ given by~\eqref{Lem.Moser.constant}.

\stepitem\textit{The case $p<0$.} Assume that $1/2\le\varrho<r\le1$. We work in the cylinder $Q_r^-=\supp(\psi)$. Here, we choose $\phi(t,x)=\varphi_{\rho, r}(|x|)\,\varphi_{\rho^2, r^2}(-t)$, where $\varphi$ is defined as in Appendix~\ref{Appendix:Truncation}, so that $\psi=1$ on $Q_\varrho^-$ and $\psi=0$ outside $Q_r^-$.

After collecting~\eqref{Lem.Moser.Proof.5},~\eqref{Lem.Moser.Proof.7} and~\eqref{Lem.Moser.Proof.9}, we obtain
\[\label{Lem.Moser.Proof.19}
\sup_{-\varrho^2<t<0}\int_{B_\varrho}w^2(t,x)+\lambda_0\,\varepsilon\iint_{Q_\varrho^-} \left|\nabla w\right|^2\dx\dt
\le\frac1{(r-\varrho)^2}\(\tfrac{\lambda_1}\varepsilon+1\)\iint_{Q_r^-}\,w^2\dx\dt\,.
\]
Then the proof follows exactly the same scheme as for $p>0$, with the simplification that we do not have to take extra precautions in the choice of $p$. The constant $c_1$ is the same. 
\end{steps}
\end{proof}

%%%%%%%%%%%%%%%%%%%%%%%%%%%%%%%%%%%%%%%%%%%%%%%%%%%%%%%%%%%%%%%%%%%%%%%%%%
%%%%%%%%%%%%%%%%%%%%%%%%%%%%%%%%%%%%%%%%%%%%%%%%%%%%%%%%%%%%%%%%%%%%%%%%%%
\newpage\section{Logarithmic Estimates}\label{Sec:LogarithmicEstimates}

We prove now fine level set estimates on the solutions by Cacciop\-poli-type energy estimates. These estimates are based on a weighted Poincar\'e inequality (see Step~2 of the proof of Lemma~\ref{Lem.Log.Est}) originally due to F. John, as explained by J.~Moser in~\cite{Moser1964}). This is a fundamental step for this approach and for the more standard approach based on BMO and John-Nirenberg estimates. The level set estimates are better understood in terms of
\[
w=-\log v\,,
\]
the \emph{logarithm of $v$}, solution to~\eqref{HE.coeff}, which satisfies the nonlinear equation
\[
 w_t =-\frac{v_t}{v}=\sum_{i,j=1}^d\partial_i \big( A_{i,j}(t,x)\,\partial_j (-\log v) \big)- \sum_{i,j=1}^d (\partial_i \log v)\,A_{i,j}(t,x)\,(\partial_j \log v)\,,
\]
\emph{i.e.},
\be{HE.coeff.log}
w_t=\nabla\cdot\(A\,\nabla w\) -(\nabla w)^T\,A\,\nabla w\,.
\ee
All computations can be justified by computing with $-\log(\delta+v)$ for an arbitrarily small $\delta>0$ and passing to the limit as $\delta\to0_+$. We recall that $\mu=\lambda_1+1/\lambda_0$. Let us choose a test function $\psi$ as follows:
\be{Lem.Log.Est.3}
\psi(x):=\prod_{\nu=1}^d \chi_\nu(x_\nu),\quad\mbox{where}\quad
\chi_\nu(z):=\left\{\begin{array}{lll}
1 &\quad\mbox{if }|z|\le 1\\
2-|z|&\quad\mbox{if }1\le|z|\le 2\\
0 &\quad\mbox{if }|z|\ge 2\\
\end{array}\right.\,.
\ee
Note that this test function has convex super-level sets, or equivalently said, on any straight line segment, $\psi(x)$ assumes its minimum at an end point.

Even if~\eqref{HE.coeff.log} is a nonlinear equation, the nonlinear term actually helps. The reason for that lies in the following result.
%---------------------------------------------------------------------
\begin{lemma}\label{Lem.Log.Est} Assume that $\psi$ is a smooth compactly supported test function as in~\eqref{Lem.Log.Est.3}. If $w$ is a (sub)solution to~\eqref{HE.coeff.log} in
\[
\big\{(t,x)\in\R\times\R^d\,:\,|t|<1\,,\;|x|<2\big\}=(-1,1)\times B_2(0)\,,
\]
then there exist positive constants $a$ and $c_2(d)$ such that, for all $s>0$,
\begin{multline}\label{Lem.Log.Est.Ineq.w}
\left|\big\{(t,x)\in Q^+_1\,:\,w(t,x)>s-a\big\}\right|\\
+\left|\big\{(t,x)\in Q^-_1\,:\,w(t,x)<-s-a \big\}\right|\le c_2\,|B_1|\,\frac\mu s\,,
\end{multline}
where
\be{Lem.Log.Est.Ineq.a}
c_2=2^{d+2}\,3^d\,d\quad\mbox{and}\quad a=-\frac{\int w(0,x)\,\psi^2(x)\dx}{\int \psi^2(x)\dx}\,.
\ee
\end{lemma}
%---------------------------------------------------------------------
Equivalently, the above inequality stated in terms of $v$ reads
\begin{multline}\label{Lem.Log.Est.Ineq.u}
\left|\big\{(t,x)\in Q^+_1\,:\,\log v(t,x)<-s+a\big\}\right|\\
+\left|\big\{(t,x)\in Q^-_1\,:\,\log v(t,x)>s+a \big\}\right|\le c_2\,|B_1|\,\frac\mu s\,,
\end{multline}
where $a=\int \log v(0,x)\,\psi^2(x)\dx/\int \psi^2(x)\dx$.
\begin{proof}
\begin{steps}
We follow the proof of Lemma 2 of~\cite{Moser1971}, which in turn refers to~\cite[p.~121-123]{Moser1964}. We provide some minor improvements and quantify all constants. For better readability, we split the proof into several steps.

\stepitem\textit{Energy estimates.} Testing equation (or inequality)~\eqref{HE.coeff.log} with $\psi^2(x)$, we obtain
\begin{multline}\label{Lem.Log.Est.1}
\int \psi^2\,w(t_2)\dx-\int \psi^2\,w(t_1)\dx+\frac12\iint \psi^2\,(\nabla w)^T\,A\,\nabla w\dx\dt\\\le 2 \iint (\nabla \psi)^T\,A\,\nabla \psi\dx\dt\,.
\end{multline}
Using the conditions~\eqref{HE.coeff.lambdas}, we have that
\begin{align*}\label{Lem.Log.Est.2}
&\lambda_0 \iint \psi^2\,|\nabla w|^2 \dx\dt \le\iint \psi^2\,(\nabla w)^T\,A\,\nabla w\dx\dt\,,\\
&\iint (\nabla \psi)^T\,A\,\nabla \psi\dx\dt \le \lambda_1 \iint |\nabla \psi|^2 \dx\dt\,.
\end{align*}
Combining the above two inequalities, we obtain
\be{Lem.Log.Est.2b}\begin{split}
\int \psi^2\,w(t_2)\dx&-\int \psi^2\,w(t_1)+\frac{\lambda_0}2 \iint \psi^2\,|\nabla w|^2 \dx\dt\\
&\le 2\,\lambda_1 \iint |\nabla \psi|^2 \dx\dt\le 2^d\,\lambda_1\,(t_2-t_1)\,|B_1|\,\|\nabla \psi\|_{\mathrm L^\infty}^2\,.
\end{split}\ee

\stepitem\textit{Weighted Poincar\'e inequalities.} Let $b\ge 0$ be a continuous function with support of diameter $D=\mathrm{diam}(\supp(\psi))$ such that the domains $\{x\in\R^d\,:\,b(x)\ge \mbox{const}\}$ are convex. Then for any function $f\in\mathrm L^2_{b}$ with $|\nabla f|\in\mathrm L^2_b$, we have that
\[\label{Lem.Log.Est.4}
\int \left|f(x)-\overline{f}_b\right|^2 b(x)\dx \le \lambda_b\,D^2\int |\nabla f(x)|^2 b(x)\dx
\]
where
\[\label{Lem.Log.Est.4b}
\lambda_b=\frac{|\supp(b)|\,\|b\|_{\mathrm L^\infty}}{2\int b(x)\dx}
\quad\mbox{and}\quad \overline{f}_b=\frac{\int f(x)\,b(x)\dx}{\int b(x)\dx}\,.
\]
The proof follows from the unweighted Poincar\'e inequality: see for instance~\cite[Lemma~3]{Moser1964}.

\noindent\textit{Poincar\'e inequality with weight $\psi^2$.} We have that $D=2\,d$ in the particular case of $b=\psi^2$, where $\psi$ is given in~\eqref{Lem.Log.Est.3} and such that $0\le\psi\le1$, as for the constant $\lambda_b$ we have
\[
\lambda_b\le\frac{|B_2|}{2\int_{B_1}\psi^2\dx}=\frac{|B_2|}{2\,|B_1|}=2^{d-1}\,.
\]
Since $\|b\|_{\mathrm L^\infty}=\|\psi^2\|_{\mathrm L^\infty}=1$, $|B_1|\le \int\psi^2\dx\le3^d\,|B_1|$, we obtain
\be{Lem.Log.Est.5}
\iint \left|w(t,x)-\overline{w(t)}_\psi\right|^2\,\psi^2(x)\dx\dt \le 2^d\,d \iint |\nabla w(t,x)|^2\,\psi^2(x)\dx\dt\,,
\ee
with
\[
\overline{w(t)}_\psi:=\frac{\int w(t,x)\,\psi^2(x)\dx}{\int \psi^2(x)\dx}\,.
\]

\stepitem\textit{Differential inequality.} Let us recall that $\|\nabla \psi\|_{\mathrm L^\infty}^2\le 1$. We combine inequalities~\eqref{Lem.Log.Est.2b} and~\eqref{Lem.Log.Est.5} into
\[\label{Lem.Log.Est.6}\begin{split}
\int \psi^2\,w(t_2)\dx-\int \psi^2\,w(t_1)+\frac{\lambda_0}{2^{d+1}\,d} \int_{t_1}^{t_2}\int \left|w(t,x)-\overline{w(t)}_\psi\right|^2\,\psi^2(x)\dx\dt
\\ \le 2^d\,\lambda_1\,(t_2-t_1)\,|B_1|\,.
\end{split}\]
Recalling that $\psi=1$ on $B_1$ and the expression of $\overline{w(t)}_\psi$ given in~\eqref{Lem.Log.Est.5}, we obtain
\[\label{Lem.Log.Est.7}\begin{split}
\frac{\overline{w(t_2)}_\psi-\overline{w(t_1)}_\psi}{t_2-t_1}+\frac{\lambda_0}{2^{d+1}\,3^d\,d}\frac1{(t_2-t_1)\,|B_1|}\int_{t_1}^{t_2}\int_{B_1} \left|w(t,x)-\overline{w(t)}_\psi\right|^2 \dx\dt\\
\le\frac{2^d\,\lambda_1\,|B_1|}{\int\psi^2\dx}\le\,2^d\,\lambda_1\,.
\end{split}\]
Here we have used that $|B_1|\le \int\psi^2\dx\le3^d\,|B_1|$. Recalling that $\mu=\lambda_1+1/\lambda_0$, so that $\lambda_0\,\mu>1$, we obtain
\[\label{Lem.Log.Est.7b}\begin{split}
\frac{\overline{w(t_2)}_\psi-\overline{w(t_1)}_\psi}{t_2-t_1}+\frac1{2^{d+1}\,3^d\,d\,\mu}\frac1{(t_2-t_1)\,|B_1|}\int_{t_1}^{t_2}\int_{B_1} \left|w(t,x)-\overline{w(t)}_\psi\right|^2 \dx\dt \\
\le\frac{2^d\,\lambda_1\,|B_1|}{\int\psi^2\dx}\le\,2^d\,\mu\,.
\end{split}\]

Letting $t_2\to t_1$ we obtain the following differential inequality for $\overline{w(t)}_\psi $
\be{Lem.Log.Est.8}
\frac{\rd}{\dt}\overline{w(t)}_\psi+\frac1{2^{d+1}\,3^d\,d\,\mu}\frac1{|B_1|}\int_{B_1} \left|w(t,x)-\overline{w(t)}_\psi\right|^2\dx
\le\,2^d\,\mu\,.
\ee
The above inequality can be applied to
\[
\underline{w}(t,x)=w(t,x)-\overline{w(0)}_\psi-2^d\,\mu\,t\,.
\]
Notice that $\underline{w}$ is a subsolution to~\eqref{HE.coeff.log} since $w$ is. With $a=-\overline{w(0)}_\psi$, we can write~\eqref{Lem.Log.Est.8} in terms~of
\[
W(t)=\overline{w(t)}_\psi+a-2^d\,\mu\,t\quad\mbox{such that}\quad W(0)=0
\]
as
\[
\frac{\rd}{\dt}W(t)+\frac1{2^{d+1}\,3^d\,d\,\mu}\frac1{|B_1|}\int_{B_1} \left|\underline{w}(t,x)-W(t)\right|^2\dx
\le 0\,.
\]
An immediate consequence of the above inequality is that $W(t)\le W(0)=0$ for all $t\in (0,1)$.

Let $Q_s(t)=\{x\in B_1\,:\,w(t,x)>s\}$, for a given $t\in (0,1)$. For any $s>0$, we have
\[
\underline{w}(t,x)-W(t)\ge s-W(t)\ge 0\quad\forall\,x\in Q_s(t)\,,
\]
because $W(t)\le 0$ for $t\in (0,1)$. Using $\frac{\rd}{\dt}W=-\frac{\rd}{\dt}(s-W)$, the integration restricted to $Q_s$ in~\eqref{Lem.Log.Est.8} gives
\[\label{Lem.Log.Est.9}
\frac{\rd}{\dt}\big(s-W(t)\big) \ge\frac1{2^{d+1}\,3^d\,d\,\mu}\frac{|B_s(t)|}{|B_1|}\,\big(s-W(t)\big)^2\,.
\]
By integrating over $(0,1)$, it follows that
\begin{multline*}
\left|\left\{(t,x)\in Q_1^+\,:\,\underline{w}(t,x)>s\right\}\right|
=\iint_{\{\underline{w}>s\}\cap Q_1^+}\dx\dt=\int_0^1 |Q_s(t)|\dt\\
\le 2^{d+1}\,3^d\,d\,\mu\,|B_1|\left(\frac1{s-W(0)}-\frac1{s-W(1)}\right)\le 2^{d+1}\,3^d\,d\,|B_1|\,\frac\mu s\,,
\end{multline*}
which proves the first part of inequality~\eqref{Lem.Log.Est.Ineq.w}.

\stepitem\textit{Estimating the second term of inequality~\eqref{Lem.Log.Est.Ineq.w}.} We just replace $t$ by $-t$ and repeat the same proof. Upon setting $a=-\overline{w(0)}_\psi$, we obtain
\[\label{Lem.Log.Est.11}
 \left|\left\{(t,x)\in Q_1^-\,:\,w<-s-a\right\}\right| \le2^{d+1}\,3^d\,d\,|B_1|\,\frac\mu s\,.
 \]
\end{steps}
\end{proof}

%%%%%%%%%%%%%%%%%%%%%%%%%%%%%%%%%%%%%%%%%%%%%%%%%%%%%%%%%%%%%%%%%%%%%%%%%%
%%%%%%%%%%%%%%%%%%%%%%%%%%%%%%%%%%%%%%%%%%%%%%%%%%%%%%%%%%%%%%%%%%%%%%%%%%
\section{A lemma by Bombieri and Giusti}\label{Sec:Bombieri-Giusti}

To avoid the direct use of BMO spaces (whose embeddings and inequalities, like the celebrated John-Nirenberg inequality, may not have explicit constants), we use the parabolic version, due to J.~Moser, of a Lemma attributed to E.~Bombieri and E.~Giusti, in the elliptic setting: see~\cite{Bombieri_1972}. We use the version of~\cite[Lemma~3]{Moser1971}, which applies to measurable functions $f$, not necessarily solutions to a PDE, and to any family of domains $(\Omega_r)_{0<r<R}$ such that $\Omega_r\subset\Omega_R$.
%---------------------------------------------------------------------
\begin{lemma}[Bombieri-Giusti~\cite{Bombieri1972}, Moser~\cite{Moser1971}]\label{BGM.Lemma}
Let $\beta$, $c_1$, $\mu>0$, $c_2\ge 1/\mathrm{e}$, $\theta\in [1/2,1)$ and $p\in(0,1/\mu)$ be positive constants, and let $f>0$ be a positive measurable function defined on a neighborhood of $\Omega_1$ for which
\be{hyp.BGM.Lemma.1}
\sup_{\Omega_\varrho}f^p<\frac{c_1}{(r-\varrho)^\beta\,|\Omega_1|}\iint_{\Omega_r}f^p\dx\dt
\ee
for any $r$ and $\varrho$ such that $\theta\le\varrho<r\le 1$, and
\be{hyp.BGM.Lemma.2}
\big|\big\{ (t,x)\in\Omega_1\,:\,\log f>s\big\}\big|<c_2\,|\Omega_1|\,\frac\mu s\quad\forall\,s>0\,.
\ee
Let $\sigma$ be as in~\eqref{sigma}. Then we have
\be{BGM.Lemma.ineq}
\sup_{\Omega_\theta} f<\kappa_0^\mu\,,\quad\mbox{where}\quad \kappa_0:=\exp\left[2\,c_2\vee\frac{8\,c_1^3}{(1-\theta)^{2\,\beta}}\right]\,.
\ee
\end{lemma}
%---------------------------------------------------------------------
The difference between the upper bounds~\eqref{hyp.BGM.Lemma.1} and~\eqref{BGM.Lemma.ineq} is subtle. The first inequality depends on the solution on the whole space-time set $\Omega_r$ and is somehow implicit. By assumption~\eqref{hyp.BGM.Lemma.2}, if the set where $f$ is super-exponential has small measure, then on a slightly smaller set the solution is quantitatively bounded by an explicit and uniform constant, given by~\eqref{BGM.Lemma.ineq}.

\begin{proof}[Proof of Lemma~\ref{BGM.Lemma}] We sketch the relevant steps of the proof of~\cite[Lemma~3]{Moser1971}. Our goal is to provide some minor technical improvements and quantify all constants. Without loss of generality, after replacing $s$ by $s\,\mu$, we reduce the problem to the case $\mu=1$. Analogously, we also assume that $|\Omega_1|=1$. We define the nondecreasing function
\[
\varphi(\varrho)=\sup_{\Omega_\varrho}(\log f)\quad\forall\,\varrho\in[\theta,1)\,.
\]
We will prove that assumptions~\eqref{hyp.BGM.Lemma.1} and~\eqref{hyp.BGM.Lemma.2} imply the following dichotomy:\\
-- either $\varphi(r)\le 2\,c_2$ and there is nothing  to prove: $\kappa_0=\mathrm{e}^{2\,c_2}$,\\
-- or $\varphi(r)>2\,c_2$ and we have
\be{BGM.Lemma.01}
\varphi(\varrho)\le\frac34\,\varphi(r)+\frac{8\,c_1^3}{(r-\varrho)^{2\,\beta}}
\ee
for any $r$ and $\varrho$ such that $\theta\le\varrho<r\le 1$. We postpone the proof of~\eqref{BGM.Lemma.01} and observe that~\eqref{BGM.Lemma.01} can be iterated along a monotone increasing sequence $(\varrho_k)_{k\ge0}$ such that
\[
\theta\le \varrho_0<\varrho_1<\dots<\varrho_k\le 1
\]
for any $k\in\N$ to get
\[
\varphi(\varrho_0)<\frac34\,\varphi(\varrho_k)+8\,c_1^3\sum_{j=0}^{k-1}\left(\tfrac34\right)^j\frac1{\left(\varrho_{j+1}-\varrho_j\right)^{2\,\beta}}\,.
\]
By monotonicity, we have that $\varphi(\varrho_k)\le \varphi(1)<\infty$, so that in the limit as $k\to+\infty$, we obtain
\[
\varphi(\theta)\le \varphi(\varrho_0)\le 8\,c_1^3\sum_{j=0}^{\infty}\left(\tfrac34\right)^j\frac1{\left(\varrho_{j+1}-\varrho_j\right)^{2\,\beta}}
\]
provided the right-hand side converges. This convergence holds true for the choice
\[
\varrho_j=1-\frac{1-\theta}{1+j}\,,
\]
and in that case, the estimate
\[\label{BGM.Lemma.03}
\varphi(\theta)\le\frac{8\,c_1^3\,\sigma}{(1-\theta)^{2\,\beta}}:=\tilde\kappa_0
\]
implies inequality~\eqref{BGM.Lemma.ineq} with $\mu=1$ because
\[
\sup_{\Omega_\theta} f \le \exp\(\sup_{\Omega_\theta}(\log f)\)=\mathrm{e}^{\varphi(\theta)}\le \mathrm{e}^{\tilde\kappa_0}:=\kappa_0\,.
\]
In order to complete the proof, we have to prove inequality~\eqref{BGM.Lemma.01}.

\noindent\textit{Proof of Inequality~\eqref{BGM.Lemma.01}.} We are now under the assumption $\varphi(r)>2\,c_2$. We first estimate the integral
\be{BGM.Lemma.04}\begin{split}
\iint_{\Omega_r} f^p\dx\dt
&=\iint_{\{\log f>\frac12\,\varphi(r)\}} f^p\dx\dt
+\iint_{\{\log f\le\frac12\,\varphi(r)\}} f^p\dx\dt\\
&\le \mathrm{e}^{p\,\varphi(r)}\left|\left\{ (t,x)\in\Omega_1\,:\,\log f>\tfrac12\,\varphi(r)\right\}\right|+|\Omega_1|\,\mathrm{e}^{\frac p2\,\varphi(r)}\\
&\le\frac{2\,c_2}{\varphi(r)}\,\mathrm{e}^{p\,\varphi(r)}+\mathrm{e}^{\frac p2\,\varphi(r)}\,,
\end{split}
\ee
where we have estimated the first integral using that
\[
\sup_{\Omega_r} f^p \le \sup_{\Omega_r} \mathrm{e}^{p \log f} \le \mathrm{e}^{p \sup_{\Omega_r} \log f}=\mathrm{e}^{p\,\varphi(r)}\,.
\]
In the present case, assumption~\eqref{hyp.BGM.Lemma.2} reads:
\[
\left|\left\{ (t,x)\in\Omega_1\,:\,\log f>\tfrac12\,\varphi(r)\right\}\right|<\frac{2\,c_2}{\varphi(r)}.
\]
We choose
\[
p=\frac2{\varphi(r)}\log\left(\frac{\varphi(r)}{2\,c_2}\right)
\]
such that the last two terms of~\eqref{BGM.Lemma.04} are equal, which gives
\be{BGM.Lemma.05}
 \iint_{\Omega_r} f^p\dx\dt \le 2\,\mathrm{e}^{\frac p2\,\varphi(r)}\,.
\ee
The exponent $p$ is admissible, that is, $0<p<1/\mu=1$, if $\varphi(r)>2/\mathrm e$, which follows from the assumption $c_2>1/\mathrm e$. Now, using assumption~\eqref{hyp.BGM.Lemma.1} and inequality~\eqref{BGM.Lemma.05}, we obtain
\begin{multline*}
\varphi(\varrho)=\frac1p \sup_{\Omega_\varrho}\log(f^p)=\frac1p \log\(\sup_{\Omega_\varrho} f^p\)\\
\le\frac1p\log\(\frac{c_1}{(r-\varrho)^\beta}\iint_{\Omega_r}f^p \dx \dt\)\hspace*{4cm}\\
\le\frac1p\log\(\frac{2\,c_1\,\mathrm{e}^{\frac p2\,\varphi(r)}}{(r-\varrho)^\beta}\)
=\frac1p\log\(\frac{2\,c_1}{(r-\varrho)^\beta}\)+\frac12\,\varphi(r)\\
\hspace*{4cm}=\frac12\,\varphi(r)\(1+\frac{\log(2\,c_1)-\log(r-\varrho)^{\beta}}{\log(\varphi(r))-\log(2\,c_1)}\)\\
 \le\frac12\,\varphi(r)\(1+\frac12\)=\frac34\,\varphi(r)\,.
\end{multline*}
In the last line, we take
\[
\varphi(r)\ge\frac{8\,c_1^3}{(r-\varrho)^{2\,\beta}}
\]
so that
\be{BGM.Lemma.07}
\frac{\log(2\,c_1)-\log(r-\varrho)^{\beta}}{\log(\varphi(r))-\log(2\,c_1)}\le\frac12\,.
\ee
We again have that either $\varphi(r)<\frac{8\,c_1^3}{(r-\varrho)^{2\,\beta}}$ and~\eqref{BGM.Lemma.01} holds, or $\varphi(r)\ge\frac{8\,c_1^3}{(r-\varrho)^{2\,\beta}}$ and~\eqref{BGM.Lemma.07} holds, hence $\varphi(\varrho)\le\frac34\,\varphi(r)$. We conclude that~\eqref{BGM.Lemma.01} holds in all cases and this completes the proof.
\end{proof}

%%%%%%%%%%%%%%%%%%%%%%%%%%%%%%%%%%%%%%%%%%%%%%%%%%%%%%%%%%%%%%%%%%%%%%%%%%
%%%%%%%%%%%%%%%%%%%%%%%%%%%%%%%%%%%%%%%%%%%%%%%%%%%%%%%%%%%%%%%%%%%%%%%%%%
\section{Proof of Moser's Harnack inequality}\label{Sec:Moser-Harnack}

\begin{proof}[Proof of Theorem~\ref{Claim:3}]
\begin{steps}
We prove the Harnack inequality
\be{Harnack.Ineq}
\sup_{D^-}v\le \mathsf h^\mu\,\inf_{D^+} v
\ee
where $\mathsf h$ is as in~\eqref{h} and $D^\pm$ are the parabolic cylinders given by
\[\label{Parab.Cylinders.unitary}\begin{split}
&D=\left\{ |t|<1\,,\;|x|<2\right\}=(-1,1)\times B_2\,,\\
&D^+=\left\{\tfrac34<t<1\,,\;|x|<\tfrac12\right\}=\left(\tfrac34,1\right)\times B_{1/2}(0)\,,\\
&D^-=\left\{ -\tfrac34<t<-\tfrac14\,,\;|x|<\tfrac12\right\}=\left(-\tfrac34,-\tfrac14\right)\times B_{1/2}(0)\,.
\end{split}
\]
The general inequality~\eqref{harnack} follows by applying the admissible transformations corresponding to~\eqref{admissible.transformations}, which do not alter the values of $\lambda_0$, $\lambda_1$ and $\mu=\lambda_1+1/\lambda_0$.

Let $v$ be a positive solution to~\eqref{HE.coeff} and $a\in\R$ to be fixed later. In order to use Lemma~\ref{Lem.Moser} and Lemma~\ref{Lem.Log.Est}, we apply Lemma~\ref{BGM.Lemma} to
\[
v_+(t,x)=\mathrm{e}^{-a}\,v(t,x)\quad\mbox{and}\quad v_-(t,x)=\frac{\mathrm{e}^{+a}}{v(t,x)}\,.
\]
\stepitem\textit{Upper estimates.} Let us prove that
\be{Proof.Harnack.claim.1}
\sup_{D^-}v_+\le \ka_0^{\,\mu}
\ee
where $\ka_0$ has an explicit expression, given below in~\eqref{proof.Harnack.5}. For all $\varrho\in\left[1/2,1\right)$, let
\[\begin{split}
\Omega_\varrho&:=\left\{(t,x)\in\Omega_1\,:\,\left|t+\tfrac12\right|<\tfrac12\,\varrho^2\,,\;|x|<\varrho/\sqrt2\right\}\\
&=\left(-\tfrac12\,(\varrho^2+1),\tfrac12\,(\varrho^2-1)\right)\times B_{\varrho/\sqrt2}(0)=Q_{\varrho/\sqrt2}\left(-\tfrac12,0\right)\,.
\end{split}\]
Note that if $\varrho=1/\sqrt2$, then $\Omega_\varrho=\left(-3/4,-1/4\right)\times B_{1/2}(0)=D^-$, and also that $\Omega_\varrho\subset\Omega_1=(-1,0)\times B_1(0)=Q_1^-$ for any $\varrho\in\left[1/2,1\right)$.

The first assumption of Lemma~\ref{BGM.Lemma}, namely inequality~\eqref{hyp.BGM.Lemma.1} with $\beta=d+2$ is nothing but inequality~\eqref{Lem.Moser.Upper} of Lemma~\ref{Lem.Moser} applied to $\Omega_\varrho=Q_{\varrho/\sqrt2}\left(-1/2,0\right)$, that is,
\be{proof.Harnack.1}
\sup_{\Omega_\varrho} v_+^p \le\frac{c_1\,2^\frac{d+2}2}{(r-\varrho)^{d+2}}\iint_{\Omega_r}v_+^p \dx\dt\quad\forall\,p\in(0,1/\mu)\,.
\ee
Note that the results of Lemma~\ref{Lem.Moser} hold true for these cylinders as well, with the same constants, since $Q_{\varrho/\sqrt 2}(-1/2,0)$ can be obtained from $Q_{\varrho}(0,0)$ by means of admissible transformations~\eqref{admissible.transformations} which leave the class of equations unchanged, \emph{i.e.}, such that $\lambda_1$, $\lambda_0$ and $\mu$ are the same.

The second assumption of Lemma~\ref{BGM.Lemma}, namely inequality~\eqref{hyp.BGM.Lemma.2} of Lemma~\ref{BGM.Lemma}, if stated in terms of super-level sets of $\log v_+$, reads
\[\label{proof.Harnack.2}
\left|\left\{x\in\Omega_1\,:\,\log v_+>s\right\}\right|=\left|\{(t,x)\in Q^-_1\,:\,\log v>s+a \}\right|\le c_2\,|B_1|\,\frac\mu s
\]
according to Lemma~\ref{Lem.Log.Est}. Hence we are in the position to apply Lemma~\ref{BGM.Lemma} with $\theta=1/\sqrt 2$ to conclude that~\eqref{Proof.Harnack.claim.1} is true with
\be{proof.Harnack.5}
\ka_0:=\exp\left[2\,c_2\vee\frac{8\,c_1^3(\sqrt 2)^{3\,(d+2)}\,\sigma}{(1-1/\sqrt 2)^{2\,(d+2)}}\right]\,.
\ee
This concludes the first step.

\stepitem\textit{Lower estimates.} Let us prove that
\be{Proof.Harnack.claim.2}
\sup_{D^+}v_- \le \kb_0^{\,\mu}
\ee
where $\kb_0$ has an explicit expression, given below in~\eqref{proof.Harnack.10}. For all $\varrho\in\left[1/2,1\right)$, let
\[
\Omega_\varrho=\left\{(t,x)\in\Omega_1\,:\,0<1-t<\varrho^2\,,\;|x|<\varrho\right\}
=\left(1-\varrho^2, 1\right)\times B_{\varrho}(0)=Q^-_\varrho(1,0)\,.
\]
Note that if $\varrho=1/2$ then $\Omega_\varrho=\left(3/4,1\right)\times B_{1/2}(0)=D^+$, and $\Omega_\varrho\subset\Omega_1=(0,1)\times B_1(0)=Q_1^+$ for any $\varrho\in\left[1/2,1\right)$.

The first assumption of Lemma~\ref{BGM.Lemma}, namely inequality~\eqref{hyp.BGM.Lemma.1} with $\beta=d+2$ is nothing but inequality~\eqref{Lem.Moser.Lower} of Lemma~\ref{Lem.Moser} applied to $\Omega_\varrho=Q^-_\varrho(1,0)$
\[\label{proof.Harnack.6}
\sup_{\Omega_\varrho} v_-^p \le\frac{c_1}{(r-\varrho)^{d+2}}\iint_{\Omega_r}v_-^p \dx\dt\quad\forall\,p\in(-\tfrac1\mu,0)\,.
\]
Note that the results of Lemma~\ref{Lem.Moser} hold true for these cylinders as well, with the same constants, since $Q^-_\varrho(1,0)$ can be obtained from $Q_{\varrho}(0,0)$ by means of admissible transformations~\eqref{admissible.transformations}.

The second assumption of Lemma~\ref{BGM.Lemma}, namely inequality~\eqref{hyp.BGM.Lemma.2} of Lemma~\ref{BGM.Lemma}, if stated in terms of super-level sets of $\log v_-$, reads
\[\label{proof.Harnack.7}
\left|\left\{x\in\Omega_1\,:\,\log v_->s\right\}\right|=\left|\{(t,x)\in Q^+_1\,:\,\log v<-s+a\}\right|\le c_2\,|B_1|\,\frac\mu s\,.
\]
and follows from inequality~\eqref{Lem.Log.Est.Ineq.u} of Lemma~\ref{Lem.Log.Est}. With the same~$a$ and $c_2$, we are in the position to apply Lemma~\ref{BGM.Lemma} with $\theta=1/2$ to conclude that~\eqref{Proof.Harnack.claim.2} is true with
\be{proof.Harnack.10}
\kb_0:=\exp\left[2\,c_2\vee c_1^3 2^{2\,(d+2)+3}
\,\sigma\right]\,.
\ee
This concludes the second step.

\stepitem\textit{Harnack inequality and its constant.} We deduce from~\eqref{Proof.Harnack.claim.1} and~\eqref{Proof.Harnack.claim.2} that
\[\label{proof.Harnack.11}
\ka_0^{-\mu}\sup_{D^-}v\,\le \mathrm{e}^{a} \le \kb_0^{\,\mu}\,\inf_{D^+}v
\]
or, equivalently,
\[
\sup_{D^-}v\,\le (\ka_0\,\kb_0)^\mu\,\inf_{D^+}v=\widetilde{\mathsf h}^\mu\,\inf_{D^+}v\,.
\]
Using~\eqref{proof.Harnack.5} and~\eqref{proof.Harnack.10}, we compute
\[\label{HE.Harnack.Constant.proof}\begin{split}
\widetilde{\mathsf h}=\ka_0\,\kb_0&=\exp\left[2\,c_2\vee c_1^3\,2^{2\,(d+2)+3}\,\sigma\right]\,\exp\left[2\,c_2\vee\tfrac{8\,c_1^3\,(\sqrt 2)^{3\,(d+2)}}{(1-1/\sqrt 2)^{2\,(d+2)}}\,\sigma\right]\\
&\le \exp\left[4\,c_2+c_1^3\left( 2^{2\,(d+2)+3}+\tfrac{8\,(\sqrt 2)^{3\,(d+2)}}{(1-1/\sqrt2)^{2\,(d+2)}}\right)\,\sigma\right]\\
&=\exp\left[4\,c_2+c_1^3\,2^{2\,(d+2)+3}\left( 1+\tfrac{2^{ d+2 }}{(\sqrt 2-1)^{2\,(d+2)}}\right)\,\sigma\right]:=\mathsf h\,.
\end{split}
\]
The expressions of $c_1$ and $c_2$ are given in~\eqref{Lem.Moser.constant} and~\eqref{Lem.Log.Est.Ineq.a} respectively. The above expression of $\mathsf h$ agrees with the simplified expression of~\eqref{h}, which completes the proof.
\end{steps}\end{proof}

%%%%%%%%%%%%%%%%%%%%%%%%%%%%%%%%%%%%%%%%%%%%%%%%%%%%%%%%%%%%%%%%%%%%%%
\section{Harnack inequality implies H\"older continuity}\label{sec:holder}

In this section we shall show a standard application of the Harnack inequality~\eqref{harnack}. It is well known that~\eqref{harnack} implies H\"older continuity of solutions to~\eqref{HE.coeff}, as in Moser's celebrated paper~\cite[pp.~108-109]{Moser1964}, here we keep track of all constants and obtain a quantitative expression of the (small) H\"older continuity exponent, which only depends on the Harnack constant, \emph{i.e.}, only depends on the dimension $d$ and on the ellipticity constants $\lambda_0, \lambda_1$~\eqref{HE.coeff.lambdas}.

Let $\Omega_1\subset\Omega_2\subset \R^d$ two bounded domains and let us consider $ Q_1:=(T_2, T_3)\times \Omega_1\subset (T_1, T_4)\times \Omega_2=:Q_2$, where $0\le T_1<T_2<T_3<T<4$. We define the \emph{parabolic distance} between $Q_1$ and $Q_2$ as
\be{parabolic-distance}
d(Q_1, Q_2):=\inf_{\substack{(t, x)\in Q_1 \\ (s, y)\in [T_1, T_4]\times\partial\Omega_2 \cup \{T_1, T_4\}\times \Omega_2}} |x-y|+|t-s|^\frac12\,.
\ee
In what follows, for simplicity, we shall consider $\Omega_1, \Omega_2$ as convex sets, however, this is not necessary and the main result of this section holds without such restriction.

%---------------------------------------------------------------------
\begin{theorem}\label{Claim:4} Let $v$ be a nonnegative solution of~\eqref{HE.coeff} on $Q_2$ and assume that $u$ satisfies~\eqref{HE.coeff.lambdas}.
Then we have
\be{holder-continuity-inequality}
\sup_{(t,x),(s,y)\in Q_1}\frac{|v(t,x)-v(s,y)|}{\big(|x-y|+|t-s|^{1/2}\big)^\nu}\le\,2\(\frac{128}{d(Q_1, Q_2)}\)^\nu\,\|v\|_{\mathrm L^\infty(Q_2)}\,.
\ee
where
\[\label{nu}
\nu:=\log_4\Big(\frac{\overline{\mathsf h}}{\overline{\mathsf h}-1}\Big)\,,
\]
and $\overline{\mathsf h}$ is as in~\eqref{h-bar}.
\end{theorem}
%---------------------------------------------------------------------
From the expression of $\mathsf h$ in~\eqref{h} it is clear that $\overline{\mathsf h}\ge\frac43$, from which we deduce that $\nu\in(0,1)$.
%---------------------------------------------------------------------
\begin{proof}
 We proceed in steps: in step 1 we shall show that inequality~\eqref{harnack} implies a \emph{reduction of oscillation} on cylinders of the form~\eqref{cylinder.harnack}. In step 2 we will iterate such reduction of oscillation and directly show estimate~\eqref{holder-continuity-inequality}.

\begin{steps}
\stepitem\textit{Reduction of oscillation.} Let us define $D_R(t_0,x_0)=(t_0-R^2,t_0+R^2)\times B_{2\,R}(x_0)$ and let $D_R^+(t_0,x_0)$, $D_R^-(t_0,x_0)$ be as in~\eqref{cylinder.harnack}. Let us define
\[\label{max.min}
M:=\max_{D_R(t_0,x_0)}v\,,\quad M^{\pm}=\max_{D_R^\pm(t_0,x_0)}v\,,\quad m=\min_{D_R(t_0,x_0)}v\,,\quad m^{\pm}=\max_{D_R^\pm(t_0,x_0)}v\,,
\]
and let us define the oscillations $\omega$ and $\omega^{+}$ namely
\[\label{oscillation}
\omega=M-m\quad\mbox{and}\quad\omega^{+}=M^+-m^+\,.
\]
We observe that the function $M-u$ and $u-m$ are nonnegative solution to~\eqref{HE.coeff} which also satisfy~\eqref{HE.coeff.lambdas} with $\lambda_0$ and $\lambda_1$ as in~\eqref{HE.coeff.lambdas}. We are therefore in the position to apply inequality~\eqref{harnack} to those functions and get
\[\begin{split}
M-m^{-}=\sup_{D_R^-(t_0,x_0)} M-u\,\le \overline{\mathsf h} \inf_{D_R^+(t_0,x_0)}u-M\,=\overline{\mathsf h}\left(M-M^{+}\right)\,,\\
M^{-}-m=\sup_{D_R^-(t_0,x_0)} u-m\,\le \overline{\mathsf h} \inf_{D_R^+(t_0,x_0)}u-m\,=\overline{\mathsf h}\left(m^{+}-m\right)\,.
\end{split}\]
Summing up the two above inequalities we get
\[
\omega\,\le\,\omega+(M^{-}-m^{-}) \le\,\overline{\mathsf h}\,\omega- \overline{\mathsf h}\,\omega^{+}
\]
which can be rewritten as
\begin{equation}\label{reduction.oscillation}
\omega^{+} \le\frac{\overline{\mathsf h}-1}{\overline{\mathsf h}}\,\omega\,=: \zeta\,\omega\,,
\end{equation}
which means that the oscillation on $D_R^+(t_0,x_0)$ is smaller then the oscillation on $D_R(t_0,x_0)$, recall that $\zeta<1$. In the next step we will iterate such inequality in a sequence of nested cylinders to get a \emph{geometric} reduction of oscillations.
\stepitem\textit{Iteration.} Let us define $\delta=d(Q_1,Q_2)/64$. The number $\delta$ has the following property:
\begin{equation}\label{property.distance}\tag{P}\begin{split}
&\mbox{Let $(t, x)\in Q_1$ and $(s, y)\in(0, \infty)\times\R^d$.}\\
&\mbox{If}\,\,|x-y|+|t-s|^\frac12\le \delta\,\,\mbox{then,}\,\,(s,y)\in Q_2\,.
\end{split}\end{equation}
Let us consider $(t,x), (s,y)\in Q_1$, then either
\begin{equation*}\label{A}\tag{A}
|x-y|+|t-s|^\frac12<\delta\,,
\end{equation*}
or
\begin{equation*}\tag{B}
|x-y|+|t-s|^\frac12 \ge \delta\,.
\end{equation*}
If~\eqref{A} happens, then there exists an integer $k\ge0$ such that
\[
\frac{\delta}{4^{k+1}}\le|x-y|+|t-s|^\frac12\le\frac{\delta}{4^k}\,.
\]
Let us define $z=\frac{x+y}2$ and $\tau_0=\frac{t+s}2$. Since $Q_1$ is a convex set we have that $z, \tau_0\in Q_1$. Let us define,
\[\label{choices.radii.times}
R_{i+1}:=4\,R_i\quad\tau_{i+1}:=\tau_{i}-14\,R_i^2\quad\forall\,i \in\{0, \cdots, k-1\}\,\,\mbox{where}\,\,R_0=\frac{\delta}{4^{k-1}}\,.
\]
With such choices we have that
\be{incapsulation}
D_{R_i}(z, \tau_{i})\subset D_{R_{i+1}}^{+}(z, \tau_{i+1})\quad\forall\,i\in\{0, \cdots, k-1\}\,,
\ee
and
\[
(t, x)\,,(s, y) \in D_{R_0}(z, \tau_0)\subset D_{R_1}^{+}(z, \tau_1)\,.
\]
We also observe that, as a consequence of property~\eqref{property.distance} we have that $D_{R_k}(z, \tau_k)\subset Q_2$. Let us define, for any $i\in\{0, \cdots, k-1\}$
\[
\omega_i:=\max_{D_{R_i}(z, \tau_i)}u - \min_{D_{R_i}(z, \tau_i)}u\quad\mbox{and}\quad\omega_i^{+}:=\max_{D^+_{R_i}(z, \tau_i)}u-\min_{D^+_{R_i}(z, \tau_i)}u\,.
\]
As a consequence of~\eqref{incapsulation}
\be{incapsulation-2}
\omega_i\le\omega^+_{i+1}\,.
\ee
By iterating inequalities~\eqref{incapsulation-2} - \eqref{reduction.oscillation}, we obtain that
\[\begin{split}
|v(t,x)-v(s,y)|\le \omega_0 \le \omega_1^{+}&\le \xi\,\omega_1 \\
& \le\,\xi^k\,\omega_k=\(\frac14\)^{k\,\nu} \omega_k \\
& \le\(\frac4{\delta}\)^\nu\(\frac{\delta}{4^{k+1}}\)^\nu\,\omega_k \\
& \le 2\(\frac4{\delta}\)^\nu\(|x-y|+|t-s|^\frac12\)^\nu\,\|v\|_{\mathrm L^\infty(Q_2)}\,.
\end{split}\]
This concludes the proof of~\eqref{holder-continuity-inequality} under assumption $(A)$.

Let us now assume that $(B)$ happens. In this case we have that
\[\begin{split}
|v(t,x)-v(s,y)|\le 2\,\|u\|_{\mathrm L^\infty(Q_2)}\,\frac{\delta^\nu}{\delta^\nu} &\le 2\,\|v\|_{\mathrm L^\infty(Q_2)}\(\frac{|x-y|+|t-s|^\frac12}{\delta}\)^\nu\\
&\le\,2\(\frac4{\delta}\)^\nu\(|x-y|+|t-s|^\frac12\)^\nu\,\|v\|_{\mathrm L^\infty(Q_2)}\,.
\end{split}\]
\end{steps}
The proof is then completed.
\end{proof}

%%%%%%%%%%%%%%%%%%%%%%%%%%%%%%%%%%%%%%%%%%%%%%%%%%%%%%%%%%%%%%%%%%%%%%
%%%%%%%%%%%%%%%%%%%%%%%%%%%%%%%%%%%%%%%%%%%%%%%%%%%%%%%%%%%%%%%%%%%%%%
%%%%%%%%%%%%%%%%%%%%%%%%%%%%%%%%%%%%%%%%%%%%%%%%%%%%%%%%%%%%%%%%%%%%%%
%%%%%%%%%%%%%%%%%%%%%%%%%%%%%%%%%%%%%%%%%%%%%%%%%%%%%%%%%%%%%%%%%%%%%%
\newpage\part{\Large Constants and estimates:\\[4pt] a handbook, with proofs}\label{PartII}
\setcounter{section}{0}

%%%%%%%%%%%%%%%%%%%%%%%%%%%%%%%%%%%%%%%%%%%%%%%%%%%%%%%%%%%%%%%%%%%%%%
%%%%%%%%%%%%%%%%%%%%%%%%%%%%%%%%%%%%%%%%%%%%%%%%%%%%%%%%%%%%%%%%%%%%%%
\section{Scope of the handbook}\label{Sec:Intro}

This part comes as supplementary material for the computations in~\cite{BDNS2020}. In order to facilitate the reading, the titles of the sections (but not of the sub-sections) are the same as in~\cite{BDNS2020}. However, some results are of independent interest: for this reason, we provide independent statements whenever possible.

%%%%%%%%%%%%%%%%%%%%%%%%%%%%%%%%%%%%%%%%%%%%%%%%%%%%%%%%%%%%%%%%%%%%%%
\subsection{Definitions and notations}\label{Sec:def}

Let us consider the \emph{fast diffusion equation}
\be{FD}
\frac{\partial u}{\partial t}=\Delta u^m\,,\quad u(t=0,\cdot)=u_0
\ee
on $\R^d$ with $d\ge1$ and $m\in(m_1,1)$ with $m_1:=(d-1)/d$.

We introduce the following parameter that will be of constant use in this document
\be{param:alpha}
\alpha=2-d\,(1-m)=1+d\,(m-m_1)=d\,(m-m_c)\,\quad m_c=\frac{d}{d-2}.
\ee

%%%%%%%%%%%%%%%%%%%%%%%%%%%%%%%%%%%%%%%%%%%%%%%%%%%%%%%%%%%%%%%%%%%%%%
\subsection{Outline}\label{Sec:Outline}

In Section~\ref{Sec:FDE}, we provide details on the comparison of the entropy - entropy production inequality with its linearized counterpart, \emph{i.e.}, the Hardy-Poincar\'e inequality: see Proposition~\ref{Prop:iEEP}. Section~\ref{Sec:local.estimates} is devoted to various results on the solutions of the fast diffusion equation~\eqref{FD} which are needed to establish the uniform convergence in relative error.
\begin{enumerate}
\item The local $\mathrm L^1$ bound of Lemma~\ref{HP-Lemma}, known as Herrero-Pierre estimate, is established with explicit constants in Section~\ref{sec.herrero.pierre}.
\item An explicit local upper bound is proved in Lemma~\ref{Lem:LocalSmoothingEffect} in Section~\ref{Sec:LocalUpperBounds}.
\item The Aleksandrov Reflection Principle is applied in Proposition~\ref{Local.Aleks} to prove a first local lower bound in Section~\ref{sec:Aleksandrov-Reflection-Principle}, which is extended in Section~\ref{Sec:Locallowerbounds}: see Lemma~\ref{Posit.Thm.FDE}.
\item Details on the inner estimate in terms of the free energy are collected in Section~\ref{Sec:InnerDetails}: see Proposition~\ref{Proposition11}.
\item In the Appendix~\ref{Appendix}, some useful observations are summarized or detailed: a \emph{user guide for the computation of the threshold time} collects in Appendix~\ref{Appendix:UserGuide} all necessary informations for the computation of the threshold time $t_\star$ of~\cite[Theorem~4]{BDNS2020} and~\cite[Proposition~12]{BDNS2020}; the numerical value of the optimal constant in the Gagliardo-Nirenberg inequality on the disk is the established in Appendix~\ref{Appendix:Numerics}; details on the truncation function are provided in Appendix~\ref{Appendix:Truncation}.
\end{enumerate}

%%%%%%%%%%%%%%%%%%%%%%%%%%%%%%%%%%%%%%%%%%%%%%%%%%%%%%%%%%%%%%%%%%%%%%
%%%%%%%%%%%%%%%%%%%%%%%%%%%%%%%%%%%%%%%%%%%%%%%%%%%%%%%%%%%%%%%%%%%%%%
\section{Relative entropy and fast diffusion flow}\label{Sec:FDE}

Here we deal with \emph{the asymptotic time layer improvement} of~\cite[Section~2.3]{BDNS2020}.

Let us consider the Barenblatt profile
\[
\mB(x)=\(1+|x|^2\)^\frac1{m-1}\quad\forall\,x\in\R^d
\]
of mass $\Mstar:=\ird{\mB(x)}$ and a nonnegative function $v\in\mathrm L^1(\R^d)$ such that $\ird{v(x)}=\Mstar$. The \emph{free energy} (or \emph{relative entropy}) and the \emph{Fisher information} (or \emph{relative entropy production}) are defined respectively by
\[
\mathcal F[v]:=\frac1{m-1}\ird{\(v^m-\mB^m-m\,\mB^{m-1}\,(v-\mB)\)}
\]
and
\[
\mathcal I[v]:=\frac m{1-m}\ird{v\,\left|\nabla v^{m-1}-\nabla\mB^{m-1}\right|^2}\,.
\]
We also define the \emph{linearized free energy} and the \emph{linearized Fisher information} by
\[
\mathsf F[g]:=\frac m2\ird{|g|^2\,\mB^{2-m}}\quad\mbox{and}\quad\mathsf I[g]:=m\,(1-m)\ird{|\nabla g|^2\,\mB}\,,
\]
in such a way that
\be{quadratization}
\mathsf F[g]=\lim_{\varepsilon\to0}\varepsilon^{-2}\,\mathcal F[\mB+\varepsilon\,\mB^{2-m}\,g]\quad\mbox{and}\quad\mathsf I[g]=\lim_{\varepsilon\to0}\varepsilon^{-2}\,\mathcal I[\mB+\varepsilon\,\mB^{2-m}\,g]\,.
\ee
By the \emph{Hardy-Poincar\'e inequality} of~\cite{Blanchet2009}, for any function $g\in\mathrm L^2(\R^d,\mB^{2-m}\,dx)$ such that $\nabla g\in\mathrm L^2(\R^d,\mB\,dx)$ and $\ird{g\,\mB^{2-m}}=0$, if $d\ge1$ and $m\in(m_1,1)$, then we have
\[
\mathsf I[g]\ge4\,\mathsf F[g]\,.
\]
This inequality can be proved by spectral methods as in~\cite{MR1982656,Denzler2005} or obtained as a consequence of the \emph{entropy - entropy production inequality}
\be{EEP}
\mathcal I[v]\ge4\,\mathcal F[v]
\ee
of~\cite{DelPino2002}, using~\eqref{quadratization}. If additionally we assume that $\ird{x\,g\,\mB^{2-m}}=0$, then we have the \emph{improved Hardy-Poincar\'e inequality}
\be{HP-PNAS}
\mathsf I[g]\ge4\,\alpha\,\mathsf F[g]\,.
\ee
where $\alpha=2-d\,(1-m)=d\,(m-m_c)$. Details can be found in~\cite[Lemma~1]{Bonforte2010c} (also see~\cite[Proposition~1]{Dolbeault2011a} and~\cite{Scheffer01,MR1982656,Denzler2005} for related spectral results).
 
Now let us consider
\be{gv}
g:=v\,\mB^{m-2}-\mB^{m-1}
\ee
and notice that $\ird{x\,v(x)}=0$ if and only if $\ird{x\,g\,\mB^{2-m}}=0$. Our goal is to deduce an improved version of~\eqref{EEP} from~\eqref{HP-PNAS}, in a neighborhood of the Barenblatt functions determined by a relative error measured in the uniform convergence norm. We choose the following numerical constant
\[
\chi:=\frac1{322}\quad\mbox{if}\quad d\ge2\,,\quad\chi:=\frac m{266+56\,m}\quad\mbox{if}\quad d=1\,.
\]
In view of~\cite{BDNS2020}, notice that $\chi\ge m/(266+56\,m)$ in any dimension.
%---------------------------------------------------------------------
\begin{proposition}\label{Prop:iEEP} Let $m\in(m_1,1)$ if $d\ge2$, $m\in(1/2,1)$ if $d=1$ and $\eta:=2\,d\,(m-m_1)$. If $v\in\mathrm L^1(\R^d)$ is nonnegative and such that $\ird{v(x)}=\Mstar$, $\ird{x\,v(x)}=0$, and
\be{Uniform}\tag{$H_{\varepsilon,T}$}
(1-\varepsilon)\,\mB\le v\le(1+\varepsilon)\,\mB\quad\mbox{a.e.}
\ee
for some $\varepsilon\in(0,\chi\,\eta)$, then
\be{Ineq:iEEP}
\mathcal I[v]\ge(4+\eta)\,\mathcal F[v]\,.
\ee
\end{proposition}
%---------------------------------------------------------------------
\begin{proof} We estimate the free energy $\mathcal F$ and the Fisher information $\mathcal I$ in terms of their linearized counterparts $\mathsf F$ and $\mathsf I$ as in~\cite{Blanchet2009}. Let $g$ be as in~\eqref{gv}. Under Assumption~\eqref{Uniform}, we deduce by a simple Taylor expansion that
\be{equivalence-entropy-l2}
(1+\varepsilon)^{-a}\,\mathsf F[g]\le\mathcal F[v]\le(1-\varepsilon)^{-a}\,\mathsf F[g]
\ee
as in~\cite[Lemma~3]{Blanchet2009}, where $a=2-m$. Slightly more complicated but still elementary computations based on~\cite[Lemma~7]{Blanchet2009} show that
\be{linear.nonlinear.fisher.inq}
\mathsf I[g]\le s_1(\varepsilon)\,\mathcal I[v]+s_2(\varepsilon)\,\mathsf F[g]\,,
\ee
where
\[
s_1(\varepsilon):=\frac{(1+\varepsilon)^{2\,a}}{1-\varepsilon}\quad\mbox{and}\quad s_2(\varepsilon):=\frac{2\,d}m\,(1-m)^2\(\frac{(1+\varepsilon)^{2\,a}}{(1-\varepsilon)^{2\,a}}-1\).
\]
Collecting~\eqref{HP-PNAS},~\eqref{equivalence-entropy-l2} and~\eqref{linear.nonlinear.fisher.inq}, elementary computations show that~\eqref{Ineq:iEEP} holds with $\eta=f(\varepsilon)$, where
\[
f(\varepsilon)=\frac{4\,\alpha\,(1-\varepsilon)^a\,-4\,s_1(\varepsilon)-(1+\varepsilon)^a\,s_2(\varepsilon)}{s_1(\varepsilon)}\,.
\]
We claim that
\[
\max_{\varepsilon\in(0,\chi\,\eta)}f(\varepsilon)\ge2\,d\,(m-m_1)\,.
\]
Let us consider
\[
g(\varepsilon):=1-\frac{(1-\varepsilon)^{1+a}}{(1+\varepsilon)^{2\,a}}\quad\mbox{and}\quad h(\varepsilon):=\frac{1-\varepsilon}{(1+\varepsilon)^a}\(\frac{(1+\varepsilon)^{2\,a}}{(1-\varepsilon)^{2\,a}}-1\)
\]
and observe that $g$ is concave and $g(\varepsilon)\le g'(0)\,\varepsilon=(1+3\,a)\,\varepsilon\le7\,\varepsilon$ for any $\varepsilon\in[0,1]$ and $a\in[1,2]$, while $h$ is convex and such that $h(\varepsilon)\le h'(1/2)\,\varepsilon$ for any $\varepsilon\in[0,1/2]$ with $h'(1/2)\le133$ for any $a\in[1,2]$. By writing
\[
f(\varepsilon)=2\,\eta-4\,\alpha\,g(\varepsilon)-2\,\frac dm\,(1-m)^2\,h(\varepsilon)\,,
\]
and after observing that $4\,\alpha\le8$ and $\frac dm\,(1-m)^2\le1$ if $d\ge2$ and $m\in(m_1,1)$, $\frac dm\,(1-m)^2\le\frac1m$ if $d=1$ and $m\in(0,1)$, we conclude that
\[
f(\varepsilon)\ge2\,\eta-\tfrac\varepsilon\chi\ge\eta\quad\forall\,\varepsilon\in(0,\chi\,\eta)\,.
\]
\end{proof}

\medskip\noindent Let us conclude this section by some observations:\\[4pt]
$\rhd$ Proposition~\ref{Prop:iEEP} is an \emph{improved entropy - entropy production inequality}. It can be understood as a \emph{stability result} for the standard entropy - entropy production inequality, which is equivalent to the Gagliardo-Nirenberg inequalities
\be{GN}
\nrm{\nabla f}2^\theta\,\nrm f{p+1}^{1-\theta}\ge\mathcal C_{\mathrm{GN}}\,\nrm f{2\,p}\quad\forall\,f\in\mathcal D(\R^d)\,,
\ee
where the exponent is $\theta=\frac{p-1}p\,\frac d{d+2-p\,(d-2)}$, $p$ is in the range $(1,p^*)$ with $p^*=+\infty$ if $d=1$ or $2$, and $p^*=d/(d-2)$ if $d\ge3$, and $\mathcal D(\R^d)$ denotes the set of smooth functions on $\R^d$ with compact support. Similar results with less explicit estimates can be found in~\cite{Blanchet2009,Bonforte2010c}. Compared to~\cite[Theorems~1 and~15]{BDNS2020}, this is a much weaker result in the sense that the admissible neighborhood in which we can state the stability result is somewhat artificial, or at least very restrictive. However, this makes sense in the asymptotic time layer as $t\to+\infty$, from the point of view of the nonlinear flow.\\[4pt]
$\rhd$ According to~\cite{DelPino2002}, it is known that~\eqref{GN} is equivalent to~\eqref{EEP} if $m$ and $p$ are such that
\[
p=\frac1{2\,m-1}\,.
\]
The fact that $p$ is in the interval $(1,p^*)$ is equivalent to $m\in(m_1,1)$ if $d\ge2$ and $m\in(1/2,1)$ if $d=1$. In order to define $\ird{|x|^2\,\mB}$, there is the condition that $m>d/(d+2)$, which is an additional restriction only in dimension $d=1$. This is why in Section~\ref{Sec:local.estimates} we shall only consider the case $m>1/3$ if $d=1$.
\\[4pt]
$\rhd$ In~\cite[Proposition~3]{BDNS2020}, the result is stated for a solution to the \emph{fast diffusion equation} in self-similar variables
\be{FDr}
\frac{\partial v}{\partial t}+\nabla\(v\,\nabla v^{m-1}\)=2\,\nabla\cdot(x\,v)\,,\quad v(t=0,\cdot)=v_0\,.
\ee
With the same assumptions and definitions as in Proposition~\ref{Prop:iEEP}, \emph{if~$v$ is a non-negative solution to~\eqref{FDr} of mass~$\Mstar$,~with
\be{UniformBis}
(1-\varepsilon)\,\mB\le v(t,\cdot)\le(1+\varepsilon)\,\mB\quad\forall\,t\ge T
\ee
for some $\varepsilon\in(0,\chi\,\eta)$ and $T>0$, and such that $\ird{x\,v(t,x)}=0$, then we have
\be{Prop.Rayleigh.Ineq.bis}
\mathcal I[v(t,\cdot)]\ge(4+\eta)\,\mathcal F[v(t,\cdot)]\quad\forall\,t\ge T\,.
\ee}\noindent
This result is equivalent to~Proposition~\ref{Prop:iEEP}.\\[4pt]
$\rhd$ The admissible neighborhood of $\mB$ is in fact stable under the action of the flow defined by~\eqref{FDr}. The improved inequality~\eqref{Prop.Rayleigh.Ineq.bis} holds if~\eqref{UniformBis} holds at $t=T$ and if~$v$ is a non-negative solution to~\eqref{FDr} of mass~$\Mstar$,~with $\ird{x\,v_0(x)}=0$. The condition~\eqref{UniformBis} is very restrictive if we impose it with $T=0$ as in~\cite{Blanchet2009,Bonforte2010c}. A key observation in~\cite{BDNS2020} is that it is satisfied in the asymptotic time layer as $t\to+\infty$ and that we can provide an explicit estimate of $T$.
\\[4pt]
$\rhd$ Condition~\eqref{UniformBis} is slightly different from the one appearing in~\cite{Blanchet2009,Bonforte2010c}. In those papers the initial data is assumed to be such that
\[
\(c_1+|x|^2\)^\frac1{m-1}\le v(0, x) \le\(c_2+|x|^2\)^\frac1{m-1}\quad\forall\,x\in\R^d
\] 
for some positive $c_1$ and $c_2$ such that $0<c_2\le1\le c_1$. The above condition is much stronger than~\eqref{UniformBis} as it guarantees that $(v/\mB-1)\in\mathrm L^q(\R^d)$ for some $q<\infty$. In~\cite{BDNS2020}, we only need that $(v/\mB-1)\in\mathrm L^\infty(\R^d)$.

%%%%%%%%%%%%%%%%%%%%%%%%%%%%%%%%%%%%%%%%%%%%%%%%%%%%%%%%%%%%%%%%%%%%%%
%%%%%%%%%%%%%%%%%%%%%%%%%%%%%%%%%%%%%%%%%%%%%%%%%%%%%%%%%%%%%%%%%%%%%%
\section{Uniform convergence in relative error}\label{Sec:local.estimates}

We state and prove here the local upper and lower bounds that has been used in~\cite[Section 3.2]{BDNS2020}, we provide the explicit constants here, following the proofs of~\cite{Bonforte2006,Bonforte2010b}. Comparing to the existing literature we give simpler proofs and provide explicit constants. Some of the results presented here were already contained in the PhD Thesis of N. Simonov,~\cite[Chapters 1,2 and 6]{Simonov2020}.

In this section we consider solutions to the Cauchy problem for the Fast Diffusion Equation posed in the whole Euclidean space $\RR^d$, in the range $m_1<m<1$, $d\ge1$. Global existence of non-negative solutions of~\eqref{FD} is established in~\cite{Herrero1985}. Much more is known on~\eqref{FD} and we refer to~\cite{Vazquez2006} for a general overview. Recall that we always assume $u_0\in\mathrm L^1(\RR^d)$.

%%%%%%%%%%%%%%%%%%%%%%%%%%%%%%%%%%%%%%%%%%%%%%%%%%%%%%%%%%%%%%%%%%%%%%
\subsection{Mass displacement estimates: local \texorpdfstring{$\mathrm L^1$}{L1} bounds}\label{sec.herrero.pierre}
We prove the Lemma as needed in the proof of~\cite[Theorem~1.1]{Bonforte2006} (slight modification of the original result of Herrero-Pierre~\cite[Lemma 3.1]{Herrero1985}). Our main task is to derive have an explicit expression of the constants appearing the estimate.
%---------------------------------------------------------------------
\begin{lemma} \label{HP-Lemma}
Let $m \in (0, 1)$ and $u(t,x)$ be a nonnegative solution to the Cauchy problem~\eqref{FD}. Then, for any $ t, \tau \ge 0$ and $r,R>0$ such that $\varrho_0 r\ge 2\,R$ for some $\varrho_0>0$, we have
\be{Herrero.Pierre.123}
\int_{B_{2\,R}(x_0)} u(t, x)\dx \le 2^{\frac{m}{1-m}}\,\int_{B_{2\,R+r}(x_0)}{u(\tau,x)\dx}+\cc\,\frac{|t-\tau|^{\frac1{1-m}}}{r^{\frac{2-d\,(1-m)}{1-m}}}\,,
\ee
where 
\be{C3.constant}
\cc:=2^{\frac{m}{1-m}}\,\omega_d\(\frac{16\,(d+1)\,(3+m)}{1-m}\)^{\frac1{1-m}} (\varrho_0+1)\,.
\ee
\end{lemma}
%---------------------------------------------------------------------
\begin{proof} Let $ \phi=\varphi^\beta $, for some $\beta>0$ (sufficiently large, to be chosen later) be a radial cut-off function supported in $ B_{2\,R+r}(x_0)$ and let $\varphi=1 $ in $ B_{2\,R}(x_0) $. We can take, for instance, $\varphi=\varphi_{2\,R, 2\,R+r}$, where $\varphi_{2\,R, 2\,R+r}$ is defined in~\eqref{test.funct}.
We know that, see for instance~\eqref{test.estimates} of Lemma~\ref{lem.test.funct} in Appendix~\ref{Appendix:Truncation} that,
\be{test.estimates.1}
\|\nabla\varphi\|_\infty\le\frac2{r}\quad\mbox{and}\quad \|\Delta\varphi\|_\infty\le\frac{4\,d}{r^2}.
\ee
In what follow we will write $B_{R}$ instead of $B_{R}(x_0)$ when no confusion arises. Let us compute
\be{derivation.L1.norm}\begin{split}
 \left|\frac{d}{dt} \int_{B_{2\,R+r}}{u( t,x ) \phi\left( x\right)\dx}\right|
&=\left| \int_{B_{2 R+r}}\Delta\left( u^m\right) \phi \dx\right|
=\left| \int_{B_{2 R+r}} u^m \Delta\phi \dx\right| \\
& \le \int_{B_{2 R+r}} u^m \big|\Delta\phi\big| \dx\\
&\le \( \int_{B_{2 R+r}} u\,\phi \dx\)^{m} \( \int_{B_{2\,R+r}}\frac{\left|\Delta \phi\right|^{\frac1{1-m}}}{\phi ^{\frac{m}{1-m}}} \dx\)^{1-m} \\
 &\qquad :=C\left( \phi\right) \( \int_{B_{2\,R+r}}{u\,\phi \left( x\right)\dx}\)^{m}\,,
\end{split}\ee
where we have used H\"older's inequality with conjugate exponents $\frac1{m}$ and $\frac1{1-m}$.
We have obtained the following closed differential inequality
\begin{equation*}
 \left|\frac{d}{dt} \int_{B_{2 R+r}}{u\left( t,x\right) \phi\left( x\right)\dx}\right| \le C(\phi) \( \int_{B_{2 R+r}}{u\left(t,x\right) \phi\left( x\right)\dx}\)^m.
\end{equation*}
An integration in time shows that, for all $t$, $\tau\ge 0$, we have
\begin{equation*}
\left( \int_{B_{2 R}}{u\left( t,x\right) \phi\left( x\right)\dx}\right)^{1-m} \le \left( \int_{B_{2 R}}{u\left( \tau,x\right) \phi\left( x\right)\dx}\right)^{1-m}+\left( 1-m\right)C\left( \phi\right) \left| t-\tau\right|.
\end{equation*}
Since $ \phi $ is supported in $ _{2\,R+r} $ and equal to $ 1$ in $ B_{2\,R}$, this implies~\eqref{Herrero.Pierre.123}, indeed, using
\[
(a+b)^{\frac1{1-m}}\le 2^{\frac1{1-m}-1} \left(a^{\frac1{1-m}}+b^{\frac1{1-m}}\right),
\]
we get
\begin{equation*}\begin{split}
 \int_{B_{2 R}} u(t,x)\dx &\le 2^{\frac{m}{1-m}}\( \int_{B_{2 R+r}} u(\tau,x)\dx+\big( (1-m)\,C(\phi)\big)^{\frac1{1-m}} \left| t-\tau\right|^{\frac1{1-m}}\)\\
&\le 2^{\frac{m}{1-m}}\,\int_{B_{2 R+r}} u(\tau,x)\dx+\cc\,\frac{|t-\tau|^{\frac1{1-m}}}{r^{\frac{2-d\,(1-m)}{1-m}}}\,,
\end{split}
\end{equation*}
where
\[
\cc(r):=2^{\frac{m}{1-m}}((1-m)C(\phi))^{\frac1{1-m}}r^{\frac{2-d\,(1-m)}{1-m}}.
\]
The above proof is formal when considering weak or very weak solutions, in which case, it is quite lengthy (although standard) to make it rigorous, cf.~\cite[Proof of Lemma 3.1]{Herrero1985}; indeed, it is enough to consider the time-integrated version of estimates~\eqref{derivation.L1.norm}, and conclude by a Grownwall-type argument.

The proof is completed once we show that the quantity $\cc(r)$ is bounded and provide the expression~\eqref{C3.constant}. Recall that $\phi=\varphi^\beta$, so that
\be{ineq.Delta}\begin{split}
\left|\Delta\left(\phi(x)\right)\right|^{\frac1{1-m}}\phi(x)^{-\frac{m}{1-m}}&=
\varphi(x)^{-\frac{\beta\,m}{1-m}}\left|\beta\,(\beta-1)\,\varphi^{\beta-2}\,
\left|\nabla\,\varphi\,\right|^2+\beta\,\varphi^{\beta-1}\,\Delta\varphi\right|^{\frac1{1-m}}\\
&\le\big(\beta\,(\beta-1)\big)^{\frac1{1-m}}\,\varphi^{\frac{\beta-2-\beta\,m}{1-m}}
\left|\,\left|\nabla\,\varphi\,\right|^2+\left|\Delta\varphi\right|\right|^{\frac1{1-m}}\\
&\le\(\tfrac{4\,(3+m)}{(1-m)^2}\)^{\frac1{1-m}}
\left(\tfrac{4\,(d+1)}{r^2}\right)^{\frac1{1-m}}\,.
\end{split}
\ee
The first inequality follow from the fact that we are considering a radial function $0\le \varphi(x)\le 1$, and we take $\beta=\frac4{1-m}>\frac2{1-m}$. The last one follows by~\eqref{test.estimates.1}. Finally:
\[\begin{split}
&\big( (1-m)\,C(\phi)\big)^{\frac1{1-m}}\,r^{\frac{2-d\,(1-m)}{1-m}}\\
&=(1-m)^{\frac1{1-m}} \( \int_{B_{2\,R+r}\setminus B_{2\,R}}\frac{\left|\Delta \phi\right|^{\frac1{1-m}}}{\phi ^{\frac{m}{1-m}}} \dx\)r^{\frac{2-d\,(1-m)}{1-m}}\\
&\le (1-m)^{\frac1{1-m}} \(\tfrac{4\,(3+m)}{(1-m)^2}\)^{\frac1{1-m}}\left(\tfrac{4\,(d+1)}{r^2}\right)^{\frac1{1-m}}\big|B_{2\,R+r}\setminus B_{2\,R}\big|\,r^{\frac{2-d\,(1-m)}{1-m}}\\
&\qquad=\omega_d\(\tfrac{16\,(d+1)\,(3+m)}{1-m}\)^{\frac1{1-m}}\frac{(2\,R+r)^d-(2\,R)^d}{d\,r^d}\\
&\qquad\le\omega_d\(\tfrac{16\,(d+1)\,(3+m)}{1-m}\)^{\frac1{1-m}} (\varrho_0+1)
\end{split}
\]
where we have used that the support of $\Delta\phi$ is contained in the annulus $B_{2\,R+r}\setminus B_{2\,R}$, inequality~\eqref{ineq.Delta} and in the last step we have used that $ \varrho_0\,r\ge 2\,R$ and
\[
(2\,R+r)^d-(2\,R)^d \le d\,(2\,R+r)^{d-1}\,r \le d\,(\varrho_0+1)\,r^d\,.
\]
The proof is now completed.\end{proof}

%%%%%%%%%%%%%%%%%%%%%%%%%%%%%%%%%%%%%%%%%%%%%%%%%%%%%%%%%%%%%%%%%%%%%%
\subsection{Local upper bounds}\label{Sec:LocalUpperBounds}

%---------------------------------------------------------------------
\begin{lemma}\label{Lem:LocalSmoothingEffect} Assume that $d\ge1$, $m\in(m_1,1)$. If $u$ is a solution of~\eqref{FD} with non-negative initial datum $u_0\in\mathrm L^1(\R^d)$, then there exists a positive constant~$\overline\kappa$ such that any solution $u$ of~\eqref{FD} satisfies for all $(t,R)\in(0,+\infty)^2$ the estimate
\be{BV-1}
\sup_{y\in B_{R/2}(x)}u(t,y)\le\overline{\kappa}\left(\frac1{t^{d/\alpha}}\(\int_{B_R(x)}u_0(y)\,dy\)^{2/\alpha}+\(\frac t{R^2}\)^\frac1{1-m}\right)\,.
\ee
\end{lemma}
%---------------------------------------------------------------------
The above estimate is well known, \emph{cf.}~\cite{DiBenedetto1993, DiBenedetto2012, Daskalopoulos2007,Bonforte2010b}, but the point is that we provide an explicit expression of the constant
\be{kappa}
\overline\kappa=\mathsf k\,\mathcal K^\frac{2\,q}\beta\,
\ee
where $\mathsf k=\mathsf k(m,d,\beta,q)$ is such that
\[
\mathsf k^\beta=\big(\tfrac{4\,\beta}{\beta+2}\big)^\beta\,\big(\tfrac4{\beta+2}\big)^2\,\pi^{\,8\,(q+1)}
\,e^{8\sum_{j=0}^{\infty}\log(j+1)\,\left(\frac q{q+1}\right)^j}\,2^\frac{2\,m}{1-m}\,(1+\mathtt{a}\,\omega_d)^2\,\mathtt{b}
\]
with $\mathtt{a}=\tfrac{3\,(16\,(d+1)\,(3+m))^\frac1{1-m}}{(2-m)\,(1-m)^\frac m{1-m}}+\tfrac{2^\frac{d-m\,(d+1)}{1-m}}{3^d\,d}\quad\mbox{and}\quad \mathtt{b}=\tfrac{38^{2\,(q+1)}}{\big(1-(2/3)^{\frac{\beta}{4\,(q+1)}}\big)^{4\,(q+1)}}$.\\
The constant $\mathcal K$ is the same constant as in~\eqref{sob.step2} and corresponds to the inequality
\be{GNS111}
\|f\|^2_{\mathrm L^{\pc}(B)}\le\mathcal K\(\|\nabla f\|^2_{\mathrm L^2(B)}+\|f\|^2_{\mathrm L^2(B)}\).
\ee
In other words,~\eqref{GNS111} is~\eqref{sob.step2} written for $R=1$. The other parameters are given in Table~\ref{table.k.bar} (see~\cite{BDNS2020} for details on optimality and proofs).
%---------------------------------------------------------------------
\setlength{\extrarowheight}{6pt}
\begin{table}[ht]
\begin{center}
\begin{tabular}{|c||c|c|c|c|}\hline
 ~ & $\pc$ & $\mathcal K$ & $q$ & $\beta$\\[6pt]
\hline\hline
$d\ge 3$ & $\frac{2\,d}{d-2}$ & $\frac2{\pi}\,\Gamma(\frac d2+1)^{2/d}$ & $\frac d2$& $\alpha$ \\[6pt]
\hline
$d=2$ & $4$ & $\frac2{\sqrt\pi}$ & $2$ & $2\,(\alpha-1)$ \\[6pt]
\hline
$d=1$ & $\frac4m$ & $2^{1+\frac m2}\,\max\(\frac{2\,(2-m)}{m\,\pi^2},\frac14\)$ & $\frac2{2-m}$ & $\frac{2\,m}{2-m}$\\[6pt]
\hline
\end{tabular}
\caption{\label{table.k.bar} Table of the parameters and the constant $\mathcal K$ in dimensions $d=1$, $d=2$ and $d\ge3$. The latter case corresponds to the critical Sobolev exponent while the inequality for $d\le2$ is subcritical. In dimension $d=1$, $\pc=4/m$, which makes the link with~\eqref{estim.S-p}.}
\end{center}
\end{table}
%---------------------------------------------------------------------

\begin{proof}[Proof of Lemma~\ref{Lem:LocalSmoothingEffect}] Our proof follows the scheme of~\cite{Bonforte2010b} so we shall only sketch its main steps, keeping track of the explicit expression of the constants. The point $x\in\R^d$ is arbitrary and by translation invariance it is not restrictive to assume that $x=0$ and write $B_R=B_R(0)$. We also recall that $u$ always possesses the regularity needed to perform all computations throughout the following steps.

Let us introduce the rescaled function
\be{realation-u-hatu}
\hat u(t,x)=\left(\frac{R^2}\tau\right)^{\frac1{1-m}}\,u(\tau\,t,R\,x)
\ee
which solves~\eqref{FD} on the cylinder $\left(0, 1\right]\times B_1$. In Steps 1-3 we establish on $\hat v=\max(\hat u,1)$ a $\mathrm L^2-\mathrm L^\infty$ smoothing inequality which we improve to a $\mathrm L^1-\mathrm L^\infty$ smoothing in Step 4, using a {\em de Giorgi}-type iteration. In Step 5, we scale back the estimate to get the result on $u$.

\begin{steps}
\stepitem We observe that $\hat v=\max\{\hat u,1\}$ solves $\frac{\partial \hat v}{\partial t}\le\Delta \hat v^m$. According to~\cite[Lemma~2.5]{Bonforte2010b}, we know that
\[
\sup_{s\in[T_1,T]}\int_{B_{R_1}}\hat v^{p_0}(s,x)\,dx+\iint_{Q_1}\left|\nabla \hat v^\frac{p_0+m-1}2\right|^2\,dx\,dt\le\frac8{c_{m,p_0}}\iint_{Q_0}\(\hat v^{m+p_0-1}+\hat v^{p_0}\)dx\,dt
\]
where $Q_k=(T_k,T]\times B_{R_k}$ with $0<T_0<T_1<T\le 1$, $0<R_1<R_0\le 1$ and $c_{m,p_0}=\min\left\{1-\tfrac1{p_0},\tfrac{2\,(p_0-1)}{p_0+m-1}\right\}\ge\tfrac12$. We have $\hat v^{m+p_0-1}\le \hat v^{p_0}$ because $\hat v\ge 1$, so that
\be{Lemma2.5:BV2010}
\sup_{s\in[T_1,T]}\int_{B_{R_1}}\hat v^{p_0}(s,x)\,dx+\iint_{Q_1}\left|\nabla \hat v^\frac{p_0+m-1}2\right|^2\,dx\,dt\le\mathcal C_0\iint_{Q_0}\hat v^{p_0}\,dx\,dt
\ee
where
\[
\mathcal C_0=32\(\frac1{(R_0-R_1)^2}+\frac1{T_1-T_0}\).
\]

\stepitem Let $\pc$ be as defined in Section~\ref{Sec:def} and $\mathcal K$ be the constant in the inequality~\eqref{sob.step2}. Let $q=\pc/(\pc-2)$ and $Q_i=(T_i,T]\times B_{R_i}$ as in Step 1. We claim that
\be{induction.step}
\iint_{Q_1}\hat v^{p_1}\,dx\,dt\le\mathcal K_0\(\iint_{Q_0}\hat v^{p_0}\,dx\,dt\)^{1+\frac1q}\quad\mbox{with}\quad\mathcal K_0=\mathcal K\(R_1^{-2}+\mathcal C_0\)^{1+\frac1q}\,.
\ee
Let us proven~\eqref{induction.step}. Using H\"older's inequality, for any $a\in(2,\pc)$ we may notice that
\[
\int_{B_{R_1}}|f(s,x)|^a\,dx=\int_{B_{R_1}}|f(s,x)|^2\,|f(s,x)|^{a-2}\,dx\le\|f\|_{\mathrm L^{\pc}(B_{R_1})}^2\,\|f\|_{\mathrm L^b(B_{R_1})}^{a-2}
\]
with $b=q\,(a-2)$. Using~\eqref{sob.step2}, this leads to
\[
\iint_{Q_1}|f(t,x)|^a\,dx\,dt\le\mathcal K\(\|\nabla f\|_{\mathrm L^2(Q_1)}^2+\tfrac1{R_1^2}\,\|f\|_{\mathrm L^2(Q_1)}^2\)\sup_{s\in(T_1,T)}\(\int_{B_{R_1}}|f(s,x)|^b\,dx\)^\frac1q.
\]
Choosing $f^2=\hat v^{p_0+m-1}$ with $a=2\,p_1/(p_0+m-1)$ and $b=2\,p_0/(p_0+m-1)$ we get
\[
\iint_{Q_1}\hat v^{p_1}\,dx\,dt\le\mathcal K\iint_{Q_1}\(\left|\nabla\hat v^\frac{p_0+m-1}2\right|^2+\frac{\hat v^{p_0}}{R_1^2}\)dx\,dt\sup_{s\in(T_1,T)}\(\int_{B_{R_1}}\hat v^{p_0}\,dx\)^\frac1q
\]
where
\[
p_1=\(1+\frac1q\)p_0-1+m>p_0\,.
\]
Letting $X=\nrm{\nabla\hat v^{(p_0+m-1)/2}}2^2$, $Y_i=\iint_{Q_1}\hat v^{p_i}\,dx\,dt$ and $Z=\sup_{s\in(T_1,T)}\int_{B_{R_1}}\hat v^{p_0}\,dx$, we get $Y_1\le\mathcal K\,(X+R_1^{-2}\,Y_0)\,Z^{1/q}$, while~\eqref{Lemma2.5:BV2010} reads $X+Z\le\mathcal C_0\,Y_0$. Hence $Y_1\le\mathcal K\,\big((R_1^{-2}+\mathcal C_0)\,Y_0-Z\big)\,Z^{1/q}\le\,\mathcal K\,\big((R_1^{-2}+\mathcal C_0)\,Y_0\big)^{(q+1)/q}$, that is inequality~\eqref{induction.step}.

\stepitem We perform a Moser-type iteration. In order to iterate~\eqref{induction.step}, fix $R_\infty<R_0<1$, $T_0<T_\infty<1$ and also assume that $2\,R_\infty\ge R_0$. We shall consider the sequences $(p_k)_{k\in\N}$, $(R_k)_{k\in\N}$, $(T_k)_{k\in\N}$ and $(\mathcal K_k)_{k\in\N}$ defined as follows:
\begin{align*}
&p_k=\(1+\frac1q\)^k\(2-q\,(1-m)\)+q\,(1-m)\,,\\
&R_k-R_{k+1}=\frac6{\pi^2}\,\frac{R_0-R_\infty}{(k+1)^2}\,,\quad T_{k+1}-T_k=\frac{90}{\pi^4}\,\frac {T_\infty-T_0}{(k+1)^4}\,,\\
&\mathcal K_k=\mathcal K\(R_{k+1}^{-2}+\mathcal C_k\)^{1+\frac1q}\,,\quad\mathcal C_k=32\(\frac1{(R_k-R_{k+1})^2}+\frac1{T_{k+1}-T_k}\),
\end{align*}
using the Riemann sums $\sum_{k\in\N}(k+1)^{-2}=\frac{\pi^2}6$ and $\sum_{k\in\N}(k+1)^{-4}=\frac{\pi^4}{90}$. It is clear that $\lim\limits_{k\to+\infty}R_k=R_\infty$, $\lim\limits_{k\to+\infty}T_k=T_\infty$ and $\mathcal C_k$ diverge as $k\to+\infty$. In addition, the assumption $2\,R_\infty\ge R_0$ leads to $R_{k+1}^{-2}\le (R_0-R_\infty)^{-2}$ hence $\mathcal K_k$ is explicitly bounded by
\[
\mathcal K_k \le \mathcal K \left(\pi^4\,(k+1)^4 L_{\infty}\right)^{1+\frac1q},\quad\mbox{where}\quad L_{\infty}:=\frac1{(R_0-R_\infty)^2}+\frac1{\(T_\infty-T_0\)}\,.
\]
Set $Q_\infty=(T_\infty, T)\times B_{R_\infty}$ and notice that $Q_\infty\subset Q_k$ for any $k\ge 0$. By iterating~\eqref{induction.step}, we find that
\[
\nrm{\hat v}{\mathrm L^{p_{k+1}}(Q_{\infty})}\le \nrm{\hat v}{\mathrm L^{p_{k+1}}(Q_{k+1})}\le\mathcal K_k^\frac1{p_{k+1}}\,\nrm{\hat v}{\mathrm L^{p_k}(Q_k)}^\frac{(q+1)\,p_k}{q\,p_{k+1}} \le \prod_{j=0}^k \mathcal K_j^{\frac1{p_{k+1}}\left(\frac{q+1}{q}\right)^{k-j}} \nrm{\hat v}{\mathrm L^2(Q_0)}^\frac{2\,(q+1)^{k+1}}{q^{k+1}\,p_{k+1}}\,
\]
and
\[
\prod_{j=0}^k \mathcal K_j^{\frac1{p_{k+1}}\left(\frac{q+1}{q}\right)^{k-j}} \le \left[\mathcal K \left(\pi^4\,L_{\infty}\right)^{1+\frac1q}\right]^{\frac1{p_{k+1}}\sum_{j=0}^k\(\frac{q+1}{q}\)^j}\,\prod_{j=1}^{k+1}\,j^\frac{4\(\frac{q+1}{q}\)^{k+2-j}}{p_{k+1}}.
\]
By lower semicontinuity of the $\mathrm L^\infty$ norm, letting $k\to+\infty$, we obtain
\be{step1to3}
\|\hat v\|_{\mathrm L^{\infty}((T_\infty, T]\times B_{R_\infty})} \le \mathcal C\,\|\hat v\|_{\mathrm L^2((T_0, T]\times B_{R_0})}^{\frac2{2-q\,(1-m)}}
\ee
where $0<T_0<T_\infty<T\le 1$, $1/2<R_\infty<R_0 \le 1$, $R_0 \le 2\,R_\infty$, and
\[
\mathcal C=\mathcal K^\frac{q}{2-q\,(1-m)}\left(\pi^4\,L_{\infty}\right)^\frac{(q+1)}{2-q\,(1-m)}\,e^{\frac{4\,(q+1)}{q\(2-q\,(1-m)\)}\sum_{j=1}^{\infty}\left(\frac{q}{q+1}\right)^j\log j}\,.
\]

\stepitem We show how to improve the $\mathrm L^2-\mathrm L^\infty$ smoothing estimate~\eqref{step1to3} to a $\mathrm L^1-\mathrm L^\infty$ estimate, using a de Giorgi-type iteration. Let us set
\be{beta}
\beta=2-2\,q\,(1-m)\,=\begin{cases}\begin{array}{ll}
\alpha\quad&\mbox{if}\quad d\ge 3\,,\\
2\,(\alpha-1)\quad&\mbox{if}\quad d=2\,,\\
\frac{2\,m}{2-m}\quad&\mbox{if}\quad d=1\,,
\end{array}\end{cases}
\ee
we recall that $\beta>0$ for any $m\in(m_1, 1)$ and $d\ge1$. Then, from~\eqref{step1to3}, we obtain, using H\"older's and Young's inequalities,
\begin{multline}\label{iteration.step4}
\|\hat v\|_{\mathrm L^{\infty}((1/9, 1]\times B_{1/2})} \le \mathcal C\,\|\hat v\|_{\mathrm L^{\infty}((\tau_1, 1]\times B_{r_1})}^{\frac1{2-q\,(1-m)}}\,\|\hat v\|_{\mathrm L^1((\tau_1, 1]\times B_{r_1})}^{\frac1{2-q\,(1-m)}} \\
\le\frac12\,\|\hat v\|_{\mathrm L^\infty((\tau_1, 1]\times B_{r_1})}+\mathfrak C_1\,\|\hat v\|_{\mathrm L^1((\tau_1, 1]\times B_{r_1}))}^{\frac2\beta}
\end{multline}
where $1/9<\tau_1<1$, $1/2<r_1<1$ and
\[
\mathfrak C_1=X\left(\frac1{\(r_1-\frac12\)^2}+\frac1{\frac19-\tau_1}\right)^{\frac{2\,(q+1)}{\beta}}
\]
with
\[
X=\tfrac\beta{\beta+2}\,\big(\tfrac4{\beta+2}\big)^\frac2{\beta}\,\mathcal K^\frac{2\,q}{\beta}\(\pi^q\,e^{\sum_{j=1}^{\infty}\left(\frac q{q+1}\right)^j\log j}\)^\frac{8\,(q+1)}{q\,\beta}
\,.
\]
%----------------------------------------------
To iterate~\eqref{iteration.step4} we shall consider sequences $(r_i)_{i \in\N}, (\tau_i)_{i\in\N}$ such that
\[
r_{i+1}-r_i=\tfrac16\,(1-\xi)\,\xi^i\,,\quad\tau_i-\tau_{i+1}=\tfrac19\,(1-\xi^2)\,\xi^{2i}\,.
\]
with $\xi=(2/3)^{\frac{\beta}{4\,(q+1)}}$. Since $2/3\le\xi\le 1$, we have
\[
\frac1{1-\xi^2}\le\frac1{5\,(1-\xi)^2}\,,
\]
and this iteration gives us
\[
\|\hat v\|_{\mathrm L^{\infty}((1/9, 1]\times B_{1/2})} \le\frac1{2^k}\,\|\hat v\|_{\mathrm L^\infty((\tau_k, 1]\times B_{r_k})}+\|\hat v\|_{\mathrm L^1((\tau_k, 1]\times B_{r_k})}^{\frac2\beta}\,\sum_{i=0}^{k-1}\frac{\mathfrak C_{i+1}}{2^i}
\]
where for all $i\ge 0$
\[
\frac{\mathfrak C_{i+1}}{2^i}\le\(\tfrac{38}{(1-\xi)^2}\)^\frac{2\,(q+1)}\beta X\left(\tfrac34\right)^i\,.
\]
In the limit $k\rightarrow\infty$ we find
\be{step4}
\|\hat v\|_{\mathrm L^{\infty}((1/9, 1]\times B_{1/2})} \le \mathfrak C\,\|\hat v\|_{\mathrm L^1((0, 1]\times B_{2/3})}^{\frac2\beta}
\ee
where
\be{goth.C}
\mathfrak C=4\(\tfrac{38}{(1-\xi)^2}\)^\frac{2\,(q+1)}\beta X\,.
\ee
%------------------------------------------------
\stepitem In this step we complete the proof of~\eqref{BV-1}. We recall that $\hat v=\max\{\hat u, 1\}$ and then, using inequality~\eqref{step4} and the fact that $\hat u\le\hat v\le \hat ud\,(1-m)+1$, we find
\be{inequality.notrescaled}
\sup_{y\in B_{1/2}}\hat u(1,y)\le\|\hat u\|_{\mathrm L^{\infty}\((1/9, 1]\times B_{1/2})\)} \le \mathfrak C\,\|\hat u+1\|_{\mathrm L^1\((0, 1]\times B_{2/3}\)}^{\frac2\beta}\,.
\ee
The function $\hat u$ satisfies the following inequality for any $s \in\left[0,1\right]$
\begin{equation}\label{herrero.pierre.2}
\int_{B_{2/3}}\hat u(s,x)\,dx \le 2^{\frac{m}{1-m}} \int_{B_1}\hat u_0\,dx+\mathscr C\,s^{\frac1{1-m}}\,,
\end{equation}
where
\be{calligrafic.C}
\mathscr C=2^{\frac{m}{1-m}}\left(3\omega_d\left[\frac{16\,(d+1)\,(3+m)}{1-m}\right]^{\frac1{1-m}}\right)\,.
\ee
We recall that $\omega_d=|\mathbb S^{d-1}|=\frac{2\,\pi^{d/2}}{\Gamma(d/2)}$. Inequality~\eqref{herrero.pierre.2} is obtained by aplying Lemma~\ref{HP-Lemma} with $R=1/3$, $r=1/3$ and $\rho=2$. Integrating inequality~\eqref{herrero.pierre.2} over $\left[0,1\right]$ we find
\be{HP.hat}
\|\hat u\|_{\mathrm L^1((0, 1]\times B_{2/3})} \le 2^\frac{m}{1-m}\int_{B_1}\hat u_0\,dx+\tfrac{1-m}{2-m}\,\mathscr C\,.
\ee
We deduce from inequalities~\eqref{inequality.notrescaled}-\eqref{HP.hat} that
\be{inequality.notrescaled-12}
\sup_{y\in B_{1/2}(x)}\hat u(1,y) \le \mathfrak C\,\left[2^\frac{m}{1-m} \left( \int_{B_1}\hat u_0\,dx\right)
+\tfrac{1-m}{2-m}\,\mathscr C+\left(\tfrac23\right)^d\,\frac{\omega_d}d\right]^{\frac2\beta}\,.
\ee
where $\beta$ is as in~\eqref{beta}. Let us define
\[\label{kd12}
\overline{\kappa}:=\mathfrak C\,\left[2^\frac{m}{1-m}+\tfrac{1-m}{2-m}\,\mathscr C+\left(\tfrac23\right)^d\,\frac{\omega_d}d\right]^{\frac2\beta}\,,
\]
with $\mathfrak C$ given in~\eqref{goth.C} and $\mathscr C$ in~\eqref{calligrafic.C}. We first prove inequality~\eqref{BV-1} assuming
\[
\tau \ge \tau_\star:=R^\alpha \|u_0\|_{\mathrm L^1(B_R)}^{1-m}\,,
\]
which, by~\eqref{realation-u-hatu}, is equivalent to the assumption $\|\hat u_0\|_{\mathrm L^1(B_1)}\le 1$.
Indeed, together with~\eqref{inequality.notrescaled-12}, we get
\be{inequality.rescaled-12}
\sup_{y\in B_{R/2}}u(\tau,y)\le \overline{\kappa}\(\frac\tau{R^2}\)^\frac1{1-m}\le\overline{\kappa}\left(\frac1{\tau^{\frac d{\alpha}}}\| u_0\|_{\mathrm L^1(B_R)}^\frac 2\alpha+\(\frac\tau{R^2}\)^\frac1{1-m}\right)\,,
\ee
which is exactly~\eqref{BV-1}. Now, for any $0<t\le\tau_\star$, we use the time monotonicity estimate
\[u(\tau)\le u(\tau_\star) \left(\frac{\tau_\star}\tau\right)^{\frac d{\alpha}}
\]
obtained by integrating in time the estimate $u_t\ge -\(d/\alpha\)(u/t)$ of Aronson and Benilan (see~\cite{MR524760}). Combined with the estimate~\eqref{inequality.rescaled-12} at time $\tau_\star$, this leads to
\[\begin{split}
\sup_{y\in B_{R/2}}u(\tau,y) & \le \sup_{y\in B_{R/2}}u(\tau_\star,y) \left(\frac{\tau_\star}\tau\right)^\frac d\alpha\le \overline{\kappa}\(\frac{\tau_\star}{R^2}\)^\frac1{1-m}\left(\frac{\tau_\star}\tau\right)^\frac d\alpha\\
&=\overline{\kappa}\frac{\|u_0\|_{\mathrm L^1(B_R)}^\frac2\alpha}{\tau^\frac d{\alpha}}\le\overline{\kappa}\left(\frac1{\tau^{\frac d{\alpha}}}\| u_0\|_{\mathrm L^1(B_R)}^\frac 2\alpha+\(\frac\tau{R^2}\)^\frac1{1-m}\right)\,
\end{split}
\]
and concludes the proof.
\end{steps}
\end{proof}

%%%%%%%%%%%%%%%%%%%%%%%%%%%%%%%%%%%%%%%%%%%%%%%%%%%%%%%%%%%%%%%%%%%%%%
%%%%%%%%%%%%%%%%%%%%%%%%%%%%%%%%%%%%%%%%%%%%%%%%%%%%%%%%%%%%%%%%%%%%%%
\subsection{A comparison result based on Aleksandrov's Reflection Principle}\label{sec:Aleksandrov-Reflection-Principle}

In this section we are going to prove the Aleksandrov's Reflection Principle, which will be a key tool to prove the lower bounds of Lemma~\ref{Posit.Thm.FDE}. This proof borrows some ideas from the proof of the Aleksandrov's Reflection Principle found in~\cite{Galaktionov2004}.
%---------------------------------------------------------------------
\begin{proposition}\label{Local.Aleks}
\noindent Let $B_{\lambda R}(x_0)\subset\RR^d$ be an open ball with center in $x_0\in\RR^d$ of radius $\lambda\,R$ with $R>0$ and $\lambda>2$. Let $u$ be a solution to problem
\be{FDE.Problem.Aleks}
\begin{split}
\left\{\begin{array}{lll}
u_t=\Delta (u^m) & ~\quad {\rm in}~ (0,+\infty)\times \RR^d\,,\\
u(0,x)=u_0(x) & ~\quad x\in\R^d\,.\\
\end{array}\right.
\end{split}
\ee
with $\supp(u_0)\subset B_{R}(x_0)$. Then, one has:
\be{inequality-0}
u(t,x_0)\ge u(t,x)
\ee
for any $t>0$ and for any $x\in D_{\lambda, R}(x_0)=B_{\lambda
R}(x_0)\setminus B_{2\,R}(x_0)$. Hence,
\be{Aleks.Mean}
u(t,x_0)\ge \left|D_{\lambda, R}(x_0)\right|^{-1}\int_{D_{\lambda,
R}(x_0)}u(t,x)\dx\,.
\ee
\end{proposition}
%---------------------------------------------------------------------
We use the mean value inequality (\ref{Aleks.Mean}) in following form:
\begin{equation} \label{Aleks.Mean.r}
\int_{B_{2\,R+r}(x_0)\setminus B_{2^b R}(x_0)}u(t,x)\dx\le A_d\,r^d\,u(t,x_0)\,,
\ee
with $b=2-(1/d)$, $r>2\,R(2^{1-\frac1{d}}-1)=:r_0$ and a suitable positive constant $A_{d}$. This inequality can easily be obtained from~\eqref{Aleks.Mean}. Let us first assume $d\ge2$, note that in this case $b-1\ge1/2 $ and therefore $r\ge 2\,R\(\sqrt2-1\)$. By Taylor expansion we obtain that for some $\xi\in\(r_0, r\)$ that
\[\begin{split}
\left|B_{2\,R+r}(x_0)\setminus
B_{2^bR}(x_0)\right|&=\frac{\omega_d}{d}\left[(2\,R+r)^d-2^{bd}\,R^d\right]=\omega_d\,(2\,R+\xi)^{d-1}\(r-r_0\)\\
& \le\,\omega_d\,(2\,R+\xi)^{d-1}\,r \le\,\omega_d\,r^d\(\frac{\sqrt2}{\sqrt2-1}\)^{d-1}\,,
\end{split}\]
a simple computation shows that $\sqrt2/\(\sqrt2-1\)\approx3.4142135\le4 $. In the case $d=1$ we have that $b=1$ and thefore
\[
\left|B_{2\,R+r}(x_0)\setminus
B_{2\,R}(x_0)\right|=\omega_1\,r\,.
\]

In conclusion we obtain that for $r\ge2\,R\,(2^{1-\frac1{d}}-1)$ we have
\be{AD}
\left|B_{2\,R+r}(x_0)\setminus
B_{2^bR}(x_0)\right| \le A_d\,r^d\quad\mbox{where}\quad A_d:=\omega_d\,4^{d-1}\,.
\ee

\begin{proof}
Without loss of generality we may assume that $x_0=0$ and write $B_R$ instead of $B_R(0)$. Let us recall that the support of $u_0$ is contained in $B_R$. Let us consider an hyperplane $\Pi$ of equation $\Pi=\{x\in\RR^d~|~ x_1=a\}$ with $a\ge R>0$, in this way $\Pi$ is tangent to the the sphere of radius $a$ centered in the origin. Let us as well define $\Pi_+=\{x\in\RR^d~|~ x_1>a\}$ and $\Pi_-=\{x\in\RR^d~|~x_1<a\}$, and the reflection $\sigma(z)=\sigma(z_1,z_2,\ldots,z_n)=(2a-z_1,z_2,\ldots,z_n)$. By these definitions we have that $\sigma(\Pi_+)=\Pi_-$ and $\sigma(\Pi_-)=\Pi_+$.
Let us denote $Q=(0, \infty)\times \Pi_-$ and the parabolic boundary $\partial_pQ:=\partial Q$. We now consider the \emph{Boundary Value Problem} (BVP) defined as
\be{FDE.Problem.BVP}\tag{BVP}
\begin{split}
\left\{\begin{array}{lll}
u_t=\Delta (u^m) & ~ {\rm in}~ Q,\\
u(t,x)=g(t,x) & ~{\rm in}~ \partial_p Q,\\
\end{array}\right.
\end{split}
\ee
for some (eventually continuous) function $g(t,x)$. Let us define $u_1(t,x)$ to be the restriction of $u(t,x)$ to $Q$ and $u_2(t,x)=u_1(t, \sigma(x))$. We recall that $u_2(t,x)$ is still a solution to problem~\eqref{FDE.Problem.Aleks}. Also, both $u_1(t,x)$ and $u(t,x)$ are solutions to~\eqref{FDE.Problem.BVP} with boundary values $g_1(t,x)$ and $g_2(t,x)$. Furthermore, for any $t>0$ and for any $x\in\Pi$ we have that $g_1(t,x)=g_2(t,x)$, as well $g_1(t,x)=u_0\ge g_2(t,x)=0$ for any $x\in\Pi_-$. By comparison principle we obtain for any $(t,x)\in Q$
\be{1-inequality}
u_1(t,x)\ge u_2(t,x)\,.
\ee
The comparison principle for generic boundary value problems is classical in the literature, however we were not able to find the exact reference for a version on a hyperplane. We refer to the books~\cite{Daskalopoulos2007,Vazquez2007,Vazquez2006,Galaktionov2004}, see also~\cite[Lemma 3.4]{Herrero1985} for a very similar comparison principle, and also~\cite[Remark 1.5]{MR712265} for a general remark about such principles.

Inequality~\eqref{1-inequality} implies for any $t>0$ that
\[
u(t,0)\ge u(t,(2a, \dots, 0)).
\]
By moving $a$ in the range $(R,\lambda R/2)$ we find that $u(t,0)\ge u(t,x)$ for any $x\in D_{\lambda,R}$ such that $x=(x_1, 0, \dots, 0)$. It is clear that by rotating the hyperplane $\Pi$ we can generalize the above argument and obtain inequality~\eqref{inequality-0}. Lastly, we observe that inequality~\eqref{Aleks.Mean} can be easily deduced by averaging inequality~\eqref{inequality-0}. The proof is complete. \end{proof}

%%%%%%%%%%%%%%%%%%%%%%%%%%%%%%%%%%%%%%%%%%%%%%%%%%%%%%%%%%%%%%%%%%%%%%
\subsection{Local lower bounds}\label{Sec:Locallowerbounds}

We recall Lemma~\cite[Lemma~6]{BDNS2020} which follows from~\cite[Theorem~1.1]{Bonforte2006}.
%---------------------------------------------------------------------
\begin{lemma}[test]\label{Posit.Thm.FDE}
Let $u(t,x)$ be a solution to~\eqref{FD} and let $R>0$ such that $M_R(x_0):=\|u_0\|_{\mathrm L^1(B_R(x_0))}>0$. Then the inequality
\be{BV-3}
\inf_{|x-x_0|\le R}u(t,x)\ge\kappa\left(R^{-2}\,t\right)^\frac1{1-m}\quad\forall\,t\in[0,2\,\underline t]
\ee
holds with
\[\label{BV-3t1a}
\underline t=\tfrac12\,\kappa_\star\,M_R^{1-m}(x_0)\,R^{\alpha}\,.
\]
\end{lemma}
%---------------------------------------------------------------------
This estimate is based on the results of Sections~\ref{sec.herrero.pierre} and~\ref{sec:Aleksandrov-Reflection-Principle}. Our contribution here is to establish that the constants are
\be{kappaExpr-kappastarExpr}
\kappa_\star=2^{\,3\,\alpha+2}\,d^{\,\alpha}\quad\mbox{and}\quad\kappa=\alpha\,\omega_d\(\frac{(1-m)^4}{2^{38}\,d^{\,4}\,\pi^{16\,(1-m)\,\alpha}\,\overline\kappa^{\,\alpha^2\,(1-m)}}\)^\frac2{(1-m)^2\,\alpha\,d}\,.
\ee

\begin{proof}\begin{steps} Without loss of generality we assume that
$x_0=0$. The proof is a combination of several steps. Different positive constants that depend on $m$ and $d$ are denoted by $C_i$.

\stepitem {\sl Reduction.} By comparison we may assume $\supp(u_0)\subset B_{R}(0)$. Indeed, a general $u_0\ge 0$ is greater than $u_0\chi_{B_R}$, $\chi_{B_R}$ being the characteristic function of $B_{R}$. If $v$ is the solution of the fast diffusion equation with initial data $u_0\chi_{B_R}$ (existence and uniqueness are well known in this case), then we obtain by comparison:
\[
\inf_{x\in B_{R}}u(t,x)\ge \inf_{x\in B_{R}}v(t,x)\,.
\]

\stepitem{\sl A priori estimates.} The so called smoothing effect (see e.g.~\cite[Theorem~2.2]{Herrero1985} , or~\cite{Vazquez2006}) asserts that for any $t>0$ and $x \in\RR^d$ we have:
\be{est.hp}
u(t,x)\le \overline{\kappa}\,\frac{\|u_0\|_1^\frac2\alpha}{t^\frac d\alpha}\,.
\ee
where $\alpha=2-d\,(1-m)$. We remark that~\eqref{est.hp} can be deduced from inequality~\eqref{BV-1} of Lemma~\ref{Lem:LocalSmoothingEffect} by simply takin the limit $R\rightarrow \infty$. The explicit expression of the constant $\overline{\kappa}$ is given in~\eqref{kappa}. We remark that $\|u_0\|_1=M_{R}$ since $u_0$ is nonnegative and supported in $B_{R}$, so that we get $ u(t,x)\le \overline{\kappa} M_{R}^\frac2\alpha\,t^{-\frac d\alpha}$. Let $b=2-1/d$, an integration over $B_{2^bR}$ gives then:
\be{HP.1.Apriori}
\int_{B_{2^bR}}u(t,x)\dx\le \overline{\kappa}\,\frac{\omega_d}{d}\,\frac{M_R^\frac2\alpha}{t^\frac d\alpha} \left(2^b\,R\right)^d\le C_2\,\frac{M_R^\frac2\alpha}{t^\frac d\alpha}\,R^d\,,
\ee
where $C_2$ can be chosen as
\be{C2}
C_2:=2^d\,\max\Big\{1,\overline{\kappa}\,\frac{\omega_d}{d} \Big\}\,.
\ee

\stepitem {\sl Aleksandrov Principle}. In this step we use the so-called Aleksandrov Reflection Principle, see Proposition~\ref{Local.Aleks} in section~\ref{sec:Aleksandrov-Reflection-Principle} for its proof. This principle reads:
\be{Posit.Alex}
\int_{B_{2\,R+r}\setminus B_{2^bR}}u(t,x)\dx\le A_{d}\,r^d u(t,0)
\ee
where $A_d$ is as in~\eqref{AD} and $b=2-1/d$. One has to remember of the condition
\be{r-condition}
r\,\ge\,(2^{(d-1)/d}-1)\,2\,R.
\ee
We refer to Proposition~\ref{Local.Aleks} and formula {\rm (\ref{Aleks.Mean.r})} in section~\ref{sec:Aleksandrov-Reflection-Principle} for more details.

\stepitem{\sl Integral estimate.} Thanks to Lemma~\ref{HP-Lemma}, for any $R,r>0$ and $s,t\ge 0$ one has
\[
\int_{B_{2\,R}}u(s,x)\dx\le C_3
\left[\int_{B_{2\,R+r}}u(t,x)\dx+\frac{|s-t|^{1/(1-m)}}{r^{(2-d\,(1-m))/(1-m)}}\right],
\]
where the constant $C_3$ has to satisfy $C_3\ge\max(1,\cc)$ and $\cc$ is defined in~\eqref{C3.constant}. In what follows we prefer to take a larger constant (for reasons that will be clarified later) and put
\[\label{C_3}
C_3=\left(\frac{16}{1-m}\right)^{\frac1{1-m}}\max\left(1,2\,\omega_d\,\left[\frac{16\,(d+1)\,(3+m)}{1-m}\right]^{\frac1{1-m}}\right).
\]
We let $s=0$ and rewrite it in a form more useful for our purposes:
\be{Quasi.Mass.Posit}
\int_{B_{2\,R+r}}u(t,x)\dx\ge\frac{M_{R}}{C_3}
-\frac{t^{\frac1{1-m}}}{r^{\frac{\alpha}{1-m}}}\,.
\ee
We recall that $M_{2\,R}=M_R$ since $u_0$ is nonnegative and supported in $B_{R}$.

\stepitem We now put together all previous calculations:
\[\begin{split}
\int_{B_{2\,R+r}}u(t,x)\dx&=\int_{B_{2\,R}}u(t,x)\dx+\int_{B_{2\,R+r}\setminus
B_{2^bR}}u(t,x)\dx\\
&\le C_2\,\frac{M_{R}^\frac2\alpha\,R^d}{t^\frac d\alpha}+A_d\,r^d\,u(t,0)\,.
\end{split}\]
This follows from (\ref{HP.1.Apriori}) and (\ref{Posit.Alex}). Next, we use (\ref{Quasi.Mass.Posit}) to obtain:
\[
\frac{M_{R}}{C_3}
-\frac{t^{\frac1{1-m}}}{r^{\frac{\alpha}{1-m}}}\le\int_{B_{2\,R+r}}u(t,x)\dx\le
C_2\,\frac{M_{R}^\frac2\alpha\,R^d}{t^\frac d\alpha}+A_d\,r^d u(t,0)\,.
\]
Finally we obtain
\[\label{Posit.Non.Opt}
u(t,0) \ge\frac1{A_d}\left[\left(\frac{M_{R}}{C_3}-C_2\,\frac{M_{R}^\frac2\alpha\,R^d}{t^\frac d\alpha}
\right)\frac1{r^d}-\frac{t^{\frac1{1-m}}}{r^{\frac2{1-m}}}\right]=\frac1{A_d}\left[\frac{B(t)}{r^d}-\frac{t^{\frac1{1-m}}}{r^{\frac2{1-m}}}\right]\,.
\]

\stepitem The function $B(t)$ is positive when
\[
B(t)=\frac{M_{R}}{C_3}-C_2\,\frac{M_{R}^\frac2\alpha\,R^d}{t^\frac d\alpha}>0
\Longleftrightarrow t>\left(C_3\,C_2\right)^{\frac{\alpha}{d}}\,.
M_{R}^{1-m} R^\alpha
\]
Let us define
\be{tilde-kappa-star}
\tilde{\kappa}_\star:=4\,\left(C_3\,C_2\right)^\frac{\alpha}{d}\quad\mbox{and}\quad \tilde{\underline{t}}=\tfrac12\,\tilde{\kappa}_\star\,M_R^{1-m}\,R^{\alpha}\,.
\ee
We assume that $t\ge 2\,\tilde{\underline{t}}$ and optimize the function
\[
f(r)=\frac1{A_d}\left[\frac{B(t)}{r^d}-\frac{t^{\frac1{1-m}}}{r^{\frac2{1-m}}}\right]
\]
with respect to $r(t)=r>0$. The function $f$ reaches its maximum at $r=r_{max}(t)$ given by
\[
r_{max}(t)=\left(\frac2{d\,(1-m)}\right)^\frac{1-m}{\alpha}\,\frac{t^\frac1{\alpha}}{B(t)^\frac{1-m}{\alpha}}\,.
\]
We recall that we have to verify that $r_{max}$ satisfies condition~\eqref{r-condition}, namely that $r_{max}(t)>\left(2^{(d-1)/d}-1\right)2\,R$. To check this we optimize in $t$ the function $r_{max}(t)$ with respect to $t\in(2\tilde{\underline{t}},+\infty)$. The minimum of $r_{max}(t)$ is attained at a time $t=t_{min}$ given by
\[
t_{min}=\left(\frac2\alpha\,C_2\,C_3\right)^\frac{\alpha}{d}\,M_R^{1-m}\,R^\alpha\,.
\]
We compute $r_{max}(t_{min})$ and find that
\[
r_{max}(t_{min})=\left(\frac2{d\,(1-m)}\right)^\frac{2\,(1-m)}{\alpha}\,\left(\frac2\alpha\,C_2\right)^\frac1{d}\,C_3^\frac2{d\alpha}\,R\,.
\]
Therefore the condition $r_{max}(t_{min})>\left(2^{(d-1)/d}-1\right)2\,R$ is nothing more than a lower bound on the constants $C_2$ and $C_3$, namely that
\[
\left(\frac2{d\,(1-m)}\right)^\frac{2\,(1-m)}{\alpha}\,\left(\frac2\alpha\,C_2\right)^\frac1{d}\,C_3^\frac2{d\alpha}\,\ge 2^{(d-1)/d}-1\,.
\]
Such a lower bound is easily verified, by using the fact $m\in(m_1,1)$, we have $(1-m)^{-1}>d$ and therefore we have the following inequalities
\be{first-estimates}
\frac2{d\,(1-m)}\ge2\,,\quad\frac2\alpha=\frac2{2-d\,(1-m)}\ge1\,,\quad C_2\ge2^d\quad\mbox{and}\quad C_3\ge 16^d\,d^d\,,
\ee
therefore, from the above inequalities we find that
\[
\left(\frac2{d\,(1-m)}\right)^\frac{2\,(1-m)}{\alpha}\,\left(\frac2\alpha\,C_2\right)^\frac1{d}\,C_3^\frac2{d\alpha} \ge 32\,d\ge2^{(d-1)/d}-1\,,
\]
and so the such a lower bound is verified. Let us now continue with the proof.
\stepitem After a few straightforward computations, we show that the maximum value is attained for all $t>2\,\tilde{\underline{t}}$ as follows:
\[
f(r_{max})=\alpha\,A_d\frac{\left[
d\,(1-m)\right]^{\frac{d\,(1-m)}{\alpha}}}{2^\frac 2\alpha}
\left[\frac1{C_3}-C_2\,\frac{M_{R}^{\frac{d\,(1-m)}{\alpha}}R^d}{t^\frac d\alpha}\right]^\frac 2\alpha
\frac{M_{R}^\frac 2\alpha}{t^\frac d\alpha}>0\,.
\]
We get in this way the estimate:
\[
\begin{split}
u(t,0) &\ge K_1\,H_1(t)\,\frac{M_R^\frac 2\alpha}{t^\frac d\alpha}\,,
\end{split}
\]
where
\[
H_1(t)=\left[\frac1{C_3}-C_2\,\frac{M_{R}^{\frac{d\,(1-m)}{\alpha}}R^d}{t^\frac d\alpha}\right]^\frac 2\alpha\quad\mbox{and}\quad K_1=\alpha\,A_d\frac{\left[
d\,(1-m)\right]^{\frac{d\,(1-m)}{\alpha}}}{2^\frac 2\alpha}.
\]
A straightforward calculation shows that the function is non-decreasing in time, thus if $t\ge 2\,\tilde{\underline{t}}$:
\[
H_1(t)\ge H_1(2\,\tilde{\underline{t}})=C_3^{-\frac2\alpha}\,\left(1-4^{-\frac d\alpha}\right)^\frac2\alpha\,,
\]
and finally we obtain for $t\ge 2\,\tilde{\underline{t}}$ that
\be{Posit.Center}
u(t,0) \ge\,K_1\,C_3^{-\frac2\alpha}\,\left(1-4^{-\frac d\alpha}\right)^\frac2\alpha\,\frac{M_{R}^\frac 2\alpha}{t^\frac d\alpha}=\tilde{\underline{\kappa}}\,\frac{M_{R}^\frac 2\alpha}{t^\frac d\alpha}\,.
\ee

\stepitem {\sl From the center to the infimum.} Now we want to obtain a positivity estimate for the infimum of the solution $u$ in the ball $B_R=B_R(0)$. Suppose that the infimum is attained in some point $x_m\in\overline{B_R}$, so that $\inf_{x\in B_R}u(t,x)=u(t,x_m)$, then one can apply {\rm(\ref{Posit.Center})} to this point and obtain:
\[\label{Posit.Min}
u(t,x_m) \ge
\tilde{\underline{\kappa}}\,\frac{M_{2\,R}(x_m)^\frac 2\alpha}{t^\frac d\alpha}
\]
for $t>\tilde{\kappa}_\star M_{R}^{1-m}(x_m) R^\alpha$. Since the point $x_m\in\overline{B_R(0)}$ then it is clear that $B_R(0)\subset B_{2\,R}(x_m)\subset B_{4R}(x_0)$, and this leads to the inequality:
\[
M_{2\,R}(x_m)\ge M_R(0)\quad{\rm and}\quad M_{2\,R}(x_m)\le M_{4R}(0)
\]
since $M_{\varrho}(y)=\int_{B_{\varrho}(y)}u_0(x)\dx$ and $u_0\ge 0$. Thus, we have found that:
\[\label{Posit.Final}
\inf_{x\in B_R(0)}u(t,x) =u(t,x_m) \ge
\tilde{\underline{\kappa}}\,\frac{M_{2\,R}^\frac 2\alpha(x_m)}{t^\frac d\alpha}\ge
\tilde{\underline{\kappa}}\,\frac{M_{2\,R}^\frac 2\alpha(0)}{t^\frac d\alpha}=
\tilde{\underline{\kappa}}\,\frac{M_{R}^\frac 2\alpha(0)}{t^\frac d\alpha}\,.
\]
for $t>2\,\tilde{\underline{t}}(0)=\tilde{\kappa}_\star M_{4R}^{1-m}(0) R^\alpha=\tilde{\kappa}_\star\,M_{R}^{1-m}(0) R^\alpha$, after noticing that $M_{4R}(0)=M_{2\,R}(0)=M_R(0)$, since $\supp(u_0)\subset B_R(0)$. Finally we obtain the claimed estimate
\[\label{FDE.R}
\inf_{x\in B_R(0)}u(t,x)\ge
\tilde{\underline{\kappa}}\,\frac{M_{R}^\frac 2\alpha}{t^\frac d\alpha}\quad\forall\,t\ge 2\,\tilde{\underline{t}}\,.
\]

\stepitem The last step consists in obtaining a lower estimate when $0\le t\le 2\,\tilde{\underline{t}}$. To this end we consider the fundamental estimate of B\'enilan-Crandall~\cite{Benilan1981}:
\[
u_t(t,x)\le\frac{u(t,x)}{(1-m)t}\,.
\]
This easily implies that the function:
\[
u(t,x)t^{-1/(1-m)}
\]
is non-increasing in time, thus for any $t\in (0,2\,\tilde{\underline{t}})$ we have that
\[
u(t,x)\ge u(2\,\underline{t},x)\,\frac{t^{1/(1-m)}}{(2\,\tilde{\underline{t}})^{1/(1-m)}}\ge\,\tilde{\underline{\kappa}}\,\tilde{\kappa}_\star^{-\frac2{1-m}}\,\left(t\,R^{-2}\right)^\frac1{1-m}\,.
\]
which is exactly inequality~\eqref{BV-3}. It is straightforward to verify that the constant~$\tilde{\kappa}$ has the value
\begin{equation}\label{kappa-tilde}
\tilde{\kappa}
=\tilde{\underline{\kappa}}\,\tilde{\kappa}_\star^{-\frac2{1-m}}\,=\alpha\,A_d\frac{\left[
d\,(1-m)\right]^{\frac{d\,(1-m)}{\alpha}}}{2^\frac 2\alpha}\,C_3^{-\frac2\alpha}\,\left(1-4^{-\frac d\alpha}\right)^\frac2\alpha\,\tilde{\kappa}_\star^{-\frac2{1-m}}\,.
\ee

\stepitem\textit{Simplification of the constants}. In this step we are going to simplify the expression of some constants in order to obtain the expression in~\eqref{kappaExpr-kappastarExpr}. This translates into estimates from below of the actual values of constants $\tilde{\overline{\kappa}}$ and $\tilde{\kappa}_\star$, and in order to do so, we need to estimate $C_2$ and $ C_3$. Let us begin with $C_2$, since we only need an estimate from below. We learn from~\eqref{omega-over-d} that $\omega_d/d\le \pi^2$ for any $d\ge1$. It is then clear from~\eqref{C2} that
\[\label{C2-2}
2^d\le C_2\le 2^d\,\overline{\kappa}\,\pi^2\,.
\]

In the case of $C_3$ we already have a lower bound given in~\eqref{first-estimates}, in what follows we compute the upper bound. Let us recall that from~\eqref{omega-d} we have that for any $d\ge1$, $\omega_d\le 16\pi^3\,/15$. Since $m<1$ we have that
\[
16\(d+1\)\(3+m\)\le 64\(d+1\)\le 128\,d\,.
\]

Combining the above inequality, with the estimates on $\omega_d$ and the defintion of $C_3$ we get
\[\label{second-estimate}
\(4d\)^d\,\le C_3\,\le\(\frac{128\,d}{1-m}\)^\frac2{1-m}\,4\,\pi^3\,.
\]
Therefore, we can estimate $\tilde{\kappa}_\star$ and obtain the expression of $\kappa_\star$
\[\label{kappa-star}
\tilde{\kappa}_\star=4\(C_2\,C_3\)^\frac{\alpha}{d} \ge 2^2\(2^{5\,d}\,d^d\)^\frac{\alpha}{d}=2^{3\,\alpha+2}\,d^{\alpha}=:\kappa_\star\,.
\]
Let us simplify $\tilde{\kappa}$. By combining~\eqref{kappa-tilde},~\eqref{tilde-kappa-star} and~\eqref{AD}, we get that
\[\label{second-definition-kappa}
\tilde{\kappa}\ge\alpha\,\omega_d\,2^{2d-2-\frac{2\,(1-m)+4\,\alpha}{\alpha(1-m)}}\,\left[
d\,(1-m)\right]^{\frac{d\,(1-m)}{\alpha}}\,C_3^{-\frac{4}{\alpha\,d\,(1-m)}}\,C_2^{-\frac{2\,\alpha}{d\,(1-m)}}\,\left(1-4^{-\frac d\alpha}\right)^\frac2\alpha\,.
\]
Let us begin simplifying the expression $\left(1-4^{-\frac d\alpha}\right)$. We first notice that, since $\alpha \in\(1,2\)$ we have that $1-4^{-\frac d\alpha} \ge 1-4^{-\frac{d}2} $, which is an expression monotone increasing in $d$. We have therefore that
\[
\left(1-4^{-\frac d\alpha}\right)^\frac2\alpha\ge \left(1-4^{-\frac12}\right)^\frac2\alpha=2^{-\frac2\alpha}\,.
\]
Combining all together we find
\[
\tilde{\kappa}\ge\alpha\,\omega_d\,2^{-\mathfrak{a}}\,\pi^{-\mathfrak{b}}\,\overline{\kappa}^{-\frac{2\,\alpha}{d\,(1-m)}}\,d^{\frac{d\,(1-m)}{\alpha}-\frac{8}{\alpha(1-m)^2d}}\(1-m\)^\frac{d^2\,(1-m)^3+8}{\alpha\,(1-m)^2\,d}\,,
\]
where
\[
\mathfrak{a}=\frac{56+8\,(1-m)+2\,\alpha^2\,d\,(1-m)+2\,\alpha\,(1-m)^2\,d}{\alpha\,(1-m)^2\,d}-2\,d\quad\mbox{and}\quad
\mathfrak{b}=\frac{12+4\,\alpha^2}{d\,(1-m)}\,.
\]
Since $m_1<m<1$, and $d\,(1-m)<1$, we can simplify the expression of $\mathfrak{a}$ and $\mathfrak{b}$ into
\[
\mathfrak{a}\le\frac{76}{\alpha\,(1-m)^2\,d}\quad\mbox{and}\quad\mathfrak{b}\le\frac{32}{d\,(1-m)}\,.
\]
By summing up all estimates above and estimating the exponents of $(1-m)$ and~$d$, we get
\[\label{kappa-definition}
\tilde{\kappa} \ge\frac{\alpha\,\omega_d\,\left(\frac{1-m}{d}\right)^\frac8{\alpha\,(1-m)^2\,d}}{2^\frac{76}{\alpha\,(1-m)^2\,d}\,\pi^\frac{32}{d\,(1-m)}\,\overline{\kappa}^\frac{2\,\alpha}{d\,(1-m)}}=\kappa\,.
\]
\end{steps}
\end{proof}

%%%%%%%%%%%%%%%%%%%%%%%%%%%%%%%%%%%%%%%%%%%%%%%%%%%%%%%%%%%%%%%%%%%%%%
\subsection{Details on the inner estimate in terms of the free energy}\label{Sec:InnerDetails}

In~\cite[Propostion 11]{BDNS2020} the following proposition is proven.
%---------------------------------------------------------------------
\begin{proposition}\label{Proposition11} Assume that $m\in(m_1,1)$ if $d\ge2$, $m\in(1/3,1)$ if $d=1$ and let $\varepsilon\in(0,1/2)$, small enough and $G>0$ be given. There exist a numerical constant $\mathsf K>0$ and an exponent $\vartheta\in(0,1)$ such that, for any $t\ge 4\,T(\varepsilon)$, any solution $u$ of~\eqref{FD} with nonnegative initial datum $u_0\in\mathrm L^1(\R^d)$ satisfying~ $\ird{u_0}=\ird{\mB}$ and $\mathcal F[u_0]\le G$, then satisfies
\be{control-radius-inequality2}
\left|\frac{u(t,x)}{B(t,x)}-1\right|\le\frac{\mathsf K}{\varepsilon^\frac1{1-m}}\,\left(\frac1t+\frac{\sqrt G}{R(t)}\right)^\vartheta\quad\mbox{if}\quad |x|\le 2\,\rho(\varepsilon)\,R(t)\,.
\ee
\end{proposition}
%---------------------------------------------------------------------
The values of the parameters $\rho(\varepsilon)$ and $T(\varepsilon)$ can be found in Appendix~\ref{Appendix:UserGuide}. Here we give details on the computation of $\mathsf K$.

\begin{proof} In this section we have to explain how to compute the constant $c_1$, to prove that $\lambda_0$ and $\lambda_1$ are bounded and bounded away from zero and to obtain the final form of the constant $\overline C$ in formula $(83)$ of~\cite{BDNS2020}. Let us start with the first task

\noindent\textit{How to compute $c_1$.} Let us recall the cylinders
\begin{align}\label{cilinders-s}
Q_1&:=\left(1/2, 3/2\right)\times B_1(0)\,,\quad Q_2:=\left(1/4, 2\right)\times B_{8}(0)\,,\\
Q_3&:=\left(1/2, 3/2\right)\times B_1(0)\setminus B_{1/2}(0)\quad\mbox{and}\quad Q_4:=\left(1/4, 2\right)\times B_{8}(0)\setminus B_{\tfrac14}(0)\,.
\end{align} 
and let us assume that $v$ is a solution to~\eqref{HE.coeff} which satisfies~\eqref{HE.coeff.lambdas} for some $0<\lambda_0\le\lambda_1<\infty$.
In what follows we shall explain how to compute the constant $c_1$ in:

\boxedresult{
~\cite[Inequality $(66)$]{BDNS2020},
\[\label{holder-continuity-inequality-s}
\sup_{(t,x),(s,y)\in Q_{i}}\frac{|v(t,x)-v(s,y)|}{\big(|x-y|+|t-s|^{1/2}\big)^\nu}\le\,c_1\,\|v\|_{\mathrm L^\infty(Q_{i+1})}\quad\forall\,i\in\{1,2\}\,.
\]}
By applying Theorem~\ref{Claim:4} it is clear that the only ingredient needed is to estimate from below $d(Q_1,Q_2)$ and $d(Q_3,Q_4)$, where $d(\cdot, \cdot)$ is defined in~\eqref{parabolic-distance}. Let us consider the case of $d(Q_1, Q_2)$. By symmetry, it is clear that the infimum in~\eqref{parabolic-distance} is achieved by a couple of points $(t,x)\in\overline{Q_1}$, $(s, y)\in\partial Q_2$ such that either $|x|=1, t\in\left(1/2, 3/2\right)$ and $|y|=8, s\in\left(1/2, 3/2\right)$ or $t=1/2, y=1/4$ and $x, y\in B_1$. In both cases we have that $d(Q_1, Q_2)=|x-y|+|t-s|^\frac12\ge1/4$. By a very similar argument we can also conclude that $d(Q_3, Q_4)\ge1/4$. Therefore, we conclude that, in both cases, $c_1$ can be taken (accordingly to inequality~\eqref{holder-continuity-inequality}).
\[\label{c-1}
2\(128\)^\nu \max\left\{\frac1{d(Q_1, Q_2)}\,,\frac1{d(Q_3, Q_4)}\,\right\}^\nu\le 2\(512\)^\nu \le 2^{10}=:c_1\,,
\]
where we have used the fact that $\nu\in(0,1)$.

\noindent\textit{Estimates of $\lambda_0$ and $\lambda_1$ of~\cite[formula $(73)$]{BDNS2020}}. In \textit{Step 2} of~\cite[Proposition 11]{BDNS2020} we considered a solution $u(t,x)$ to~\eqref{FD} as a solution to the linear equation~\eqref{HE.coeff} with coefficients
\[
a(t,x)=m\,u^{m-1}(t,x)\,,\quad A(t,x)=a(t,x)\,\mathrm{Id}\,,
\]
where $\mathrm{Id}$ is the identity matrix on $\R^d$. We also observed that $u(t,x)$ (and its rescaled version $\hat{u}_{\tau, k}$) satisfies the condition~\eqref{HE.coeff.lambdas} (with the coefficient $a(t,x)$ given by the above expression above) and with

\boxedresult{
~\cite[Definition ($73$)]{BDNS2020}\,
\[\label{lambdas-s}\begin{split}
\lambda_0^{\frac1{m-1}}&:=m^\frac1{m-1}\,\overline C\,\max\big\{\sup_{Q_2}B(t-\tfrac1\alpha,x),\,\sup_{k\ge1}\sup_{Q_4}B(t-\tfrac1\alpha,x;k^\frac\alpha{1-m}\,\Mstar)\big\}\,,\\
\lambda_1^{\frac1{m-1}}&:=\,m^\frac1{m-1}\,\underline{C}\,\min\big\{\inf_{Q_2}B(t-\tfrac1\alpha,x)\,, \inf_{k\ge1}\inf_{Q_4}B(t-\tfrac1\alpha,x;k^\frac\alpha{1-m}\,\Mstar)\big\}\,.
\end{split}
\]
}
where $Q_i$ are as in~\eqref{cilinders-s}. Our task here is to give an estimate on $\lambda_0$, $\lambda_1$ and to show that they are bounded and bounded away from zero. Let us consider firs the case of $B(t-\tfrac1\alpha,x)$: for any $(t,x)\in(0, \infty)\times \R^d$ we have that
\[ 
B(t-\tfrac1\alpha,x)=\frac{t^\frac1{1-m}}{\lambdaBarenblatt^\frac\alpha{1-m}}\(\frac{t^\frac2\alpha}{\lambdaBarenblatt^2}+|x|^2\)^\frac1{m-1}\mbox{where}\quad\lambdaBarenblatt=\(\tfrac{1-m}{2\,m\,\alpha}\)^\frac1\alpha\,.
\]
We deduce therefore that, for any $(t,x)\in Q_2$ we have that
\[
\frac1{4^\frac1{1-m}\,\lambdaBarenblatt^\frac{\alpha}{1-m}}\(\frac{2^\frac2\alpha}{\lambdaBarenblatt^2}+2^6\)^\frac1{m-1}\le B(t-\tfrac1\alpha,x)\le \lambdaBarenblatt^d\,4^\frac d\alpha\,.
\]
This is enough to prove that $\lambda_0>0$ and $\lambda_1<\infty$. Let us consider $B(t-\tfrac1\alpha,x;k^\frac\alpha{1-m}\,\Mstar)$, we recall that
\[
B(t-\tfrac1\alpha,x;k^\frac\alpha{1-m}\,\Mstar)=\frac{t^\frac1{1-m}}{\lambdaBarenblatt^\frac{\alpha}{1-m}}\(\frac{t^\frac2\alpha}{k^2\,\lambdaBarenblatt^2}+|x|^2\)^\frac1{m-1}\,.
\]
Let us consider $(t,x)\in Q_4$, we have therefore
\[
\frac1{4^\frac1{1-m}\,\lambdaBarenblatt^\frac{\alpha}{1-m}}\(\frac{2^\frac2\alpha}{k^2\,\lambdaBarenblatt^2}+64\)^\frac1{m-1} \le B(t-\tfrac1\alpha,x;k^\frac\alpha{1-m}\,\Mstar) \le\frac{2^\frac1{1-m}}{\lambdaBarenblatt^\frac{\alpha}{1-m}}\(\frac1{\lambdaBarenblatt^2\,k^2\,4^\frac2\alpha}+\frac1{16}\)^\frac1{m-1}\,.
\]
From the above computation we deduce that
\[
\sup_{k\ge1}\sup_{Q_4}B(t-\tfrac1\alpha,x;k^\frac\alpha{1-m}\,\Mstar)\big\}\le\frac{2^\frac1{1-m}}{\lambdaBarenblatt^\frac{\alpha}{1-m}}\(\frac1{\lambdaBarenblatt^2\,4^\frac2\alpha}+\frac1{16}\)^\frac1{m-1}\,,
\]
while
\[
\frac1{2^\frac7{1-m}\,\lambdaBarenblatt^\frac{\alpha}{1-m}}\le \inf_{k\ge1}\inf_{Q_4}B(t-\tfrac1\alpha,x;k^\frac\alpha{1-m}\,\Mstar)\,.
\]
Combining all estimates together we obtain that
\[
0<\lambda_0\le \lambda_1<\infty\,,
\]
this completes the proof of this part.

\noindent\textit{Simplification of the constant $\overline{C}$}. Recall that $\alpha\in (1,2)$ so that
\[
\overline c_3+\frac{2\,\overline c_2}{\alpha}=\frac1{1-m}+\frac{2m}{2\,(1-m)^2\,\alpha^3}
\le\frac1{1-m}+\frac{m}{(1-m)^2}=\frac1{(1-m)^2}
\]
hence,
\begin{align*}\label{constant-c-control-radius.est}
\overline C&=\,2^{\frac{d}\alpha}\,\left(C\,+\(\overline c_3+\tfrac2\alpha\,\overline c_2\)\right)
\left(\sqrt{\tfrac{4\,\alpha\,\Mstar}m}+\(\overline c_3+\tfrac2\alpha\,\overline c_2\)\Mstar\right)^\vartheta\\
&\le\,2^{\frac{d}\alpha}\,
\left(C+\tfrac1{(1-m)^2}\right)
\left(\frac{2\,\alpha}{m}+\Mstar+\frac{\Mstar}{(1-m)^2}\right)^\vartheta\\
&\le\,2^{\frac{d}\alpha+\vartheta}\,\frac{\,\left(1+C\right)}{m^\vartheta(1-m)^{2(1+\vartheta)}}
(\alpha+\Mstar)^\vartheta
\end{align*}
where the constant $C$ is given by
\begin{align*}%\label{constant-radius-control.est}
C:=\lambdaBarenblatt^d&\left(1+4\,\lambdaBarenblatt^2\,Z^2\,\rho(\varepsilon)^2\right)^\frac1{1-m}\,C_{d, \nu, 1}\\
&\times\left(\(c_1\,4^\frac d\alpha\,\overline\kappa\,\Mstar^\frac2\alpha\,\frac{2^\nu}{2^\nu-1}+c_2\)^\frac{d}{d+\nu}+\frac1{\(2\,Z\,\rho(\varepsilon)\)^d}\,(2\,\Mstar)^\frac d{d+\nu}\right)\,.
\end{align*}
where
\[
c_1:=2^{10}\,,\quad c_2:=2\,\max\Big\{\lambdaBarenblatt\,,\||\nabla B(1-\tfrac1\alpha,x)|\|_{\mathrm L^\infty(\R^d)}\Big\}\quad\mbox{and}\quad Z=\(2\,\alpha\)^\frac1\alpha\,.
\]
For any $\varepsilon\in(0, \varepsilon_{m,d})\subset(0,1/2)$ we have that
\[\label{Rbis.est}
\sqrt{
\tfrac{(1-\underline\varepsilon)^m(\underline\varepsilon-\varepsilon)}{2^{1+m}}}\,
\frac1{\mu\,\sqrt\varepsilon} \le \underline \rho(\varepsilon)=\frac1\mu\(\big(1+(1+\varepsilon)^{1-m}\big)\,\tfrac{\(\tfrac{1-\varepsilon}{1-\underline\varepsilon}\)^{1-m}-1}{1-(1-\varepsilon)^{1-m}}\)^{1/2}\le\frac{2\,\sqrt{\underline \varepsilon}}{\mu\,\sqrt\varepsilon}
\]
and
\[\label{Rter.est}
\frac1{\sqrt{1-m}}\,\frac1{\mu\,\sqrt\varepsilon}
\le\overline \rho(\varepsilon)=\frac1\mu\left(\,\frac{(1+\varepsilon)^{1-m}+1}{(1+\varepsilon)^{1-m}-1}\right)^\frac12\,\le\frac4{ \sqrt{1-m}}\frac1{\mu\,\sqrt\varepsilon}\,.
\]
We recall that $\underline \varepsilon<1$, we obtain therefore that
\[\label{rho.eps.est}\begin{split}
\rho(\varepsilon)^2&:=\max\{\overline \rho(\varepsilon), \underline \rho(\varepsilon)\}^2
\le \max\left\{\frac4{\mu^2\,\varepsilon}
\,,\,\frac{16}{(1-m)}\frac1{\mu^2\,\varepsilon}\right\}
\le\frac{16}{(1-m)^2\,\mu^2\,\varepsilon}
\end{split}\]
and also, since $\varepsilon<1/2$ we have that
\[\label{rho.eps.est.low}\begin{split}
\rho(\varepsilon)^2&
\ge\frac1{(1-m)\mu^2\,\varepsilon} \ge\frac2{(1-m)\,\mu^2}\,.
\end{split}\]

Combining all above estimates together we find that
\[\begin{split}
\left(1+4\,\lambdaBarenblatt^2\,(2\,\alpha)^\frac2\alpha\,\rho(\varepsilon)^2\right)^\frac1{1-m} &\le \left(\frac{\mu^2+2^{\frac{6\,\alpha+2}{\alpha}} \lambdaBarenblatt^2}{(1-m)^2, \mu^2\,\varepsilon}\right)^\frac1{1-m} \\ &\le\frac{2^\frac{2+6\alpha}{\alpha(1-m)}}{(1-m)^\frac2{1-m}\varepsilon^\frac1{1-m}}\left(\frac{\mu^2+\alpha^\frac2\alpha\lambdaBarenblatt^2}{\mu^2}\right)^\frac1{1-m} \\
& \le\frac{2^\frac{3+6\alpha}{\alpha(1-m)}}{(1-m)^\frac2{1-m}\varepsilon^\frac1{1-m}}
\end{split}
\]
where in the last step we have used the identity $\mu=\lambdaBarenblatt\,\alpha^\frac1{\alpha}$.

Combining the above estimates we finally get 
\[ \begin{split}
\overline C & \le\frac{2^{\frac{d}\alpha+\frac{3+6\alpha}{\alpha(1-m)}+\vartheta}}{\varepsilon^\frac1{1-m}}\,\frac{(\alpha+\Mstar)^\vartheta}{m^\vartheta(1-m)^{2(1+\vartheta)+\frac2{1-m}}}\\
&\quad\times \left[1+\lambdaBarenblatt^d\,C_{d, \nu, 1} \left(\(2^{10+\frac{2d}{\alpha}}\,\overline\kappa\,\Mstar^\frac2\alpha\,\frac{2^\nu}{2^\nu-1}+c_2\)^\frac{d}{d+\nu}+\(\frac{\mu^2}{\alpha^\frac1{\alpha}}\)^d\,(2\,\Mstar)^\frac d{d+\nu}\right)\right] \\
&\le\frac{2^{\frac{3d}\alpha+\frac{3+6\alpha}{\alpha(1-m)}+\vartheta+10}}{\varepsilon^\frac1{1-m}}\,\frac{(\alpha+\Mstar)^\vartheta}{m^\vartheta(1-m)^{2(1+\vartheta)+\frac2{1-m}}} \\
&\quad \times\left[1+\lambdaBarenblatt^d\,C_{d, \nu, 1} \left(\(\overline\kappa\,\Mstar^\frac2\alpha\,\frac{2^\nu}{2^\nu-1}+c_2\)^\frac{d}{d+\nu}+\frac{\mu^{2d}}{\alpha^\frac d\alpha}\,\Mstar^\frac d{d+\nu}\right)\right]=:\frac{\mathsf K}{\varepsilon^\frac1{1-m}}\,.
\end{split} 
\]
The proof is completed.\end{proof}

%%%%%%%%%%%%%%%%%%%%%%%%%%%%%%%%%%%%%%%%%%%%%%%%%%%%%%%%%%%%%%%%%%%%%%
%%%%%%%%%%%%%%%%%%%%%%%%%%%%%%%%%%%%%%%%%%%%%%%%%%%%%%%%%%%%%%%%%%%%%%
\renewcommand{\thesection}{\Alph{section}.}
\renewcommand{\thesubsection}{\Alph{section}.\arabic{subsection}}
\appendix
\section{Further estimates and additional results}\label{Appendix}

Here we collect additional material concerning various estimates: the ``user guide'' of Appendix~\ref{Appendix:UserGuide} collects the formulas needed for the computation of $t_\star$ in Theorem~\ref{Thm:RelativeUniformDetailed}; Appendix~\ref{Appendix:Numerics} details how the numerical value of the constant on the disk in~\cite[Appendix~C.2]{BDNS2020} is computed; Appendix~\ref{Appendix:Truncation} is devoted to the precise definition of the truncation functions used in Section~\ref{sec.herrero.pierre}.

%%%%%%%%%%%%%%%%%%%%%%%%%%%%%%%%%%%%%%%%%%%%%%%%%%%%%%%%%%%%%%%%%%%%%%
\subsection{A user guide for the computation of the threshold time}\label{Appendix:UserGuide}

Let us recall what the threshold time is. The results of~\cite[Theorem~4]{BDNS2020} and~\cite[Proposition~12]{BDNS2020} can be summarized as follows.
%---------------------------------------------------------------------
\begin{theorem}[\cite{BDNS2020}]\label{Thm:RelativeUniformDetailed} Assume that $m\in(m_1,1)$ if $d\ge2$, $m\in(1/3,1)$ if \hbox{$d=1$}. There is a numerical constant $\varepsilon_{m,d}\in(0,1/2)$, a real number $\nu>0$, and a positive numerical constant $\taustar=\taustar(m,d)$ with $\lim_{m\to1_-}\taustar(m,d)=+\infty$ such that the following property holds: for any $A>0$ and $G>0$, if $u$ is a solution of~\eqref{FD} with nonnegative initial datum $u_0\in\mathrm L^1(\R^d)$ such that $\ird{u_0}=\ird{\mB}$, $\mathcal F[u_0]\le G$ and
\begin{equation}\tag{H$_A$}\label{hyp:Harnack}
\sup_{r>0}r^\frac{d\,(m-m_c)}{(1-m)}\int_{|x|>r}u_0\,dx\le A<\infty
\ee
and if $\varepsilon\in(0,\varepsilon_{m,d})$, then
\[\label{control}
\sup_{x\in\R^d}\left|\frac{u(t,x)}{B(t,x)}-1\right|\le\varepsilon\quad\forall\,t\ge t_\star\,,
\]
where
\be{t.star.simplified}
t_\star=\frac{\taustar}{\varepsilon^{\mathsf a}}\left(1+A^{1-m}+G^\frac\alpha2\right)\quad\mbox{with}\quad\mathsf a=\frac{\alpha}{\vartheta}\,\frac{2-m}{1-m}\,.
\ee
\end{theorem}
%---------------------------------------------------------------------
\noindent We do not reproduce the proof here but establish the expression of $\taustar$ by collecting all intermediate constants and formulas. We assume from now on that $u$ is a solution of~\eqref{FD} which satisfies all assumptions of Theorem~\ref{Thm:RelativeUniformDetailed}.

\medskip Let us start by recalling the definition of the main numerical parameters:
\[
m_c=\frac{d-2}d\,,\;m_1=\frac{d-1}d\,,\;\alpha=d\,(m-m_c)\,,\;\mu=\(\frac{1-m}{2\,m}\)^\frac1\alpha\;\mbox{and}\;\lambdaBarenblatt=\(\frac{1-m}{2\,m\,\alpha}\)^\frac1\alpha\,.
\]
%---------------------------------------------------------------------
\boxedresult{\cite[Corollary 9]{BDNS2020} 
\[\label{SUPPL-GHP-6Lower}\tag{49}
u(t,x)\ge(1-\varepsilon)\,B(t\,,\,x\,)\quad\mbox{if}\quad |x|\ge R(t)\,\underline \rho(\varepsilon)
\quad\mbox{and}\quad\varepsilon\in\(0,\underline\varepsilon\).
\]
}
%---------------------------------------------------------------------
\newtagform{starred}{\cite[Eq.~(}{)]{BDNS2020}}\usetagform{starred}
Here
\[\label{SUPPL-R}\tag{21}
R(t)=(1+\alpha\,t)^{1/\alpha}
\]
and
\[
\underline\varepsilon:=1-\(\underline M\,/\Mstar\)^\frac2\alpha
\]
is a numerical constant which is computed from
\[\label{SUPPL-underlineM}\tag{47}
\underline M:=\min\left\{2^{-\,d/2}\,\Big(\frac\kappa{\lambdaBarenblatt^d}\Big)^{\alpha/2},\frac{\kappa}{\big(d\,(1-m)\big)^{ d/2}\,\alpha^\frac\alpha{2\,(1-m)}}\right\}\;\kappa_\star^\frac1{1-m}\,\Mstar^2\,.
\]
We recall that $\kappa$ and $\kappa_\star$ have been defined in~\eqrefstd{kappaExpr-kappastarExpr} and are given by
\[
\kappa_\star=2^{\,3\,\alpha+2}\,d^{\,\alpha}\quad\mbox{and}\quad\kappa=\alpha\,\omega_d\(\frac{(1-m)^4}{2^{38}\,d^{\,4}\,\pi^{16\,(1-m)\,\alpha}\,\overline\kappa^{\,\alpha^2\,(1-m)}}\)^\frac2{(1-m)^2\,\alpha\,d}\,.
\]
As a byproduct of~\cite[Proposition~8]{BDNS2020}, by integrating over $\R^d$, we deduce from
\[\label{SUPPL-GHP-2}\tag{38}
u(t,x)\ge B\big(t-\underline t-\tfrac1\alpha\,,\,x\,;\,\underline M\big)
\]
that $\underline M\,/\Mstar<1$, which proves that $\underline\varepsilon>0$. The two other constants of Corollary~9 are given by
\[\label{SUPPL-Rbis}\tag{52}
\underline \rho(\varepsilon):=\frac1\mu\(\big(1+(1+\varepsilon)^{1-m}\big)\,\frac{\(\tfrac{1-\varepsilon}{1-\underline\varepsilon}\)^{1-m}-1}{1-(1-\varepsilon)^{1-m}}\)^{1/2}
\]
and
\[\label{SUPPL-t1bis}\tag{51}
\underline T(\varepsilon):=\frac{\kappa_\star\(2\,A\)^{1-m}+\frac2\alpha}{1-(1-\varepsilon)^{1-m}}\,.
\]

%---------------------------------------------------------------------
\boxedresult{\cite[Corollary 10]{BDNS2020}
\[\label{SUPPL-GHP-6Upper}\tag{54}
u(t,x)\le(1+\varepsilon)\,B(t\,,\,x\,)\quad\mbox{if}\quad |x|\ge\,R(t)\,\overline \rho(\varepsilon)\quad\mbox{and}\quad\varepsilon\in\(0,\overline\varepsilon\)\,.
\]
}
%---------------------------------------------------------------------
 As a byproduct of Proposition~7, by integrating over $\R^d$, we deduce from
\[\label{SUPPL-GHP-1}\tag{32}
u(t,x)\le\,B\big(t+\overline t-\tfrac1\alpha\,,\,x\,;\,\overline M\big)
\]
that $\overline M\,/\Mstar>1$, which proves that
\[
\overline\varepsilon:=\(\overline M/\Mstar\)^\frac2\alpha-1>0\,.
\]
Notice that $\overline\varepsilon$ is a numerical constant. The two other constants of Corollary~10 are given by
\[\label{SUPPL-Rter}\tag{57}
\overline \rho(\varepsilon)
:=\frac1\mu\left(\frac{(1+\varepsilon)^{1-m}+1}{(1+\varepsilon)^{1-m}-1}\right)^\frac12
\]
and
\[\label{SUPPL-Tter}\tag{56}
\overline T(\varepsilon):=\frac{2\,\overline t}{(1+\varepsilon)^{1-m}-1}
\]
where 
\[\label{SUPPL-toverline}\tag{36}
c:=\max\big\{1, 2^{5-m}\,\overline\kappa^{1-m}\,\lambdaBarenblatt^\alpha\big\}\,,\quad\overline t:=c\,t_0\,,
\]
$\overline\kappa$ is given by~\eqrefstd{kappa}, and
\[\label{SUPPL-GHP-1-time}\tag{34}
t_0:=A^{1-m}\,.
\]
%---------------------------------------------------------------------
\boxedresult{\cite[Proposition 11]{BDNS2020}
\[\label{SUPPL-control-radius-inequality2}\tag{61}
\left|\frac{u(t,x)}{B(t,x)}-1\right|\le \overline\,\mathsf K\,\left(\frac1t+\frac{\sqrt G}{R(t)}\right)^\vartheta\quad\mbox{if}\quad |x|\le 2\,\rho(\varepsilon)\,R(t)\quad\mbox{and}\quad\varepsilon\in\left(0, \varepsilon_{m,d}\right)\,.
\]
}
%---------------------------------------------------------------------
The range of admissible $\varepsilon$ is determined by
\[\label{SUPPL-epsilon.md.def}\tag{59}
\varepsilon_{m,d}:=\min\left\{\overline\varepsilon,\,\underline\varepsilon,\,\tfrac12\right\}
\]
and $\varepsilon\le\chi\,\eta$ where $\eta=2\,d\,(m-m_1)$ and $\chi:=\frac m{266+56\,m}$. The exponent
\[\label{SUPPL-theta}\tag{78}
\vartheta=\frac{\nu}{d+\nu}\,.
\]
is defined as follows. Let
\[\label{SUPPL-nu}\tag{67}
\nu:=\log_4\(\frac{\overline{\mathsf h}}{\overline{\mathsf h}-1}\)\quad\mbox{with}\quad\overline{\mathsf h}:=\mathsf h^{\lambda_1+1/\lambda_0}\,.
\]
The value of the constant $\mathsf h$ has been computed in~\cite{bonforte2020fine} (also see~\eqrefstd{h}) and is given by
\[\label{SUPPL-h}\tag{65}
\mathsf h:=\exp\left[2^{d+4}\,3^d\,d+c_0^3\,2^{2\,d+7}\left(1+\frac{2^{d+2}}{(\sqrt2-1)^{2\,d+4}}\right)\sigma\right]
\]
where
\[\label{SUPPL-c_0}
c_0=3^\frac2{d}\,2^\frac{(d+2)\,(3\,d^2+18\,d+24)+13}{2\,d}\(\tfrac{(2+d)^{1+\frac4{d^2}}}{d^{1+\frac2{d^2}}}\)^{(d+1)(d+2)}\,\mathcal K^\frac{2\,d+4}{d}\,,
\]
\[
\sigma=\sum_{j=0}^{\infty}\left(\tfrac34\right)^j\,\big((2+j)\,(1+j)\big)^{2\,d+4}
\]
and $\mathcal K$ is the optimal constant in the interpolation inequality~\eqrefstd{GNS111}, that is,
\[
\|f\|^2_{\mathrm L^{\pc}(B)}\le\mathcal K\(\|\nabla f\|^2_{\mathrm L^2(B)}+\|f\|^2_{\mathrm L^2(B)}\).
\]
The values of $\mathcal K$ are given in Table~\ref{table.k.bar}. We refer to Section~\ref{Sec:InnerDetails} for the values of $\lambda_0$ and $\lambda_1$.

\medskip As for the other constants, we have
\[\label{SUPPL-RE}\tag{60}
\rho(\varepsilon):=\max\big\{\overline \rho(\varepsilon),\,\underline \rho(\varepsilon)\big\}\,,\quad T(\varepsilon):=\max\big\{\overline T(\varepsilon),\,\underline T(\varepsilon)\big\}
\]
and
\[\label{SUPPL-constant-c-control-radius.suppl.1}\tag{84} \begin{split}
\mathsf K&:=2^{\frac{3\,d}\alpha+\frac{3+6\,\alpha}{\alpha\,(1-m)}+\vartheta+10}\,\frac{(\alpha+\Mstar)^\vartheta}{m^\vartheta(1-m)^{2\,(1+\vartheta)+\frac2{1-m}}}\\
&\quad \times\left[1+\lambdaBarenblatt^d\,C_{d, \nu, 1} \left(\(\overline\kappa\,\Mstar^\frac2\alpha\,\frac{2^\nu}{2^\nu-1}+c_2\)^\frac{d}{d+\nu}+\frac{\mu^{2d}}{\alpha^\frac d\alpha}\,\Mstar^\frac d{d+\nu}\right)\right]\,.
\end{split}\]
The exponent $\vartheta$ is the same as above. The other constants in the expression of $\mathsf K$ are
\[\label{SUPPL-gamma-norm-barenblatt}\tag{71}
c_2=2\,\max\Big\{\lambdaBarenblatt\,,\|\nabla B(1-\tfrac1\alpha,\cdot)\|_{\mathrm L^\infty(\R^d)}\Big\}
\]
where
\begin{multline*}
\|\nabla B(1-\tfrac1\alpha,\cdot)\|_{\mathrm L^\infty(\R^d)}=\(\tfrac\mu{\alpha^{1/\alpha}}\)^{d+1}\,\sup_{z>0}\tfrac{2\,z}{1-m}\(1+z^2\)^{-2\,\frac{2-m}{1-m}}\\
=\frac{\mu^{d+1}}{\alpha^\frac{d+1}\alpha}\,\tfrac{2^\frac1{m-1}}{\sqrt{(1-m)(3-m)}}\(\tfrac{3-m}{2-m}\)^\frac{2-m}{1-m}\,,
\end{multline*}
and $C_{d,\nu,1}$ corresponds to the optimal constant for $p=1$ in
\[\label{SUPPL-interpolation.inequality.cpt6}\tag{102}
\nrm u{\mathrm L^\infty(B_{R}(x))}\,\le\,C_{d, \nu, p}\left(\lfloor u\rfloor_{C^\nu(B_{2\,R}(x))}^{\frac d{d+p\,\nu}}\,\|u\|_{\mathrm L^p(B_{2\,R}(x))}^{\frac{p\,\nu}{d+p\,\nu}}+R^{-\frac dp}\,\|u\|_{\mathrm L^p(B_{2\,R}(x))}\right)\,.
\]
We know from~\cite[Appendix~A]{BDNS2020} that $C_{d, \nu, p}$ is independent of $R>0$.

The last step is to collect the above estimates and compute
\[\label{SUPPL-taustarabstract}\tag{89}
\taustar(m,d)=\sup_{\varepsilon\in(0,\varepsilon_{m,d})}\max\big\{\varepsilon\,\kappa_1(\varepsilon,m),\,\varepsilon^{\mathsf a}\kappa_2(\varepsilon,m),\,\varepsilon\,\kappa_3(\varepsilon,m)\big\}
\]
where
\begin{multline*}
\kappa_1(\varepsilon,m):=\max\Big\{\frac{8\,c}{(1+\varepsilon)^{1-m}-1}\,,\frac{2^{3-m}\,\kappa_\star}{1-(1-\varepsilon)^{1-m}}\Big\}\,,\\
\kappa_2(\varepsilon,m):=\frac{\(4\,\alpha\)^{\alpha-1}\,\mathsf K^\frac\alpha\vartheta}{\varepsilon^{\,\frac{2-m}{1-m}\,\frac\alpha\vartheta}}\quad\mbox{and}\quad\kappa_3(\varepsilon,m):=\frac{8\,\alpha^{-1}}{1-(1-\varepsilon)^{1-m}}\,.
\end{multline*}
We recall that $c:=\max\big\{1, 2^{5-m}\,\overline\kappa^{1-m}\,\lambdaBarenblatt^\alpha\big\}$ as above (also see~\cite[Eq.~(36)]{BDNS2020}).

\renewtagform{starred}{(}{)}

%%%%%%%%%%%%%%%%%%%%%%%%%%%%%%%%%%%%%%%%%%%%%%%%%%%%%%%%%%%%%%%%%%%%%%
\subsection{A numerical estimate of the constant in the Gagliardo-Nirenberg inequality on the disk}\label{Appendix:Numerics}

The following two-dimensional Gagliardo-Nirenberg inequality has been established in~\cite[Lemma~18]{BDNS2020}.
%---------------------------------------------------------------------
\begin{lemma}\label{Lem:GNd=2} Let $d=2$. For any $R>0$, we have
\be{Disk}
\nrm u{\mathrm L^4(B_R)}^2\le\frac{2\,R}{\sqrt\pi}\left(\nrm{\nabla u}{\mathrm L^2(B_R)}^2+\frac1{R^2}\,\nrm u{\mathrm L^2(B_R)}^2\right)\quad\forall\,u\in\mathrm H^1(B_R)\,.
\ee
\end{lemma}
%---------------------------------------------------------------------
%---------------------------------------------------------------------
\noindent The optimal constant is approximatively
\[
0.0564922...<2/\sqrt\pi\approx1.12838\,.
\]
%---------------------------------------------------------------------
We know from the proof that $\mathcal C\le2/\sqrt\pi\approx1.12838$. Let us explain how we can compute the numerical value of the optimal constant $\mathcal C$ in the inequality~\eqref{Disk}. To compute $\mathcal C$ numerically, we observe that it is achieved among radial functions by symmetrization. The equality case is achieved by some radial function $u$, by standard compactness considerations. It is therefore enough to solve the Euler-Lagrange equation
\be{ODE}
-u''-\frac{u'}r+u=u^3\,,\quad u(0)=a>0\,,\quad u'(0)=0\,.
\ee
To emphasize the dependence of the solution in the shooting parameter $a$, we denote by $u_a$ the solution of~\eqref{ODE} with $u(0)=a$. We look for the value of $a$ for which $u_a$ changes sign only once (as it is orthogonal to the constants) and such that $u'(1)=0$, which is our shooting criterion. Let $s(a)=u_a'(1)$ for the solution of~\eqref{ODE}. With $a=1$, we find that $u_a\equiv1$.
%---------------------------------------------------------------------
\begin{figure}[ht]\begin{center}
\includegraphics[width=8cm]{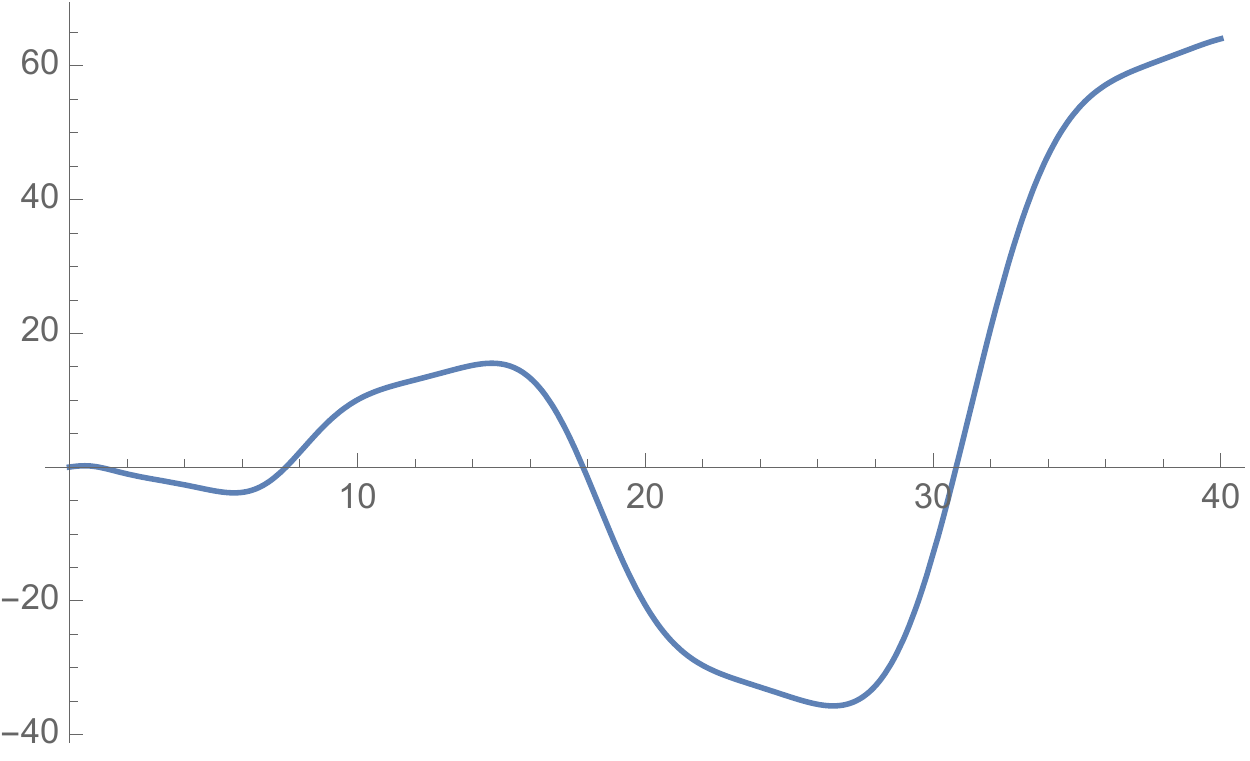}
\caption{Plot of $a\mapsto s(a)$. We find that $s(1)=0$ and also $s(a_\star)=0$ for some $a_\star\approx7.52449$ which provides us with a solution $u_{a_\star}$ with only one sign change.}
\end{center}
\end{figure}
%---------------------------------------------------------------------
\begin{figure}[ht]\begin{center}
\includegraphics[width=8cm,height=4cm]{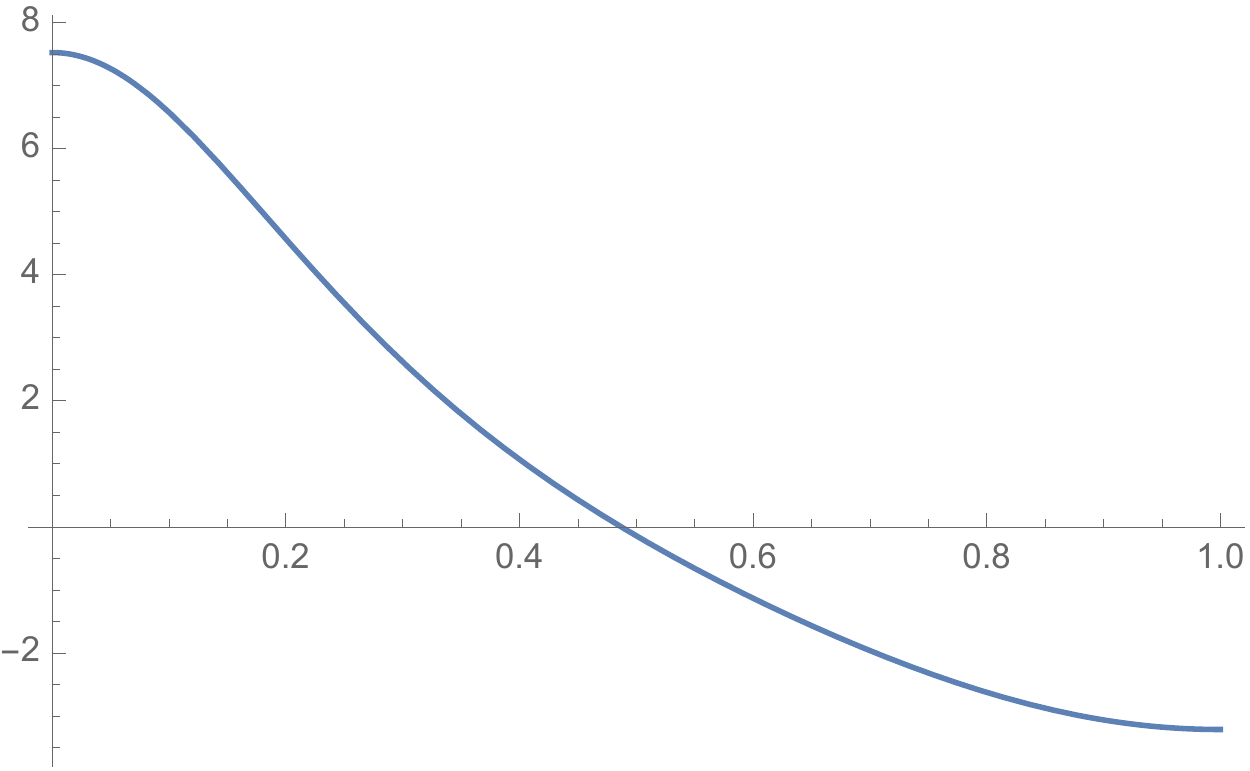}
\caption{Plot of the solution $u_{a_\star}$ of~\eqref{ODE}.}
\end{center}
\end{figure}
%---------------------------------------------------------------------

\vspace*{-0.5cm}Numerically, we obtain that
\[
2\,\pi\int_0^1\(|u'_{a_\star}|^2+|u_{a_\star}|^2\)r\,dr=2\,\pi\int_0^1|u_{a_\star}|^4\,r\,dr=\frac1{\mathcal C}\(2\,\pi\int_0^1|u_{a_\star}|^4\,r\,dr\)^{1/2}
\]
which means
\[
\mathcal C=\(2\,\pi\int_0^1|u_{a_\star}|^4\,r\,dr=\)^{-1/2}\approx0.0564922\,.
\]

%%%%%%%%%%%%%%%%%%%%%%%%%%%%%%%%%%%%%%%%%%%%%%%%%%%%%%%%%%%%%%%%%%%%%%
\subsection{Truncation functions}\label{Appendix:Truncation}
Here are some details on the truncation functions used in this document.
%---------------------------------------------------------------------
\begin{lemma}[Lemma 2.2 of~\cite{Bonforte2012a}]\label{lem.test.funct} Fix two balls $B_{R_1}\subset B_{R_0}\subset\subset\Omega$. Then there exists a test function $\varphi_{R_1, R_0}\in C_0^1(\Omega)$, with $\nabla \varphi_{R_1, R_0}\equiv0$ on $\partial\Omega$, which is radially symmetric and piecewise $C^2$ as a function of $r$, satisfies $\supp(\varphi_{R_1, R_0})=B_{R_0}$ and $\varphi_{R_1, R_0}=1$ on $B_{R_1}$, and moreover satisfies the bounds
\be{test.estimates}
\|\nabla\varphi_{R_1, R_0}\|_\infty\le\frac2{R_0-R_1}\quad\mbox{and}\quad \|\Delta\varphi_{R_1, R_0}\|_\infty\le\frac{4\,d}{(R_0-R_1)^2}.
\ee
\end{lemma}\textbf{
}
\begin{proof} With a standard abue of notation, we write indifferently that a radial function is a function of $x$ or of $|x|$. Let us consider the radial test function defined on $B_{R_0}$
\be{test.funct}
\varphi_{R_1, R_0}(|x|)=\left\{
\begin{array}{lll}
1\,&\quad\mbox{if }0\le |x|\le R_1\\[3mm]
1-\frac{2(|x|-R_1)^2}{(R_0-R_1)^2}\,&\quad\mbox{if }R_1<|x|\le\frac{R_0+R_1}2\\[3mm]
\frac{2(R_0-|x|)^2}{(R_0-R_1)^2}\,&\quad\mbox{if }\frac{R_0+R_1}2<|x|\le R_0\\[3mm]
0\,&\quad\mbox{if }|x|>R_0\\[3mm]
\end{array}
\right.
\ee
for any $0<R_1<R_0$. We have
\begin{equation*}
\nabla \varphi_{R_1, R_0}(|x|)=\left\{
\begin{array}{lll}
0\,&\mbox{if }0\le |x|\le R_1 \mbox{ or if }|x|>R_0\\[3mm]
-\frac{4(|x|-R_1)}{(R_0-R_1)^2}\frac{x}{|x|}\,&\quad\mbox{if }R_1<|x|\le\frac{R_0+R_1}2\\[3mm]
-\frac{4(R_0-|x|)}{(R_0-R_1)^2}\frac{x}{|x|}\,&\quad\mbox{if }\frac{R_0+R_1}2<|x|\le R_0\\[3mm]
\end{array}
\right.
\end{equation*}
and, recalling that $\Delta\varphi(|x|)=\varphi''(|x|)+(d-1)\varphi'(|x|)/|x|$, we have
\begin{equation*}
\Delta \varphi_{R_1, R_0}(|x|)=\left\{
\begin{array}{lll}
0\,&\quad\mbox{if }0\le |x|\le R_1 \mbox{ or if }|x|>R_0\\[3mm]
-\frac4{(R_0-R_1)^2}-\frac{d-1}{|x|}\frac{4(|x|-R_1)}{(R_0-R_1)^2}\,&\quad\mbox{if }R_1<|x|\le\frac{R_0+R_1}2\\[3mm]
-\frac4{(R_0-R_1)^2}-\frac{d-1}{|x|}\frac{4(R_0-|x|)}{(R_0-R_1)^2}\,&\quad\mbox{if }\frac{R_0+R_1}2<|x|\le R_0\\[3mm]
\end{array}
\right.
\end{equation*}
and easily obtain the bounds~\eqref{test.estimates}.\end{proof}

%%%%%%%%%%%%%%%%%%%%%%%%%%%%%%%%%%%%%%%%%%%%%%%%%%%%%%%%%%%%%%%%%%%%%%
%%%%%%%%%%%%%%%%%%%%%%%%%%%%%%%%%%%%%%%%%%%%%%%%%%%%%%%%%%%%%%%%%%%%%%
\medskip\noindent{\bf Acknowledgments:} This work has been partially supported by the Project EFI (ANR-17-CE40-0030) of the French National Research Agency (ANR). M.B. and N.S. were partially supported by Projects MTM2017-85757-P (Spain) and by the E.U. H2020 MSCA programme, grant agreement 777822. N.S. was partially funded by the FPI-grant BES-2015-072962, associated to the project MTM2014-52240-P (Spain).
\\\noindent\emph{\footnotesize\copyright\,2020 by the authors. This paper may be reproduced, in its entirety, for non-commercial purposes.}
%%%%%%%%%%%%%%%%%%%%%%%%%%%%%%%%%%%%%%%%%%%%%%%%%%%%%%%%%%%%%%%%%%%%%%
%%%%%%%%%%%%%%%%%%%%%%%%%%%%%%%%%%%%%%%%%%%%%%%%%%%%%%%%%%%%%%%%%%%%%%

\newpage\parskip=4pt\tableofcontents
%%%%%%%%%%%%%%%%%%%%%%%%%%%%%%%%%%%%%%%%%%%%%%%%%%%%%%%%%%%%%%%%%%%%%%
\end{document}